\pdfoutput=1
\documentclass[11pt]{article}
\usepackage{amssymb}
\usepackage{graphicx}
\usepackage{xcolor} 
\usepackage{tensor}
\usepackage{fullpage} 
\usepackage{amsmath}
\usepackage{amsthm}
\usepackage{verbatim}
\usepackage{hyperref}
\usepackage{enumitem}
\usepackage{mathrsfs} 
\setlist[enumerate]{leftmargin=1.5em}
\setlist[itemize]{leftmargin=1.5em}

\providecommand{\MR}{\relax\ifhmode\unskip\space\fi MR }

\providecommand{\href}[2]{#2}

\usepackage{seqsplit} 
\definecolor{green}{rgb}{0,0.8,0} 

\newcommand{\Black}[1]{\begingroup\color{black} #1\endgroup}

\newtheorem{maintheorem}{Theorem}

\newtheorem{theorem}{Theorem}[section]

\newtheorem{lemma}[theorem]{Lemma}
\newtheorem{proposition}[theorem]{Proposition}

\theoremstyle{definition}

\theoremstyle{remark}
\newtheorem{remark}[theorem]{Remark}

\numberwithin{equation}{section}

\newcommand{\nnrm}[1]{{\vert\kern-0.25ex\vert\kern-0.25ex\vert #1 
    \vert\kern-0.25ex\vert\kern-0.25ex\vert}}

\newcommand{\ud}{\mathrm{d}}
\newcommand{\rd}{\partial}
\newcommand{\nb}{\nabla}

\newcommand{\tht}{\theta}

\newcommand{\bbC}{\mathbb C}

\newcommand{\bbN}{\mathbb N}

\newcommand{\bbR}{\mathbb R}

\newcommand{\bbT}{\mathbb T}

\newcommand{\bbZ}{\mathbb Z}

\begin{document}

\title{Quantitative asymptotic stability of the quasi-linearly stratified densities in the IPM equation with the sharp decay rates}
\author{Min Jun Jo\thanks{Mathematics Department, Duke University, Durham NC 27708 USA. E-mail: minjun.jo@duke.edu} \and Junha Kim\thanks{Department of Mathematics, Ajou University, Suwon 16499, Republic of Korea. E-mail: junha02@ajou.ac.kr}}
\date{\today}

\renewcommand{\thefootnote}{\fnsymbol{footnote}}

\renewcommand{\thefootnote}{\arabic{footnote}}


\maketitle


\begin{abstract}
We analyze the asymptotic stability of the quasi-linearly stratified densities in the 2D inviscid incompressible porous medium equation on $\bbR^2$ with respect to the buoyancy frequency $N$.  Our target density of stratification is the sum of the large background linear profile with its slope $N$ and the small perturbation that could be both non-linear and non-monotone. Quantification in $N$ will be performed not only on how large the initial density disturbance is allowed to be but also on how much the target densities can deviate from the purely linear density stratification without losing their stability.

 For the purely linear density stratification, our method robustly applies to the three fundamental domains $\bbR^2,$ $\bbT^2,$ and $\bbT\times[-1,1]$, improving both the previous result by Elgindi (On the asymptotic stability of stationary solutions of the inviscid incompressible porous medium equation, Archive for Rational Mechanics and Analysis, 225(2), 573-599, 2017) on $\bbR^2$ and $\bbT^2$, and the  study by Castro-Córdoba-Lear (Global existence of quasi-stratified solutions for the confined IPM equation. Archive for Rational Mechanics and Analysis, 232(1), 437-471, 2019) on $\bbT\times[-1,1]$. The obtained temporal decay rates to the stratified density on $\bbR^2$ and to the newly found asymptotic density profiles on $\bbT^2$ and $\bbT\times[-1,1]$ are all sharp, fully realizing the level of the linearized system. We require the initial disturbance to be small in $H^m$ for any integer $m\geq 4$, which we even relax to any positive number $m>3$ via a suitable anisotropic commutator estimate.

\end{abstract}

\tableofcontents

\section{Introduction}

\subsection{Hydrodynamic stability}\label{sec_stab}

Fluid instabilities are ubiquitous and they are observed on exhaustively various scales of space and time. Due to the absence of any complete theory for such turbulent nature of fluid flows, simpler problems have been investigated as starting points. For instance, it is hoped that a deeper understanding of transition from a stable state, which is possibly laminar, to an unstable state would give a more rigorous explanation for the mechanism of spontaneous instability. \emph{Hydrodynamic stability}, which is known as one of the oldest branches of fluid mechanics, regards the robustness of fluid flows against small perturbation around a steady state, serving as a measure of intrinsic affinity of the governing fluid dynamics for turbulence.

One fundamental notion of hydrodynamic stability is \emph{Lyapunov stability.} Consider, as a toy model, the typical initial value problem
\begin{equation*}
\begin{split}
    \partial_t w &= f(w) \\
    w(0)&=w_0
    \end{split}
\end{equation*}
equipped with a stationary solution $w_s$ satisfying $f(w_s)=0$. The abstract mapping $f$ could be nonlinear. We say that the \Black{stationary} state $w_s$ is Lyapunov stable (or nonlinearly stable, equivalently) if the following holds: given \Black{a Banach spaces $X$}, for every $\varepsilon>0$ there exists $\delta>0$ such that
\begin{equation}\label{def_Lyapunov}
    \|w_0-w_s\|_{X}<\delta \implies \|w-w_s\|_{\widetilde{X}}<\varepsilon,
\end{equation}
where $w$ is the corresponding solution. The space $\widetilde{X}$ usually takes the form of $L_{t}^\infty X$ due to the evolutionary nature of the problem. Lyapunov stability at least tells us that the solution will stay close to the \Black{stationary} state once it is initially placed sufficiently close to the \Black{stationary} state. 

However, such an approximate notion of Lyapunov stability does not fully explain the dynamics of $w$ as $t\to \infty.$ This gives rise to the stronger notion called \emph{asymptotic stability.} The stationary solution $w_s$ is asymptotically stable if not only $w_s$ is Lyapunov stable but also there holds
\begin{equation}\label{def_asymp}
    \lim_{t\to\infty}\|w(t)-w_s\|_{Y}=0
\end{equation}
for a spatially normed space $Y$. This gives us the information on the explicit final destination of the solution, excluding the possibility of permanent orbiting/wandering around the \Black{stationary} state.

 The necessity of an even stronger notion of stability, compared to the mere asymptotic stability, stood out in Kelvin's suggestion \cite{Kelvin} that the governing dynamics could become increasingly sensitive to perturbation as Reynolds number gets larger, while certain equilibria are still Lyapunov stable for any finite Reynolds number. Quantification of $\delta$ in the notion of Lyapunov stability, see \eqref{def_Lyapunov}, is desired to capture the information on the sensitivity of the fluid system against perturbation in terms of Reynolds number. A \emph{quantitatively aysmptotic stability} result takes the form of the following statement. Let $N$ be any physical constant (e.g. fluid viscosity) one wants to quantitatively relate to the asymptotic stability. Then a stationary solution $w_s$ is quantitatively asymptotically stable if for all $\varepsilon>0$ there exists $\delta >0$ such that
 \begin{equation}\label{def_quant}
     \|w_0-w_s\|_{X} \leq \delta g(N) \implies
     \begin{cases}
     \|(w-w_s)(t)\|_{Y} < \varepsilon \\
     \lim\limits_{t\to\infty}\|(w-w_s)(t)\|_{Y}=0
     \end{cases}
 \end{equation}
where  $g:\bbR_{+}\to \bbR_{+}$ is a continuous function. For the results on quantitative asymptotic stability with respect to Reynolds number, see \cite{Nach8, BGM,BGM2,Nach14,Nach15,Nach80,Nach82,Nach87,Nach89}. \Black{For the Boussinesq equations, see \cite{BBZD,BZD}.}

This paper concerns the notion of quantitative asymptotic stability \eqref{def_quant} for the two-dimensional inviscid  incompressible porous medium equation \eqref{eq:IPM}. Quantification will be performed in terms of the large-scale physical phenomenon called \emph{stratification} which works via buoyancy. The motivation is that stratification tends to stabilize fluid flows by quenching the vertical motions and such a tendency of stabilization is observed to increase as the intensity of stratification grows, see \cite{QG} and references therein. 
 
\subsection{The inviscid IPM equation}

 We study the 2D inviscid incompressible porous medium (IPM) equation
\begin{equation}\label{eq:IPM} \tag{IPM}
\left\{
\begin{aligned} 
\rd_t\rho + v\cdot \nb \rho = 0, \\
v= -\nabla p + (0,\rho), \\
\mathrm{div}\, v = 0,
\end{aligned}
\right.
\end{equation}
on the three different domains $\bbR^2$, $\bbT^2$, and a confined strip $\bbT\times[-1,1]$.
The first line in \eqref{eq:IPM} is the continuity equation expressing the density transport along the fluid flow, and the second part is Darcy's law introduced in \eqref{Darcy}. The condition $\operatorname{div}v=0$ is the incompressibility condition. For the confined scenario $\bbT\times [-1,1]$, we supplement \eqref{eq:IPM} with the no-penetration boundary condition $v|_{y=-1,1}\cdot n=0$ where $n$ denotes the outward normal unit vector on the boundary.

 The above IPM equation falls into the category of \emph{active scalar} equations, as the inviscid Surface Quasi-Geostrophic (SQG) equation does. The SQG equation reads
 \begin{equation}
     \label{eq:SQG} \tag{SQG}
\left\{
\begin{aligned} 
\rd_t\theta + v\cdot \nb \theta = 0, \\
v = R^{\perp} \theta.
\end{aligned}
\right.
 \end{equation}
 where $R=(R_1,R_2)$ is the vector of the 2D Riesz transforms with Fourier symbol $\mathscr{F}R(\xi)=-i\frac{\xi}{|\xi|}.$ The SQG equation is the ensuing equation on the boundary of $\bbR_{+}^{3}$ for the 3D Quasi-Geostrophic (QG) equations, and it has been introduced in \cite{SQG} as the 2D model which has various features in common with the 3D Euler equations. There have been numerous results on the SQG equation during past several decades. For the global well-posedness type results when the dissipation term (e.g. Ekman pumping) is present, see \cite{CCW} and extensively various references therein. In the opposite direction, a small scale creation result can be found in \cite{HK} and some ill-posedness results are obtained in \cite{CM,EM,JK}, for examples.
 
 The Biot-Savart laws of \eqref{eq:IPM} and \eqref{eq:SQG}, via which the active scalars determine the corresponding velocities, are both singular operators of zero order; Darcy's law in \eqref{eq:IPM} can be rewritten as $v=R_1 R^{\perp} \rho$ while $v=R^\perp \theta$ in \eqref{eq:SQG}. Due to the similarities between \eqref{eq:IPM} and \eqref{eq:SQG}, including the same order of Biot-Savart laws mentioned above, it appears to be worth trying to establish analogous results for \eqref{eq:IPM}.
 
 However, it turns out that the structure of \eqref{eq:IPM} is distinguishable; Kiselev and Yao in \cite{KY} designed a new approach for \eqref{eq:IPM} to prove small scale creation unlike the previous hyperbolic scenarios used in \cite{HK,KS,Zlatos} for the generalized SQG (gSQG) equation equipped with the extended Biot-Savart law $v=R^{\perp}(-\Delta)^{-\frac{1-\alpha}{2}}\theta $ . Here $\alpha=0$ and $\alpha=1$, respectively,  correspond to the 2D Euler equations in vorticity formulation and the SQG equation. One notable difference is the additional partial derivative $\partial_1$ in the Biot-Savart law $v=R_1 R^{\perp} \rho$ for \eqref{eq:IPM}, compared to $v=R^{\perp}\theta$ for \eqref{eq:SQG}. This enables the IPM equation to feature \emph{density stratification} as the stabilizing mechanism in the presence of gravity. See Section~\ref{sec_Darcy} for the introduction of Darcy's law, which gives the relation $v=R_1 R^{\perp}\rho$, and  Section~\ref{sec_strat} for the background of density stratification.

 \subsubsection{The goal of this study}
 The goal of this study is to quantitatively fathom the stabilizing effect of stratification on the fluid flows moving through a porous medium, governed by \eqref{eq:IPM}. Specifically, the perturbation around the stratified densities gives rise to the extra linear structure so that the resulting linearized system becomes partially dissipative. Here we mean by partial dissipation that only certain directions of motion are damped in the system. 
 Leveraging properly the anisotropic damping of the linearized system, we prove that the IPM solution that is initially disturbed around a steady state of certain stratified density converges back to the equilibrium. Due to the lack of dissipation mechanism in \eqref{eq:IPM}, we call such damping phenomenon \emph{inviscid damping.}

\subsection{Darcy's law}\label{sec_Darcy}
In the 19th century, Henry Darcy designed a pipe system that provides water in Dijon, the city in France, and that works without using any pumps. The flows in the pipe are driven by gravity only, see \cite{Darcy}, as targeted in his past experiments. Specifically, Darcy found the linear relationship between the unit flux $q$ and the hydraulic gradient $\frac{\ud h}{\ud x}$ for the fluids flowing through a porous medium. The relationship reads
\begin{equation*}
    q=-k\frac{g\rho}{\eta}\frac{\ud h}{\ud x},
\end{equation*}
which is called \emph{Darcy's law.} In the above formulation, $k$ is the intrinsic permeability of the porous medium, $g\rho$ denotes the weight of the fluid with gravity acceleration $g$ and density $\rho$, and $\eta$ is the dynamic viscosity. The function $h$ denotes the height from the lower plate to the upper plate of the water pipe with corresponding constant water fluxes. Then the pressure due to the weight of the water in the pipe can be expressed as $p=\rho g h$. This gives rise to a different form of Darcy's law $$q= -\frac{k}{\eta}\frac{\ud p}{\ud x}.$$ Note that the derivatives, $\frac{\ud h}{\ud x}$ and $\frac{\ud p}{\ud x}$, were taken only in a horizontal direction along which the pressure drop happens, due to the confinement of the original setting. Nevertheless, such a viewpoint on the pressure drop can provide the extended Darcy's law
\begin{equation*}
    q= -\frac{k}{\eta}\nabla p,
\end{equation*}
which applies to all the directions. If we assume for simplicity that the porosity of the medium is equal to one, then the fluid velocity $v$ is exactly equal to $q$. In that case, the above version of Darcy's law reduces to the relationship for the fluid velocity as
\begin{equation}\label{Darcy}
    v= -\frac{k}{\eta} \nb p.
\end{equation}
For brevity, we set $k=\eta=1$ and so $v=-\nb p$.

When one tries to describe the motion of the large-scale fluid flows more precisely, it is desirable to take further into account the direct effect of gravity on fluid particles, other than the mere gravitaional effect via the hydraulic gradient. Then Darcy's law can be rewritten as
\begin{equation}
    v = -\nb p - (0,g\rho)
\end{equation}
which is the version we will use throughout this paper.

Experimentally, it has been observed that a fluid flow is \emph{Darcian} only when the Reynolds number is sufficiently small, say smaller than $10$. In principle, the Darcian flows can be regarded as the slow, laminar, and stationary flows. If we apply such tendency to the Navier-Stokes equations, we obtain, under the Boussinesq approximation, the Stokes equations
\begin{equation*}
    -\Delta v = -\nb p - (0,g\rho).
\end{equation*}
By using a formal averaging method, Neuman (1977) in \cite{Deriv_1} and Whitaker (1986) in \cite{Deriv_2} derived from the classical Stokes equations that in a porous medium the resistive force of viscosity would be \emph{linear} in the velocity with its proper sign. In other words, the Stokes equations become
\begin{equation*}
    v= -\nb p + (0,\rho)
\end{equation*}
which is Darcy's law revisited, taking $g=-1$ for brevity.

\subsection{Density stratification as stablizing mechanism}\label{sec_strat}

A noteworthy physical phenomenon that frequently occurs in large-scale fluid
flows is \emph{density stratification}. When the fluid density varies only in the vertical direction and every
upper fluid has a smaller density than the fluids placed in the lower layers, we say that the fluid
is (stably-)\emph{stratified}. Sufficiently stratified fluid tends to be stable because the
higher density at the lower layer causes gravity to dominate buoyancy so that the structure of
the stratified fluid remains stable against small perturbation. The density stratification can be seen as the
opposite of the Rayleigh-B\'enard (or Rayleigh-Darcy) type instability. Stratification tends to stabilize the fluid flows by
suppressing the vertical component of the fluid velocity to vanish and by layering the fluid motion
in horizontal planes \cite{Gill, Pedlosky}.  Mathematically, we say that a fluid is (stably-)stratified if $\rho'(x_2)>0$ (since we are taking $g=-1$ for brevity in this paper, the direction of gravity is upward, see \eqref{Darcy}) where $\rho=\rho(x_2)$ is the density function depending only on the vertical coordinate.

For the IPM equation \eqref{eq:IPM}, density stratification can stabilize the flow via the large-scale gravitational effect, which is encoded in the term $(0,\rho)$ in \eqref{Darcy}. This can be rigorously justified in showing asymptotic stability of certain stratified densities. As we will see later in Section~\ref{sec_stationary} under reasonably mild conditions, any stationary solution $\rho$ to \eqref{eq:IPM} on a bounded domain takes the form $\rho=\rho(x_2)$, i.e., the density is stratified, either stably or unstably, if we stretch out the notion of stratification to any vertically layered density $\rho=\rho(x_2)$ without requirement on the sign of $\rho'(x_2)$. This suggests that density stratification might be the generic mechanism of stabilization for \eqref{eq:IPM}. In Section~\ref{sec_stationary}, such a statement will be made more precisely.

\subsection{Stationary solutions to the IPM equation}\label{sec_stationary}

A remarkable feature of the IPM equation is that it has a relatively simple structure of stationary solutions, compared to other active scalar equations including the SQG equation and the 2D Euler equations. In fact, if $\rho$ is a $C^1$ stationary solution to \eqref{eq:IPM} on a bounded domain, then $\rho$ is forced to be a function of $x_2$ only while $v\equiv 0$.
See Lemma 1.1 in \cite{Elgindi}. However, the argument of specifying steady states in \cite{Elgindi} relies on the a priori square integrability of $v$, which can be guaranteed on bounded domains only. On $\bbR^2$, we take a slightly different perspective, which would be presented here shortly, to justify the steady states of the form $(\rho,v)=(\rho(x_2),0)$ as the targets of our analysis.

The previously mentioned observation in \cite{Elgindi} on the steady states motivates us to 
consider the equilibrium $(\rho_s,v_s,p_s)$ with the stratified density $\rho_s = \rho_s(x_2)$. Then, we have from the $\rho$ equation in \eqref{eq:IPM} that $v_2 = 0$. Since the divergence-free condition gives $v_1(x_1) = C$ for some $C \in \bbR$, we obtain $v_s = (C,0)$ and $\nabla p_s = (-C,\rho(x_2))$. By setting up the differences from the steady state as $$\theta(t,x) = \rho(t,x) - \rho_{s}(x_2), \qquad u(t,x) = v(t,x) - v_s, \qquad P(t,x) = p(t,x) - p_s(x),$$ we obtain the following perturbed system
\begin{equation*}  \left\{
\begin{aligned} 
\rd_t\tht + (u + v_s) \cdot \nb \tht = -u_2 \rho'_{s}, \\
u= -\nabla P + (0,\tht), \\
\mathrm{div}\, u = 0.
\end{aligned}
\right.
\end{equation*}
Employing the transformation $$\tilde{\tht}(t,x) = \tht(t,x+v_st), \qquad \tilde{u}(t,x) = u(t,x+v_st), \qquad \tilde{P}(t,x) = P(t,x+v_st),$$ we can reduce the above formulation to the case of $v_s = 0$. This allows us, without loss of generality, to work only on the form
\begin{equation}\label{st_sol}
    (\rho_{s},u_{s},p_{s}) = (\rho_{s}(x_2),0,\int_{0}^{x_2} \rho_{s}(\tau)\,\ud \tau)
\end{equation} which yields the corresponding reduced system
\begin{equation}  \label{eq:SIPM_g} \tag{SIPM}
\left\{
\begin{aligned} 
\rd_t\tht + u\cdot \nb \tht = -u_2 \rho'_{s}, \\
u= -\nabla P + (0,\tht), \\
\mathrm{div}\, u = 0.
\end{aligned}
\right.
\end{equation}
The form \eqref{st_sol} is consistent with Lemma 1.1 in \cite{Elgindi}; any stationary solution to \eqref{eq:IPM} with the stratified density $\rho_s=\rho_s(x_2)$ should take the form of \eqref{st_sol}. 


\subsection{Previous results}
In this section, we gather the previous results regarding the stability for the IPM equation. To the extent of the authors' knowledge, the existing works focus on the \emph{linearly} stratified densities only, say $\rho_s(x_2)=x_2$. In this case, \eqref{eq:SIPM_g} corresponds to the system \eqref{eq:SIPM} with $\sigma =0$, so here the previous results are stated in terms of \eqref{eq:SIPM} with $\sigma=0$. The first breakthrough was made in \cite{Elgindi} for $\bbR^2$ and $\bbT^2$, which is stated in the following theorem. Note that in the case of the torus $\bbT^2$, we are still looking near $\rho_s(x)=x_2$ on $\bbR^2$ while the \emph{perturbation} $\theta$ is defined on $\bbT^2$.
\begin{theorem}[T. Elgindi, 2017]\label{thm_Elgindi} Let $\rho_s(x_2)=x_2.$ There exists $\delta_0$ such that if $\|\theta_0\|_{H^{m}\cap W^{4,1}(
\bbR^2)} \leq \delta \leq \delta_0$ with $m\geq 20$, then the solution $\theta$ to \eqref{eq:SIPM_g} on $\bbR^2$ satisfies
\begin{equation*}
    \|\theta(t)\|_{H^3} \lesssim \frac{\delta}{t^{1/4}},\quad \|u_1\|_{H^3}\lesssim \frac{\delta}{t^{3/4}},\quad \mbox{and}\quad  \|u_2\|_{H^3} \lesssim \frac{\delta}{t^{5/4}} \quad \forall t>0,
\end{equation*}
where $u=R_1 R^{\perp} \theta$. For the case of $\bbT^2$, there exists $\delta_1>0$ such that if $\|\theta_0\|_{H^m(\bbT^2)} \leq \delta \leq \delta_1$ with $m\geq 20$, then the solution $\theta$ to \eqref{eq:SIPM_g} on $\bbT^2$ satisfies
\begin{equation*}
    \|\theta(t)\|_{H^{20}} \leq 2 \delta, \quad\mbox{and}\quad \|u\|_{H^3} \lesssim \frac{\delta}{t^{5/2}} \quad \forall t>0.
\end{equation*}
\end{theorem}
The confined strip $\Omega:=\bbT\times[-1,1]$ was dealt with in \cite{Cordoba} as a possibly more appropriate scenario in view of anisotropy stemming from gravity. The result is as follows. 
\begin{theorem}[A. Castro, D. C\'ordoba, and D. Lear, 2019]\label{thm_Cordoba} There exists $\varepsilon_0>0$ such that if $\|\theta_0\|_{X^m(\Omega)}\leq 
\varepsilon \leq  \varepsilon_0$ with $m\geq 10$, then the solution $\theta$ to \eqref{eq:SIPM_g} on $\Omega$ exists globally in time and satisfies
\begin{equation*}
    \|u\|_{H^3(\Omega)} \lesssim \frac{\varepsilon_0}{(1+t)^{5/4}}, \quad \|\bar{\theta}\|_{H^3(\Omega)} \lesssim \frac{\varepsilon_0}{(1+t)^{5/4}}, \quad \mbox{and}\quad \|\theta\|_{H^{m}(\Omega)} \leq 2\varepsilon \quad \forall t>0,
\end{equation*}
where we use the notation $\bar{f}:=f-\int_{\bbT} f \,\mathrm{d}x_1$ for any $f\in L^2(\Omega)$. The function space $X^m(\Omega)$ is defined in \eqref{def_Xm}.
\end{theorem}

\Black{In the very recent work \cite{Paicu}, viewing \eqref{eq:SIPM_g} on $\bbR^2$ as the limit system of the 2D damped Boussinesq equations, the global well-posedness result was obtained via the anisotropic Littlewood-Paley theoretical analysis for the finer homogeneous spaces $ \dot{H}^{1-\tau}(\bbR^2) \cap \dot{H}^{3+\tau}(\bbR^2)$ with $0<\tau<1$ as the following.

\begin{theorem}[R. Bianchini, T. Crin-Barat, M. Paicu, 2024] Let $\rho_s(x_2)=x_2$. Fix any $0<\tau<1$ and let $s\geq 3+\tau$. For any $\theta_0\in \dot{H}^{1-\tau}(\bbR^2) \cap \dot{H}^{3+\tau}(\bbR^2)$, there exists $\delta_0>0$ such that if $\|\theta_0\|_{\dot{H}^{1-\tau} \cap \dot{H}^{3+\tau}(\bbR^2)} \leq \delta_0$, then there exists a unique global solution $\theta$ to \eqref{eq:SIPM_g} on $\bbR^2$ equipped with the time integrability of $\|u_2(t)\|_{L^\infty}$.
    
\end{theorem}}

\subsection{Summary of main results}

We summarize the main contributions of our results that will be presented in Section~\ref{sec_main}.
\begin{itemize}
    \item \textbf{Quantification of asymptotic stability in $N$}. 
    We quantify the asymptotic stability for the IPM equation in terms of the stratification on all the target domains $\bbR^2$, $\bbT^2$, and $\bbT\times[-1,1]$. This work can be seen as the first work on quantitative asymptotic stability for the IPM equation with respect to the buoyancy frequency $N$.
    
    \item \textbf{Sharpness of temporal decay rates and exact asymptotic profiles.} Every decay rate obtained in Theorem~\ref{thm:stb_R}-\ref{thm:stb_O} sharply reaches the rate of the linearized system. Note that we do not require any $L^1$ condition for initial data, compared to Theorem~\ref{thm_Elgindi} where the polynomially $1/4$-order faster decay rate was obtained on $\bbR^2$ thanks to the $W^{4,1}$ assumption on the initial disturbance. For the other domains, $\bbT^2$ and $\bbT\times[-1,1]$, while there were no available temporal decay rates in the preceding works \cite{Cordoba,Elgindi}, we newly specify the exact asymptotic density profiles and provide the corresponding temporal decay rates that are \emph{sharp}. Also, the rates get faster as we impose more regularity on the initial disturbance, inspired by \cite{Cordoba2}.
    
    \item \textbf{Robustly reduced Sobolev exponents.} We give a reduced minimal Sobolev exponent that works \emph{robustly} on all the target domains. This is roughly due to the \emph{derivative loss minimizing nature} of our Fourier analytic energy method; see Lemma~\ref{key_lem_R} as an example. Specifically, we require the initial density disturbance $\theta_0$ to be small in $H^m$ for any integer $m\geq 4$, see Theorem~\ref{thm:stb_R}-\ref{thm:stb_O}. Note that on the confined scenario, $X^m$ (see \eqref{def_Xm}) was considered instead of $H^m$ but still $m$ denotes the $m$-times differentiability. This improves the existing works \cite{Elgindi} on $\bbT^2$ and \cite{Cordoba} on $\bbT\times[-1,1]$, where $m\geq 20$ and $m\geq 10$ were required, respectively. We can further relax the condition $m\geq 4$ to any positive number $m>3$ on $\bbR^2$ and $\bbT^2$ for purely linear stratification by proving a specific commutator estimate. \Black{This roughly corresponds to the homogeneous variants $\dot{H}^{1-\tau}(\bbR^2)\cap \dot{H}^{3+\tau}(\bbR^2)$ with $0<\tau<1$ for which the global well-posedness and some time integrability are obtained in \cite{Paicu}. While such $\dot{H}^{1-\tau}\cap \dot{H}^{3+\tau}$ type initial data with $0<\tau<1$ is slightly more general than the $H^m$ data with $m>3$, the obtained $H^m$ solution is more regular than the $\dot{H}^{1-\tau}\cap \dot{H}^{3+\tau}$ solution so that we establish the stronger properties. For example, the time integrability of $\|u_2(t)\|_{L^\infty}$ is achieved in \cite{Paicu} with the homogeneous norms while we prove the explicit temporal decay $\|u_2(t)\|_{L^\infty}\lesssim (1+t)^{-5/4}$ for the $H^m$ solutions. See Remark~\ref{rmk_decay}.}
 
    \item \textbf{Non-linearly stratified densities.} On $\bbR^2$, we show the quantitative asymptotic stability of certain quasi-linearly stratified densities. This paper seems to provide the very first touch on the IPM stability of densities that are \emph{non-linear} in the vertical direction. Quantification is also performed on \emph{how much the target densities can deviate from the purely linear density-stratification without losing their stability,} with respect to the intensity of background stratification. Such a perturbative nature of our quasi-linear densities can be similarly found in showing the stability of shear flows close to the planar Couette flow in the 2D Euler equations \cite{BM}.

    \item \textbf{Non-monotone deviation.} We would like to emphasize here another aspect in our treatment of the non-linearly stratified densities. Our target densities of stratification can be decomposed into large linear background profile and small \emph{non-monotone} deviation. Non-monotonicity of density is known to largely affect the formations of fluid instabilities \cite{instability_1,instability_2,instability_3}; to get a better grasp of the stabilizing effect of stratification on fluids, we try to quantify the admissible non-monotonic deviation from the relatively well-investigated linear density profile \cite{Cordoba,Elgindi} in the view of buoyancy-driven stability. The precise statement can be found in Theorem~\ref{thm:stb_R}.


    \item \textbf{Anisotropic commutator estimate.} The standard commutator estimates deal with either the full partial derivatives $\partial^{\alpha}$ or the isotropic Fourier multipliers including $\Lambda^{\alpha}$ and $J^{\alpha}$  that are defined by $|\xi|^{\alpha}$ and $(1+|\xi|^2)^{\alpha/2}$ in Fourier variables, respectively. Here we give an elementary proof of the anisotropic commutator estimate that transfers specific derivatives $\partial_1$ or $\partial_2$ in a desirable way during the estimate. This allows us to use the energy estimates in the low regularity setting, compared to the energy estimates in the former studies \cite{Elgindi} and \cite{Cordoba}. See Lemma~\ref{lem_commu}. One may refer to \cite{Paicu} for a different type of anisotropic commutator estimate.
\end{itemize}

\section{Main results}\label{sec_main}

\subsection{Notations}
We denote by $\mathscr{S}(X)$ the space of the Schwartz functions that are defined on the space $X$. For any $f\in \mathscr{S}(\bbR^2)$, we use $\mathscr{F}f$ for the Fourier transform of $f$. For the one-dimensional function $g\in \mathscr{S}(\bbR)$, we write $\widehat{g}$ instead. The function spaces $W^{k,p}$ and $H^s$ for positive numbers $s,$ $k,$ and $p$ stand for the usual Sobolev spaces. 

\subsection{Statements}

We consider the domain $\bbR^2$ first. For any positive constant $N$ and smooth function $\sigma \in \mathscr{S}(\bbR)$, we consider a special class of stratified solutions
\begin{equation}\label{str_class}
    \mathscr{C}_{N,\sigma} := \left\{ \rho_{s} \in C^{\infty}(\bbR) ; \,\,\rho_{s} = \int_{0}^{x_2} N + \sigma(\eta) \,\ud \eta \right\}.
\end{equation}
Then, we can rewrite \eqref{eq:SIPM_g} as follows:
\begin{equation}  \label{eq:SIPM} 
\left\{
\begin{aligned} 
\rd_t\tht + u\cdot \nb \tht = -u_2 (N + \sigma(x_2)), \\
u= -\nabla P + (0,\tht), \\
\mathrm{div}\, u = 0.
\end{aligned}
\right.
\end{equation}
Our first main theorem quantifies the asymptotic stability of the stationary solutions $\rho_s \in \mathscr{C}_{N,\sigma}$ in the IPM equation, provided that $N$ is sufficiently large in some sense. 
\begin{maintheorem}[Quantitative asymptotic stability of quasi-linear densities]\label{thm:stb_R}
	Let $m \in \bbN$ with $m > 3$. Then, there exists a constant $C_0=C_0(m)>0$ such that for any triple $(N,\sigma,\theta_{0}) \in (0,\infty)\times \mathscr{S}(\bbR) \times H^{m}(\bbR^2)$ satisfying \begin{equation}\label{ass}
	N \geq C_0 \left( \| \tht_0 \|_{H^{m}} + \| (1+|\xi_2|^2)^{\frac{m+1}{2}} \widehat{\sigma}(\xi_2) \|_{L^1} \right),
	\end{equation} \eqref{eq:SIPM} possesses a unique global-in-time solution $$\theta \in C([0,\infty);H^{m}(\bbR^2)) \cap C^1((0,\infty);H^{m-1}(\bbR^2))$$
	which fulfills the sharp temporal decay rates for the velocity $u$ as
	\begin{equation}\label{tem_decay_est_sgm_2}
	     \| u(t) \|_{H^{m}} \leq \frac{C\|\theta_0\|_{H^m}}{(1+Nt)^{\frac{1}{2}}},\quad \| \nabla u_2(t) \|_{H^{m-1}} \leq \frac{C \| \tht_0 \|_{H^m}}{(1+Nt)}
	\end{equation}
for some $C=C(m)>0$. Moreover, the density $\theta$ itself features the $L^\infty$ decay in time
	\begin{equation}\label{tem_decay_est_sgm}
		\| \theta(t) \|_{L^{\infty}} \to 0 \quad \mbox{as} \quad Nt \to \infty.
	\end{equation}
\end{maintheorem}

\begin{remark}
    \Black{The class of stratified solutions $\mathscr{C}_{N,\sigma}$ defined by \eqref{str_class} is not optimal and can be enlarged to a more general one. In the proof of Theorem~\ref{thm:stb_R}, any kind of decay property of $\sigma$ is not required, only the finiteness of $\| (1+|\xi_2|^2)^{\frac{m+1}{2}} \widehat{\sigma}(\xi_2) \|_{L^1}$ with \eqref{ass} is used.}
\end{remark}

\begin{remark}
    The temporal decay rates in \eqref{tem_decay_est_sgm_2} and \eqref{tem_decay_est_R} are sharp, i.e., for any given $\varepsilon>0$, we can find a triple $(N,\sigma,\tht_0) \in (0,\infty)\times \mathscr{S}(\bbR) \times H^{m}(\bbR^2)$ such that the solution to the linearized system of \eqref{eq:SIPM} satisfies $$\| u(t) \|_{H^{m}} \geq \frac{C}{(1+Nt)^{\frac{1}{2}+\varepsilon}},\quad \| \nabla u_2(t) \|_{H^{m-1}} \geq \frac{C}{(1+Nt)^{1+\varepsilon}}$$ for some $C>0$. For details, we refer to Proposition~\ref{prop_sharp}.
\end{remark}

\begin{remark}\label{rmk_decay}
    In particular, if $\sigma = 0$ is assumed, we can drop the condition $m \in \bbN$. Furthermore, $L^{\infty}$ type temporal decay estimates can be performed as
	\begin{equation}\label{tem_decay_est_R}
		(1+Nt)^{\frac{1}{4}} \| \theta(t) \|_{W^{1,\infty}} + (1+Nt)^{\frac{3}{4}} \| u_{1}(t) \|_{W^{1,\infty}} + (1+Nt)^{\frac{5}{4}} \| u_{2}(t) \|_{L^{\infty}} \leq C\| \tht_0 \|_{H^m}.
	\end{equation}
	Since the proof is standard, we only refer to \cite{KL} where the proof of temporal decay estimate for the Boussinesq equations is provided in a way that can be applied to \eqref{tem_decay_est_R}. For the sharpness of the decay rates, see Proposition~\ref{prop_sharp}. 
\end{remark}
\begin{remark}
    The solutions obtained in Theorem~\ref{thm_Elgindi} decay faster than those of Theorem~\ref{thm:stb_R} due to the extra $W^{4,1}$ condition on $\theta_0$, while both theorems provide sharp decay rates. We fully eliminate the $W^{4,1}$ condition by observing that temporally $L^\infty$ type quantities decay faster than $L^2$ type ones; compare \eqref{tem_decay_est_sgm_2}   to \eqref{tem_decay_est_R}. This is possible essentially because our approach (Lemma~\ref{key_lem_R}) does not leverage the  temporal decays a priori in proving global existence.
\end{remark}

\begin{remark}
The global solution $\theta$ obtained in the above theorem further satisfies	\begin{equation}\label{sol_bdd}
	    \sup_{t \in [0,\infty)} \| \tht(t) \|_{H^{m}}^2 + N\int_0^{\infty} \| R_1 \tht(t) \|_{H^m}^2 \,\ud t \leq 4 \| \tht_0 \|_{H^{m}}^2
	\end{equation}
	and \begin{equation}\label{sol_bdd1}
	    N\int_0^\infty \| u_2(t) \|_{W^{1,\infty}} \,\ud t \leq N\int_0^\infty \| (1+|\xi|^2)^{\frac{1}{2}} \mathscr{F} u_2(t) \|_{L^1} \,\ud t \leq C \| \tht_0 \|_{H^m}.
	\end{equation}
\end{remark}

Next, we consider the domains $\bbT^2$ and $\bbT \times [-1,1]$. The primary feature that distinguishes those domains from $\bbR^2$ is the boundedness of $x_2$. This \emph{blurs} the difference between the previously considered quasi-linear case and the exactly linear case. For this reason, we focus on the stationary state of the purely linear density profile $(\rho_{s}, u_{s}, p_{s}) = (Nx_2, 0, \frac{N}{2}x_2^2)$.

Let us consider the two-dimensional torus case first. Adopting the perturbative regime, we take $$(\rho,u,p) = (Nx_2 + \tht, u, \frac{N}{2}x_2^2 + P)$$ as a solution to \eqref{eq:IPM} in $\bbT^2$, where $\tht$, $u$, and $P$ are periodic in $x$. Then, the triple $(\tht,u,P)$ satisfies the system
\begin{equation}  \label{eq:SIPM_T} 
\left\{
\begin{aligned} 
\rd_t\tht + u\cdot \nb \tht = -N u_2, \\
u= -\nabla P + (0,\tht), \\
\mathrm{div}\, u = 0.
\end{aligned}
\right.
\end{equation}
Our second main theorem establishes that the steady state $\rho_s(x)=x_2$ is nonlinearly stable against the initial periodic perturbation $\theta_0$ defined on $\bbT^2$, and further that the solution converges asymptotically to a suitable profile defined in \eqref{def_sgm} with the sharp decay rates.
\begin{maintheorem}\label{thm:stb_T}
	Let $m > 3$. Then, there exists a constant $C_0=C_0(m)>0$ such that for any $N>0$ and $\theta_{0} \in H^{m}(\bbT^2)$ with $N > C_0 \| \theta_{0} \|_{H^{m}}$, \eqref{eq:SIPM_T} possesses a unique solution $$\theta \in C([0,\infty);H^{m}(\bbT^2)) \cap C^1([0,\infty);H^{m-1}(\bbT^2))$$ with \eqref{sol_bdd}. Moreover, there exist a constant $C>0$ and a function
	\begin{equation}\label{def_sgm}
        \sigma(x_2) := \int_{\bbT} \tht_0 \,\ud x_1 - \int_0^{\infty} \int_{\bbT} \left((u \cdot \nabla)\tht + Nu_2 \right)\,\ud x_1 \ud t
    \end{equation} such that
	\begin{equation}\label{tem_decay_est_T}
		(1+Nt)^{\frac{m-s}{2}} \| \theta(t) - \sigma(x_2) \|_{H^{s}} + (1+Nt)^{\frac{1}{2} + \frac{m-s}{2}}\| u(t) \|_{H^{s}} + (1+Nt)^{1+\frac{m-s}{2}}\| u_2(t) \|_{H^{s}} \leq C \| \tht_0 \|_{H^m}
	\end{equation}
	for all $s \in [0,m]$.
\end{maintheorem}
\begin{remark}
    The decay rates in \eqref{tem_decay_est_T} are sharp (we refer to \cite[Section 7]{JK2}).
\end{remark}
	    
\begin{remark}
    We prove Theorem~\ref{thm:stb_T} by computing directly the velocity field with the pressure term without requiring that $\int_{\bbT^2} \tht_{0}(x) \,\ud x=0$ anymore; In the torus case, the only available result \cite{Elgindi} employed the stream-function formulation to construct the velocity, necessitating $\int_{\bbT^2} \tht_{0}(x) \,\ud x=0$ as a consequence. See Lemma~\ref{lem_avg} and below.  
\end{remark}

 As the third scenario, we consider \eqref{eq:SIPM_T} on the horizontally periodic strip $\Omega = \bbT \times [-1,1]$ with no penetration boundary condition. In \cite{Cordoba}, it is shown that the boundary condition can be propagated to the higher derivatives. In that sense, the natural solution spaces in the Sobolev space $H^m(\Omega)$ are of the forms
\begin{equation}\label{def_Xm}
    X^m := \left\{ f \in H^m(\Omega) ; \partial_2^k f \big|_{\partial \Omega} = 0, \qquad k = 0,2,4, \cdots , m^* \right\},
\end{equation}
\begin{equation}\label{def_Ym}
    Y^m := \left\{ f \in H^m(\Omega) ; \partial_2^k f \big|_{\partial \Omega} = 0, \qquad k = 1,3,5, \cdots , m_* \right\},
\end{equation}
where $m^*$ and $m_*$ are the largest even and odd number with $m^* < m$, respectively. Due to the definitions of $X^m$ and $Y^m$, it seems appropriate to consider the cases $m\in \bbN$ only, while one may try to extend it to the general non-integer cases by setting up the proper analouges of $X^m$ and $Y^m$. We establish the following theorem.
\begin{maintheorem}\label{thm:stb_O}
	Let $m \in \bbN$ with $m > 3$. Then, there exists a constant $C_0=C_0(m)>0$ such that for any $N>0$ and $\theta_{0} \in X^{m}$ with $N > C_0 \| \theta_{0} \|_{H^{m}}$, \eqref{eq:SIPM_T} possesses a unique solution $$\theta \in C([0,\infty);X^{m}) \cap C^1([0,\infty);X^{m-1})$$ with \eqref{sol_bdd}. Moreover, there exist a constant $C>0$ and a function $\sigma(x_2)$ with \eqref{def_sgm} such that \eqref{tem_decay_est_T} holds for all $s \in [0,m]$.
\end{maintheorem}
Since the proof of Theorem~\ref{thm:stb_O} is analogous to the proof for the $\bbT^2$ case, we briefly prove it in Section~\ref{sec_O} with strong emphasis on the nontrivial parts only originating in the inherent differences between $X^m(\Omega)$ and $H^m(\bbT^2)$. 

\subsection{Strategy for proofs}
We prove the global well-posedness part of Theorem~\ref{thm:stb_R} by performing a continuity argument based on a specific type of $H^m$ energy estimate. Such an energy estimate, see \eqref{energy}, allows us to isolate the key term $\int_{0}^{T} \|u_2(t)\|_{W^{1,\infty}}\,\mathrm{d}t$ in which the main technical difficulty arises. To control the key term, we rely on its Duhamel formulation via Darcy's law. Then we can access the usual eigenvalue/eigenvector analysis for the corresponding anisotropic linear propagator in Fourier variables. The trick is, the key term is controlled by itself with the multiplying factor that can be made very small for sufficiently large $N>0$. The largeness of $N$ essentially represents the dominance of stratification over the entire stabilizing mechanism. By assuring that the key term gets small for large $N$, we close the energy inequality and apply a standard continuity argument to finish the proof.

The proof of the temporal decay part of Theorem~\ref{thm:stb_R} uses a similar method combined with the interpolation inequalities in Lemma~\ref{lem_intp} that are necessary to treat the anisotropy of the linear operator in a more sophisticated fashion. We obtain the temporal decay of $\|\theta(t)\|_{L^\infty}$ via the aforementioned delicate technique  with a contradiction argument, but the explicit decay rates were out of our method of Fourier analysis due to the presence of the small perturbation $\sigma\in\mathscr{S}$. The computation of the explicit temporal decay rates for the spatially $L^2$-type quantities becomes relatively straightforward once we employ the idea of \emph{balancing out} the multiples as in \eqref{ineq_balancing}, which was introduced in \cite{Elgindi}.

On the torus $\bbT^2$, it appears that the stabilizing mechanism is totally different from the one on $\bbR^2$. As noted in \cite{Elgindi}, the three crucial difficulties, which distinguish $\bbT^2$ from $\bbR^2$, are (1) the density disturbance $\theta(t)$ does not decay itself, (2) the linear propagator has infinitely many non-decaying modes, and (3) showing the decay of $u$ requires a loss of derivatives. To overcome those difficulties, we start by newly introducing a simple but powerful observation that \emph{the mean value of $\theta(t)$ on $\bbT^2$ indeed decays over time} in a stark contrast to $\theta(t)$ itself, see Lemma~\ref{lem_avg}. This broadens the applicability of our method to the more general densities having non-zero mean value, generalizing the result in \cite{Elgindi} where it is assumed that the mean value of $\theta(t)$ (or equivalently $\theta_0$) vanishes.
We further see that the horizontally averaged $u_2(t)$ over $\bbT$ also decays and, in particular, the horizontally averaged $u_1$ over $\bbT$ is zero. Those basic decays allow us, based on the Duhamel formula, to design the corresponding $H^m$ energy estimate and the control of the key term, analogously to the ones on $\bbR^2$. Finally, we heavily use the previously mentioned idea of balancing out the multiples to compute the explicit decay rates in Theorem~\ref{thm:stb_T} by effectively handling the anisotropic nature of the linear propagator.

The confined scenario $\Omega=\bbT\times[-1,1]$ is dealt with similarly to the case of $\bbT^2$. As suggested for $\Omega$ in \cite{Cordoba}, we set up the different functional setting adapted to the presence of boundary, which is based on $X^m$ and $Y^m$ defined in \eqref{def_Xm} and \eqref{def_Ym}. After ironing out the unique features coming from such a functional environment, which could be problematic mainly due to boundary, we can mimic the blueprint that was originally designed for $\bbT^2$, providing the completely analogous results. Compare the statement in Theorem~\ref{thm:stb_O} to the one in Theorem~\ref{thm:stb_T}.

\section{Preliminaries}\label{sec_prl}
Beginning this section, we provide convolution inequalities for $\bbR^2$ and $\bbT^2$ domains, which are useful to prove Lemma~\ref{key_lem_R} and \ref{key_est_T}. We refer to Lemma~\ref{lem_conv_O} for the strip domain.
\begin{lemma}\label{lem_intp}
    Let $s \geq 0$. For any $p,q,r \in [1,\infty]$ with $\frac 1q + \frac 1r = 1 + \frac 1p$, there exists a constant $C=C(s)>0$ such that \begin{equation}\label{conv_ineq_R2}
        \| (1+|\xi|^2)^{\frac{s}{2}} \mathscr{F} (fg) \|_{L^p} \leq C \| (1+|\xi|^2)^{\frac{s}{2}}  \mathscr{F} f \|_{L^q} \| (1+|\xi|^2)^{\frac{s}{2}}  \mathscr{F} g \|_{L^r}
    \end{equation}
    and
    \begin{equation}\label{conv_ineq_R2_2}
        \| (1+|\xi|^2)^{\frac{s}{2}}  \mathscr{F} (f(x)\sigma(x_2)) \|_{L^p} \leq C \| (1+|\xi|^2)^{\frac{s}{2}}  \mathscr{F} f \|_{L^q_{\xi_2}L^p_{\xi_1}} \| (1+|\xi_2|^2)^{\frac{s}{2}}  \widehat{\sigma} \|_{L^r},
    \end{equation}
    for any $f,g \in \mathscr{S}(\bbR^2)$ and $\sigma \in \mathscr{S}(\bbR)$.
\end{lemma}
\begin{proof}
    By the property $\mathscr{F}(fg) = \mathscr{F}f * \mathscr{F}g$ and $$(1+|\xi|^2)^{\frac{s}{2}} \leq C(1+|\xi-\eta|^2)^{\frac{s}{2}} + C(1+|\eta|^2)^{\frac{s}{2}} \leq C(1+|\xi-\eta|^2)^{\frac{s}{2}} (1+|\eta|^2)^{\frac{s}{2}}$$ for $\xi,\eta \in \bbR^2$, we have 
    \begin{align*}
    (1+|\xi|^2)^{\frac{s}{2}} \mathscr{F} (fg) &\leq \int_{\bbR^2} \int_{\bbR^2} (1+|\xi|^2)^{\frac{s}{2}} |\mathscr{F}f(\xi-\eta)| |\mathscr{F}g(\eta)| \,\ud \eta \ud \xi \\
    &\leq C\int_{\bbR^2} \int_{\bbR^2} (1+|\xi-\eta|^2)^{\frac{s}{2}} |\mathscr{F}f(\xi-\eta)| (1+|\eta|^2)^{\frac{s}{2}} |\mathscr{F}g(\eta)| \,\ud \eta \ud \xi.
    \end{align*}
    Applying Young's convolution inequality to this, we obtain \eqref{conv_ineq_R2}.
    
    We can prove \eqref{conv_ineq_R2_2} similarly, noting that $\sigma$ does not depend on $x_1$. We omit the details.
\end{proof}

\begin{lemma}
    Let $s \geq 0$. Then, there exists a constant $C=C(s)>0$ such that \begin{equation}\label{conv_ineq_T2}
        \| (1+|n|^2)^{\frac{s}{2}} \mathscr{F} (fg) \|_{l^1} \leq C \| (1+|n|^2)^{\frac{s}{2}}  \mathscr{F} f \|_{l^1} \| (1+|n|^2)^{\frac{s}{2}}  \mathscr{F} g \|_{l^1}
    \end{equation}
    for any $f,g \in \mathscr{S}(\bbT^2)$.
\end{lemma}
\begin{proof}
    One can modify the proof of \eqref{conv_ineq_R2} to deduce \eqref{conv_ineq_T2}.
\end{proof}

\subsection{Existence of local-in-time solutions}

Here, we consider the local-in-time solutions to \eqref{eq:SIPM_g} given certain smooth function $\rho_s : \bbR \to \bbR$, including the case $\rho_s \in \mathscr{C}_{N,\sigma}$. In $\bbR^2$, whether the solution exists depends on the choice of the function $\rho_s$ with respect to $\|\rho'_s\|_{W^{m,\infty}}$ which distinguishes the purely linear case $\rho_s=Nx_2$ and the quasi-linear case $\rho_s=Nx_2+\sigma$. Such dependence can be seen in the following proposition we provide. It is concerned with $\rho_s$ satisfying $\|\rho_s'\|_{L^{\infty}} < \infty$, which does not cover the case of super-linear function, but which covers $\rho_s \in \mathscr{C}_{N,\sigma}$ for any $N \in \bbR$ and any $\sigma \in \mathscr{S}(\bbR^2)$.

Note that $\|\rho'_s\|_{W^{m,\infty}} < \infty$ is automatically satisfied in the confined domain or the periodic domain, the local-in-time existence will be similarly obtained there. More precisely, any choice of $\rho_s$ can be absorbed into the initial disturbance $\theta_0$ in $H^m$ on both $\bbT^2$ and $[-1,1]\times \bbT$. Keeping this in mind, one may refer to \cite{Cordoba} (or see Proposition~\ref{prop_energy_T}) for the confined case, as an example.


\begin{proposition}
    Let $m \in \bbN$ with $m > 2$ and $\rho_s$ be a smooth function with $\| \rho_s' \|_{W^{m,\infty}} < \infty$. Then for any $\tht_0 \in H^{m}(\bbR^2)$, there exist a maximal time of existence $T^*=T^*(m,\rho_s,\tht_0)>0$ and a unique local-in-time solution to \eqref{eq:SIPM_g},
    \begin{equation*}
        \tht \in C([0,T^*);H^m(\bbR^2)) \cap C^1([0,T^*);H^{m-1}(\bbR^2)).
    \end{equation*}
    Moreover, there exists a constant $C = C(m)>0$ such that
    \begin{equation}\label{loc_bdd}
        \| \tht(t) \|_{H^{m}} \leq \frac{\| \rho_s' \|_{W^{m,\infty}} \| \tht_0 \|_{H^m}}{(\| \rho_s' \|_{W^{m,\infty}}+\| \tht_0 \|_{H^m})e^{-C\| \rho_s' \|_{W^{m,\infty}}t} - \| \tht_0 \|_{H^m}}
    \end{equation}
    for all $t \in [0,T^*)$.
\end{proposition}
\begin{proof}
    We only provide the a priori $H^{m}$ energy estimate for the solution $\tht$. Interested readers may follow a standard local well-posedness theory (cf. \cite{Majda}) to complete the proof. Let $\tht$ be a global classical solution to \eqref{eq:SIPM_g} with the initial data $\tht_0 \in H^m(\bbR^2)$. From the $\tht$-evolution equation in \eqref{eq:SIPM_g}, we can have
    \begin{equation}\label{loc_est}
        \frac{1}{2} \frac{\ud}{\ud t} \sum_{|\alpha| \leq m} \int_{\bbR^2} | \partial^{\alpha} \tht |^2 \,\ud x = - \sum_{|\alpha| \leq m} \int_{\bbR^2} \partial^{\alpha} (u \cdot \nabla) \tht \partial^{\alpha} \tht \,\ud x - \sum_{|\alpha|\leq m} \int_{\bbR^2} \partial^{\alpha} (u_2 \rho_s'(x_2)) \partial^{\alpha} \tht \,\ud x.
    \end{equation}
    By $H^m(\bbR^2) \hookrightarrow L^{\infty}(\bbR^2)$ and $\| u \|_{H^m} \leq \| \tht \|_{H^m}$, there holds 
    \begin{equation}\label{loc_est_2}
        \left| - \sum_{|\alpha| \leq m} \int_{\bbR^2} \partial^{\alpha} (u \cdot \nabla) \tht \partial^{\alpha} \tht \,\ud x \right| \leq C(\| \nabla u \|_{L^{\infty}} + \| \nabla \tht \|_{L^{\infty}}) \| \tht \|_{H^m}^2 \leq C\| \tht \|_{H^m}^3.
    \end{equation}
    By the chain rule, we have
    \begin{equation*}
        \left| - \sum_{|\alpha|\leq m} \int_{\bbR^2} \partial^{\alpha} (u_2 \rho_s'(x_2)) \partial^{\alpha} \tht \,\ud x \right| \leq C\| \rho_s' \|_{W^{m,\infty}} \| \tht \|_{H^m}^2.
    \end{equation*}
    Therefore,
    \begin{equation*}
        \frac{\ud}{\ud t} \| \tht \|_{H^m} \leq C \| \tht \|_{H^m}^2 + C\| \rho_s' \|_{W^{m,\infty}} \| \tht \|_{H^m}.
    \end{equation*}
    Letting $y(t) := \| \tht(t) \|_{H^{m}}$ and $C_{\rho} := \| \rho_s' \|_{W^{m,\infty}}$ gives
    \begin{equation*}
        \frac{\ud}{\ud t} y(t) \leq C \left(y(t)^2 + C_{\rho} y(t) \right).
    \end{equation*}
    Dividing both terms by $y^2(t)$, we have
    \begin{equation*}
        -\frac{\ud}{\ud t} \left( \frac{1}{y(t)} + \frac{1}{C_{\rho}} \right) \leq C C_{\rho} \left( \frac{1}{y(t)} + \frac{1}{C_{\rho}} \right).
    \end{equation*}
    Gr\"{o}nwall's inequality implies
    \begin{equation*}
        \left( \frac{1}{y_0} + \frac{1}{C_{\rho}} \right) e^{-CC_{\rho}t} \leq \left( \frac{1}{y(t)} + \frac{1}{C_{\rho}} \right),
    \end{equation*}
    thus,
    \begin{equation*}
        y(t) \leq \frac{C_{\rho} y_0}{(C_{\rho} + y_0)e^{-CC_{\rho}t} - y_0}.
    \end{equation*}
    This completes the proof.
\end{proof}

\section{Proof of Theorem~\ref{thm:stb_R}}\label{sec_R}
In this section, we use the Fourier transform with the following notations:
\begin{equation*}
    \widehat{\sigma}(\xi_2) := \int_{\bbR} e^{-2\pi i x_2 \xi_2} \sigma(x_2) \,\ud x_2, \qquad \mathscr{F} \tht(\xi) := \int_{\bbR^2} e^{-2\pi i x \cdot \xi} \tht(x) \,\ud x.
\end{equation*}

\subsection{Energy inequality}
\begin{proposition}\label{prop_energy}
Let $\theta$ be a global classical solution of \eqref{eq:SIPM}, where $N>0$ and $\sigma \in \mathscr{S}(\bbR)$ with $|\sigma(x_2)| \leq \frac{N}{2}$ \Black{for all $x_2 \in \bbR$.} Then for any given $m \in \bbN$ with $m > 2$, there exists a constant $C=C(m)>0$ such that
\begin{equation}\label{energy}
\begin{gathered}
    \frac{1}{3}\sup_{t \in [0,T]} \|\theta\|_{H^{m}}^2 + N \int_0^T \|R_1\theta\|_{H^{m}}^{2} \,\ud t \leq \| \tht_0 \|_{H^{m}}^2 \\
    + C\sum_{k=1}^{m} \left( N^{-1} \| \sigma \|_{W^{m+1,\infty}} \right)^k \left( \sup_{t \in [0,T]} \| \theta \|_{H^m}^2 + N\int_0^T \| R_1 \theta \|_{H^m}^2 \,\ud t \right) \\
    + C\sum_{k=0}^{m} \left( N^{-1} \| \sigma \|_{W^{m+1,\infty}} \right)^k \left( \sup_{t \in [0,T]} \| \theta \|_{H^m} \int_0^T \|R_1 \theta \|_{H^m}^2 \,\ud t + \sup_{t \in [0,T]} \| \theta \|_{H^m}^2 \int_0^T \|u_2\|_{W^{1,\infty}} \, \ud t \right).
\end{gathered}
\end{equation}
\end{proposition}

\begin{proof}
Multiplying \eqref{eq:SIPM} by $\frac{N\theta}{N + \sigma(x_2)}$ and integrating over, we have
\begin{equation*}
    \frac{1}{2} \frac{\ud}{\ud t} \int \frac{N|\tht|^2}{N+\sigma(x_2)} \,\ud x + N\int |R_1 \theta|^2 \,\ud x \leq \left| \int (u \cdot \nabla) \theta \frac{N\theta}{N + \sigma(x_2)} \,\ud x \right|.
\end{equation*}
Integration by parts and the hypothesis \Black{$|\sigma(x_2)|\leq \frac{N}{2}$} yield
\begin{align*}
    \left| \int (u \cdot \nabla) \theta \frac{N\theta}{N + \sigma(x_2)} \,\ud x \right| &= \left| \frac{1}{2} \int u_2 |\theta|^2 \frac{N\sigma'(x_2)}{(N + \sigma(x_2))^2} \,\ud x \right| \leq N^{-1} \| \sigma' \|_{L^{\infty}} \| u_2 \|_{L^{\infty}} \int \frac{N|\tht|^2}{N+\sigma(x_2)} \,\ud x.
\end{align*}
Hence, integrating both sides over time, we can infer
\begin{equation}\label{L2_est}
\begin{gathered}
    \frac{1}{2} \sup_{t \in [0,T]} \int \frac{N|\tht|^2}{N+\sigma(x_2)} \,\ud x + N \int_0^T \|R_1 \theta\|_{L^2}^2 \,\ud t \\
    \leq \frac{1}{2} \int \frac{N|\tht_0|^2}{N+\sigma(x_2)} \,\ud x + N^{-1} \| \sigma' \|_{L^{\infty}} \sup_{t \in [0,T]} \int \frac{N|\tht|^2}{N+\sigma(x_2)} \,\ud x \int_0^T  \| u_2 \|_{L^{\infty}} \,\ud t. 
\end{gathered}    
\end{equation}
From \eqref{eq:SIPM}, we have for $1 \leq |\alpha| \leq m$ that
\begin{equation*}
	\frac 12 \frac {\ud}{\ud t} \int|\partial^\alpha \theta|^2 \,\ud x + N \int |R_1 \partial^\alpha \theta|^2 \,\ud x = - \int \partial^\alpha(u \cdot \nabla) \theta \partial^\alpha \theta \,\ud x - \int \partial^{\alpha} (u_2 \sigma(x_2)) \partial^{\alpha} \theta \,\ud x.
\end{equation*}
We only consider the case $|\alpha| = m$ because the others can be treated similarly. Integrating over time, we obtain
\begin{equation}\label{alp_energy}
    \begin{gathered}
    	\frac 12 \sup_{t \in [0,T]} \|\partial^\alpha \theta\|_{L^2}^2 + N \int_0^T \|R_1 \partial^\alpha \theta\|_{L^2}^2 \,\ud t \\
    	= \frac 12 \|\partial^\alpha \theta_0\|_{L^2}^2 - \int_0^T \int \partial^\alpha(u \cdot \nabla) \theta \partial^\alpha \theta \,\ud x \ud t - \int_0^T \int \partial^{\alpha} (u_2 \sigma(x_2)) \partial^{\alpha} \theta \,\ud x \ud t.
	\end{gathered}
\end{equation}
We estimate the second and third integrals on the right-hand side considering the cases $\partial^{\alpha} \neq \partial_2^m$ and $\partial^{\alpha} = \partial_2^m$ separately. In the former case, the divergence free condition implies
\begin{align*}
    \left| \int_0^T \int \partial^\alpha(u \cdot \nabla) \theta \partial^\alpha \theta \,\ud x \ud t \right| &= 	\left| \int_0^T \int \left(\partial^\alpha(u \cdot \nabla) \theta - (u \cdot \nabla) \partial^{\alpha} \theta \right) \partial^\alpha \theta \,\ud x \ud t \right| \\
    &\leq C\int_0^T (\| \nabla u \|_{L^\infty} \| \tht \|_{H^m} + \| u \|_{H^m} \| \nabla \tht \|_{L^{\infty}}) \| R_1 \tht \|_{H^m} \,\ud t \\
    &\leq C\sup_{t \in [0,T]} \| \theta \|_{H^m} \int_0^T \|R_1 \theta \|_{H^m}^2 \,\ud t.
\end{align*}
We used $H^{m-1}(\bbR^2) \hookrightarrow L^{\infty}(\bbR^2)$ in the last inequality. The third integral is bounded by 
\begin{equation*}
    \left| \int_0^T \int \partial^{\alpha} (u_2 \sigma(x_2)) \partial^{\alpha} \theta \,\ud x \ud t \right| \leq C \| \sigma \|_{W^{m+1,\infty}} \int_0^T \| R_1 \theta \|_{H^m}^2 \,\ud t.
\end{equation*}
Otherwise, we can see from $(u \cdot \nabla) \tht = u_1\partial_1 \tht + u_2 \partial_2 \tht$ and the divergence free condition that
\begin{align*}
    \left| \int_0^T \int \partial_2^m(u \cdot \nabla) \theta \partial_2^m \theta \,\ud x \ud t \right| &\leq \left| \int_0^T \int \partial_2^{m-1}(\partial_2u_1 \partial_1 \tht )\partial_2^m \theta \,\ud x \ud t \right| + \left| \int_0^T \int \partial_2^{m-1}(\partial_2u_2 \partial_2 \tht)\partial_2^m \theta \,\ud x \ud t \right| \\
    &\leq C\sup_{t \in [0,T]} \| \theta \|_{H^m} \int_0^T \|R_1 \theta \|_{H^m}^2 \,\ud t + \left| \int_0^T \int \partial_2^{m-1}(\partial_1u_1 \partial_2 \tht)\partial_2^m \theta \,\ud x \ud t \right|.
\end{align*}
Using the integration by parts twice, we can deduce
\begin{align*}
    \left| \int_0^T \int \partial_2^{m-1}(\partial_1u_1 \partial_2 \tht)\partial_2^m \theta \,\ud x \ud t \right| \leq C\sup_{t \in [0,T]} \| \theta \|_{H^m} \int_0^T \|R_1 \theta \|_{H^m}^2 \,\ud t + \left| \int_0^T \int \partial_1u_1 |\partial_2^m \tht|^2 \,\ud x \ud t \right|,
\end{align*} hence,
\begin{align*}
    \left| \int_0^T \int \partial_2^m(u \cdot \nabla) \theta \partial_2^m \theta \,\ud x \ud t \right| \leq C\sup_{t \in [0,T]} \| \theta \|_{H^m} \int_0^T \|R_1 \theta \|_{H^m}^2 \,\ud t + \sup_{t \in [0,T]} \| \theta \|_{H^m}^2 \int_0^T \| \partial_2 u_2\|_{L^{\infty}} \, \ud t.
\end{align*}
On the other hand, we have
\begin{align*}
    \left| \int_0^T \int \partial_2^{m} (u_2 \sigma(x_2)) \partial_2^{m} \theta \,\ud x \ud t \right|&= \left| \int_0^T \int \partial_2^{2m} (u_2 \sigma(x_2)) \theta \,\ud x \ud t \right| \\
    &\leq \left| \int_0^T \int \partial_2^{2m-1} (\partial_2u_2 \sigma(x_2)) \theta \,\ud x \ud t \right| + \left| \int_0^T \int u_2\theta \partial_2^{2m}\sigma(x_2) \,\ud x \ud t \right|.
\end{align*}
Performing the integration by parts and using the divergence free condition, we see that
\begin{align*}
   \left| \int_0^T \int \partial_2^{2m-1} (\partial_2u_2 \sigma(x_2)) \theta \,\ud x \ud t\right| &= \left| \int_0^T \int \partial_2^{2m-1} \partial_1(u_1 \sigma(x_2)) \theta \,\ud x \ud t\right| \\
   &= \left| \int_0^T \int \partial_2^{m} (u_1 \sigma(x_2)) \partial_2^{m-1} \partial_1 \theta \,\ud x \ud t \right| \\
   &\leq C \| \sigma \|_{W^{m+1,\infty}} \int_0^T \| R_1 \theta \|_{H^m}^2 \,\ud t.
\end{align*}
From our equation \eqref{eq:SIPM} and the integration by parts, we have
\begin{align*}
    \int u_2\theta \partial_2^{2m}\sigma(x_2) \,\ud x &= -\int \frac{\partial_2^{2m}\sigma(x_2)}{N + \sigma(x_2)} \theta ( \theta_t + (u \cdot \nabla) \theta) \,\ud x \\
    &= -\frac{1}{2} \frac{\ud}{\ud t} \int \frac{\partial_2^{2m}\sigma(x_2)}{N+\sigma(x_2)}|\tht|^2 \,\ud x + \frac{1}{2}\int u_2 |\tht|^2 \partial_2 \frac{\partial_2^{2m}\sigma(x_2)}{N+\sigma(x_2)} \,\ud x \\
    &= -\frac{1}{2} \frac{\ud}{\ud t} \int \frac{\partial_2^{2m}\sigma(x_2)}{N+\sigma(x_2)}|\tht|^2 \,\ud x -\frac{1}{2} \int \partial_2(u_2 |\theta|^2) \frac{\partial_2^{2m}\sigma(x_2)}{N+\sigma(x_2)} \,\ud x,
\end{align*} hence we get
\begin{gather*}
    \left| \int_0^T \int u_2\theta \partial_2^{2m}\sigma(x_2) \,\ud x \ud t \right| \leq \left| \frac{1}{2} \int \frac{\partial_2^{2m}\sigma(x_2)}{N+\sigma(x_2)}|\tht(T)|^2 \,\ud x - \frac{1}{2} \int \frac{\partial_2^{2m}\sigma(x_2)}{N+\sigma(x_2)}|\tht(0)|^2 \,\ud x \right| \\
    + \left| \frac{1}{2} \int_0^T \int \partial_2(u_2 |\theta|^2) \frac{\partial_2^{2m}\sigma(x_2)}{N+\sigma(x_2)} \,\ud x \ud t \right|.
\end{gather*}
Since
\begin{align*}
    \left| \frac{1}{2} \int \frac{\partial_2^{2m}\sigma(x_2)}{N+\sigma(x_2)}|\tht|^2 \,\ud x \right| &= \left| \frac{1}{2} \int \partial_2^m\sigma(x_2) \partial_2^{m} \frac{|\tht|^2}{N+\sigma(x_2)} \,\ud x \right| \\
    &\leq C\sum_{k=1}^{m} \left( N^{-1} \| \sigma \|_{W^{m+1,\infty}} \right)^k \sup_{t \in [0,T]} \| \theta \|_{H^{m}}^2
\end{align*}
and
\begin{gather*}
    \left| \frac{1}{2} \int_0^T \int \partial_2(u_2 |\theta|^2) \frac{\partial_2^{2m}\sigma(x_2)}{N+\sigma(x_2)} \,\ud x \ud t \right| \leq \left| \frac{1}{2} \int_0^T \int \partial_2^{m+1}\sigma(x_2) \partial_2^{m-1} \left(\frac{\partial_2(u_2 |\theta|^2)}{N+\sigma(x_2)} \right) \,\ud x \ud t \right| \\
    \leq C\sum_{k=1}^{m} \left( N^{-1} \| \sigma \|_{W^{m+1,\infty}} \right)^k \left( \sup_{t \in [0,T]} \| \theta \|_{H^m} \int_0^T \|R_1 \theta \|_{H^m}^2 \,\ud t + \sup_{t \in [0,T]} \| \theta \|_{H^m}^2 \int_0^T \|u_2\|_{L^{\infty}} \, \ud t \right),
\end{gather*}
there holds
\begin{gather*}
    \left| \int_0^T \int u_2\theta \partial_2^{2m}\sigma(x_2) \,\ud x \ud t \right| \leq C \sum_{k=1}^{m} \left( N^{-1} \| \sigma \|_{W^{m+1,\infty}} \right)^k \sup_{t \in [0,T]} \| \theta \|_{H^{m}}^2 \\
    + C\sum_{k=1}^{m} \left( N^{-1} \| \sigma \|_{W^{m+1,\infty}} \right)^k \left( \sup_{t \in [0,T]} \| \theta \|_{H^m} \int_0^T \|R_1 \theta \|_{H^m}^2 \,\ud t + \sup_{t \in [0,T]} \| \theta \|_{H^m}^2 \int_0^T \|u_2\|_{L^{\infty}} \, \ud t \right).
\end{gather*}
Inserting the above estimates into \eqref{alp_energy}, we obtain
\begin{gather*}
    \frac{1}{2} \sup_{t \in [0,T]} \| \nabla \tht \|_{H^{m-1}}^2 + N \int_0^T \| \partial_1 \tht \|_{H^{m-1}} \,\ud t \leq \frac{1}{2} \| \nabla \tht_0 \|_{H^{m-1}}^2 \\
    + C\sum_{k=0}^{m} \left( N^{-1} \| \sigma \|_{W^{m+1,\infty}} \right)^k \left( \sup_{t \in [0,T]} \| \theta \|_{H^m}^2 + N\int_0^T \| R_1 \theta \|_{H^m}^2 \,\ud t \right) \\
    + C\sum_{k=0}^{m} \left( N^{-1} \| \sigma \|_{W^{m+1,\infty}} \right)^k \left( \sup_{t \in [0,T]} \| \theta \|_{H^m} \int_0^T \|R_1 \theta \|_{H^m}^2 \,\ud t + \sup_{t \in [0,T]} \| \theta \|_{H^m}^2 \int_0^T \|u_2\|_{W^{1,\infty}} \, \ud t \right).
\end{gather*}
From \eqref{L2_est} and $|\sigma(x_2)| \leq \frac{N}{2}$, we deduce that
\begin{equation*}
\begin{gathered}
    \frac{1}{3} \sup_{t \in [0,T]} \| \tht \|_{L^2}^2 + N \int_0^T \|R_1 \theta\|_{L^2}^2 \,\ud t \leq \| \tht_0 \|_{L^2}^2 + C N^{-1} \| \sigma' \|_{L^{\infty}} \sup_{t \in [0,T]} \| \tht \|_{L^2}^2 \int_0^T  \| u_2 \|_{L^{\infty}} \,\ud t. 
\end{gathered}    
\end{equation*}
Combining the above two inequalities, we establish \eqref{energy} as desired. This completes the proof.
\end{proof}

The next proposition is for the case of $\sigma = 0$. The proof is similar to the proof of Proposition~\ref{prop_energy} with $\sigma=0$, except for using the Kato-Ponce type inequality instead of the calculus inequality. This enables us to drop the condition $m \in \bbN$. We only provide the sketch of proof in Appendix.
\begin{proposition}\label{prop_energy_R}
Let $\theta$ be a classical solution of \eqref{eq:SIPM} with $N>0$ and $\sigma = 0$. Then for any given $m \in \bbR$ with $m > 2$, there exists a constant $C=C(m)>0$ such that
\begin{equation}\label{energy_R}
\begin{gathered}
    \frac{1}{2}\sup_{t \in [0,T]} \|\theta\|_{H^{m}}^2 + N \int_0^T \|R_1\theta\|_{H^{m}}^{2} \,\ud t \\
    \leq \frac{1}{2} \| \tht_0 \|_{H^{m}}^2 + C\sup_{t \in [0,T]} \| \theta \|_{H^m} \int_0^T \|R_1 \theta \|_{H^m}^2 \,\ud t + C \sup_{t \in [0,T]} \| \theta \|_{H^m}^2 \int_0^T \int_{\bbR^2} |\xi| |\mathscr{F} u_2| \, \ud \xi \ud t.
\end{gathered}
\end{equation}
\end{proposition}

\subsection{Proof of the global existence of solutions}

\begin{lemma}\label{key_lem_R}
    Let $\theta$ be a global classical solution of \eqref{eq:SIPM} with $N>0$ and $\sigma \in \mathscr{S}(\bbR)$. Then, for any given $m > 3$ and $s \in [0,m-2)$, there exists a constant $C>0$ such that 
	\begin{equation}\label{key_est_R}
	    \begin{gathered}
    		\int_0^T \int_{\mathbb{R}^2} (1+|\xi|^2)^{\frac s2} |\mathscr{F}u_2| \, \ud \xi \ud t \leq \frac{1}{N}\int_{\mathbb{R}^2} (1+|\xi|^2)^{\frac s2} |\mathscr{F} \theta_0 | \, \ud \xi \\
    		+ \frac{C}{N}\int_0^T \| R_1 \theta \|_{H^{m}}^2 \, \ud t + \frac{C}{N} \left(\sup_{t \in [0,T]} \| \theta \|_{H^m} + \| (1+|\xi_2|^2)^{\frac{s}{2}} \widehat{\sigma}(\xi_2) \|_{L^1} \right) \int_0^T \int_{\mathbb{R}^2} (1+|\xi|^2)^{\frac s2} |\mathscr{F} u_2| \, \ud \xi \ud t
		\end{gathered}
	\end{equation}
	for all $T>0$.
\end{lemma}
\begin{proof}
	From the system \eqref{eq:SIPM}, we have $u_2 = -\partial_1^2 (-\Delta)^{-1}  \theta$, hence
	\begin{equation*}
		\int_0^T \int_{\mathbb{R}^2} (1+|\xi|^2)^{\frac s2} |\mathscr{F}u_2| \, \ud \xi \ud t = \int_0^T \int_{\mathbb{R}^2} (1+|\xi|^2)^{\frac s2} \frac {|\xi_1|^2}{|\xi|^2} |\mathscr{F}\theta| \, \ud \xi \ud t.
	\end{equation*}
	Using Duhamel's formula gives
	\begin{equation}\label{df_tht}
		\Black{\mathscr{F}\theta(\xi,t) = e^{- N\frac {|\xi_1|^2}{|\xi|^2} t} \mathscr{F} \theta_0 (\xi) - \int_0^t e^{- N\frac {|\xi_1|^2}{|\xi|^2}(t-\tau)} \mathscr{F} (u \cdot \nabla \theta) (\xi,\tau) \,\ud \tau - \int_0^t e^{- N\frac {|\xi_1|^2}{|\xi|^2}(t-\tau)} \mathscr{F} (u_2 \sigma) (\xi,\tau) \,\ud \tau .}
	\end{equation}
	Thus, it follows
	\begin{gather*}
		\int_0^T \int_{\mathbb{R}^2} (1+|\xi|^2)^{\frac s2} |\mathscr{F}u_2| \, \ud \xi \ud t \leq I_1 + I_2 + I_3,
	\end{gather*}
	where
	\begin{align*}
		I_1 &:= \frac{1}{N} \int_0^T \int_{\mathbb{R}^2} (1+|\xi|^2)^{\frac s2} N\frac {|\xi_1|^2}{|\xi|^2} e^{- N\frac {|\xi_1|^2}{|\xi|^2} t} |\mathscr{F} \theta_0 (\xi)| \, \ud \xi \ud t, \\
		I_2 &:= \frac{1}{N} \int_0^T \int_{\mathbb{R}^2} \int_0^t (1+|\xi|^2)^{\frac s2} N\frac {|\xi_1|^2}{|\xi|^2} e^{- N\frac {|\xi_1|^2}{|\xi|^2}(t-\tau)} |\mathscr{F} u \cdot \nabla \theta (\xi,\tau)| \,\ud \tau \ud \xi \ud t, \\
		I_3 &:= \frac{1}{N} \int_0^T \int_{\mathbb{R}^2} \int_0^t (1+|\xi|^2)^{\frac s2} N\frac {|\xi_1|^2}{|\xi|^2} e^{- N\frac {|\xi_1|^2}{|\xi|^2}(t-\tau)} |\mathscr{F} (u_2\sigma) (\xi,\tau)| \,\ud \tau \ud \xi \ud t.
	\end{align*}
	By Fubini's theorem, it is clear that
	\begin{equation*}
		I_1 \leq \frac{1}{N}\int_{\mathbb{R}^2} (1+|\xi|^2)^{\frac s2} |\mathscr{F} \theta_0 | \, \ud \xi.
	\end{equation*}
	Using \eqref{conv_ineq_R2_2}, we have
	\begin{align*}
		I_3 &\leq \frac{1}{N} \int_0^T \int_{\mathbb{R}^2} (1+|\xi|^2)^{\frac s2} |\mathscr{F} (u_2 \sigma)| \,\ud \xi \ud t \\
		&\leq \frac{C}{N} \| (1+|\xi_2|^2)^{\frac{s}{2}} \widehat{\sigma}(\xi_2) \|_{L^1} \int_0^T \| (1+|\xi|^2)^{\frac s2} \mathscr{F} u_2 \|_{L^1} \,\ud t. 
	\end{align*}
	From \eqref{conv_ineq_R2} and
	\begin{equation}\label{decom_est}
	    |\mathscr{F} u \cdot \nabla \theta| \leq |\mathscr{F} (u_1 \partial_1 \theta)| + |\mathscr{F}(u_2 \partial_2 \theta)|,    
	\end{equation}
	we can see
	\begin{align*}
		I_2 &\leq \frac{1}{N}\int_0^T \int_{\mathbb{R}^2} (1+|\xi|^2)^{\frac s2} |\mathscr{F} u \cdot \nabla \theta | \, \ud \xi \ud t \\
		&\leq \frac{C}{N}\int_0^T \| (1+|\xi|^2)^{\frac s2} \mathscr{F} u_1 \|_{L^1} \| (1+|\xi|^2)^{\frac s2} \mathscr{F} \partial_1 \theta \|_{L^1}\, \ud t \\
		&\hphantom{\qquad\qquad} + \frac{C}{N}\int_0^T \| (1+|\xi|^2)^{\frac s2} \mathscr{F} u_2 \|_{L^1} \| (1+|\xi|^2)^{\frac s2} \mathscr{F} \partial_2 \theta \|_{L^1}\, \ud t.
    \end{align*}
    \Black{Note for $s \in [0,m-2)$ and $\varepsilon \in (0,m-s-2]$ that
    \begin{align*}
        \| (1+|\xi|^2)^{\frac s2} \mathscr{F} u_1 \|_{L^1} &= \int_{\bbR^2} \frac{1}{(1+|\xi|^2)^{\frac{1+\varepsilon}{2}}} (1+|\xi|^2)^{\frac{s+1+\varepsilon}{2}} |\mathscr{F} u_1 (\xi)| \,\ud \xi \leq C \| u_1 \|_{H^{s+1+\varepsilon}} \leq C \| u \|_{H^{m-1}}
    \end{align*} and
    \begin{align*}
        \| (1+|\xi|^2)^{\frac s2} \mathscr{F} \partial_j \theta \|_{L^1} \leq C \| \partial_j \theta \|_{H^{s+1+\varepsilon}} \leq C \| \partial_j \tht \|_{H^{m-1}}, \qquad j = 1,\,2.
    \end{align*} Thus, it follows
    \begin{align*}
        I_2 &\leq \frac{C}{N}\int_0^T \| u \|_{H^m} \| \partial_1 \theta \|_{H^{m-1}} \, \ud t + \frac{C}{N}\int_0^T \| (1+|\xi|^2)^{\frac s2} \mathscr{F} u_2 \|_{L^1} \| \theta \|_{H^m} \,\ud t \\
		&\leq \frac{C}{N}\int_0^T \| R_1 \theta \|_{H^{m}}^2 \, \ud t + \frac{C}{N}\sup_{t \in [0,T]} \| \theta \|_{H^m} \int_0^T \| (1+|\xi|^2)^{\frac s2} \mathscr{F} u_2 \|_{L^1} \, \ud t.
	\end{align*}}
	We have used for $s \in m-2$ for the third inequality. Combining the estimates for $I_1$, $I_2$ and $I_3$, we obtain \eqref{key_est_R}. This completes the proof.
\end{proof}

Now we are ready to prove the global existence part of Theorem~\ref{thm:stb_R}. Let $m \in \bbN$ with $m > 3$ and $\theta$ be the local-in-time solution to \eqref{eq:SIPM} with $N > 0$ and $\sigma \in \mathscr{S}(\bbR)$. We set $T^* \in (0,\infty]$ be a maximal time of existence. Let $T \in (0,T^*)$ such that 
\begin{equation}\label{con}
    \sup_{t \in [0,T]} \| \tht \|_{H^{m}}^2 + N \int_0^T \| R_1 \tht \|_{H^{m}}^2 \,\ud t \leq 4 \| \tht_0 \|_{H^{m}}^2.
\end{equation}
We recall \eqref{key_est_R}. Since \eqref{ass} implies
\begin{equation*}
    \frac{C}{N} \left(\sup_{t \in [0,T]} \| \theta \|_{H^m} + \| (1+|\xi_2|^2)^{\frac{s}{2}} \widehat{\sigma}(\xi_2) \|_{L^1} \right) \leq \frac{1}{2},
\end{equation*}
together with \eqref{con} we have for $s \in [0,m-2)$ that
\begin{equation}\label{key1}
	\int_0^T \int_{\mathbb{R}^2} (1+|\xi|^2)^{\frac s2} |\mathscr{F}u_2| \, \ud \xi \ud t \leq \frac{C}{N}\int_{\mathbb{R}^2} (1+|\xi|^2)^{\frac s2} |\mathscr{F} \theta_0 | \, \ud \xi + \frac{C}{N}\int_0^T \| R_1 \theta \|_{H^{m}}^2 \, \ud t \leq \frac{C}{N} \| \tht_0 \|_{H^m}.
\end{equation}
Applying \eqref{con} and \eqref{key1} to \eqref{energy} shows that
\begin{gather*}
    \frac{1}{3}\sup_{t \in [0,T]} \|\theta\|_{H^{m}}^2 + N \int_0^T \|R_1\theta\|_{H^{m}}^{2} \,\ud t \leq \| \tht_0 \|_{H^{m}}^2 + C\sum_{k=0}^{m} \left( N^{-1} \| \sigma \|_{W^{m+1,\infty}} \right)^k \left( \| \tht_0 \|_{H^m}^2 + \frac{C}{N} \| \tht_0 \|_{H^m}^3 \right).
\end{gather*}
Therefore, \eqref{ass} with large constant $C_0>0$ gives
\begin{equation*}
     \frac{1}{3}\sup_{t \in [0,T]} \|\theta\|_{H^{m}}^2 + N \int_0^T \|R_1\theta\|_{H^{m}}^{2} \,\ud t \leq \Black{\frac{7}{6} \| \tht_0 \|_{H^{m}}^2,}
\end{equation*}
which satisfies \eqref{con}. Thus, we can take $T=T^*$ by the continuation argument, and $T^* = \infty$ follows together with \eqref{sol_bdd}. This completes the proof. \qed

\subsection{Proof of the convergence of solutions}
In this section, we prove that the classical solution $(u,\tht)$ obtained in the previous section satisfies \eqref{tem_decay_est_sgm}. We first investigate the temporal decay property of the linear stratification with the following lemma. We remark that the decay rate of \eqref{ker_est} is sharp (see Proposition~\ref{prop_sharp}).
\begin{lemma}
    Let $s>1$. Then, there exists a constant $C>0$ such that
    \begin{equation}\label{ker_est}
        \int_{\bbR^2} e^{-N\frac{\xi_1^2}{|\xi|^2}t} |f(\xi)| \,\ud \xi \leq C(1+Nt)^{-\frac{1}{4}} \| f \|_{H^{s}}
    \end{equation}
    for all $f \in \mathscr{S}(\bbR^2)$.
\end{lemma}
\begin{proof}
    It is clear that
    \begin{equation*}
        \int_{\bbR^2} e^{-N\frac{\xi_1^2}{|\xi|^2}t} |f(\xi)| \,\ud \xi  \leq \int_{\bbR^2} |f(\xi)| \,\ud \xi \leq C \| f \|_{H^{s}}.
    \end{equation*}
    We consider $\bbR^2 = \{ |\xi_1| \geq |\xi_2| \} \cup \{ |\xi_1| \leq |\xi_2| \}$. On the first set we have
    \begin{equation*}
        \int_{|\xi_1| \geq |\xi_2|} e^{-N\frac{\xi_1^2}{|\xi|^2}t} |f(\xi)| \,\ud \xi  \leq \Black{e^{-\frac{N}{2}t} \int_{\bbR^2} |f(\xi)| \,\ud \xi.}
    \end{equation*} On the other hand, we can estimate
    \begin{equation*}
        \int_{|\xi_1| \leq |\xi_2|} e^{-N\frac{\xi_1^2}{|\xi|^2}t} |f(\xi)| \,\ud \xi \leq \int_{\bbR^2} e^{-N\frac{\xi_1^2}{2\xi_2^2}t} |f(\xi)| \,\ud \xi \leq \int_{\bbR} \Big\| e^{-N\frac{\xi_1^2}{2\xi_2^2}t} \Big\|_{L^2_{\xi_1}} \| f \|_{L^2_{\xi_1}} \,\ud \xi_2.
    \end{equation*}
    Since
    \begin{equation}\label{ker_L2}
        \Big\| e^{-N\frac{\xi_1^2}{2\xi_2^2}t} \Big\|_{L^2_{\xi_1}} \leq C \left( \frac{|\xi_2|}{\sqrt{Nt}} \right)^{\frac{1}{2}},
    \end{equation}
    it follows
    \begin{equation*}
        \int_{|\xi_1| \leq |\xi_2|} e^{-N\frac{\xi_1^2}{|\xi|^2}t} |f(\xi)| \,\ud \xi \leq C (Nt)^{-\frac{1}{4}} \Black{\| |\xi_2|^{\frac{1}{2}} f \|_{L^1_{\xi_2}L^2_{\xi_1}}.}
    \end{equation*}
    Combining the above estimates with the following inequality $$\Black{\int_{\bbR^2} |f(\xi)| \,\ud \xi + \| |\xi_2|^{\frac{1}{2}} f \|_{L^1_{\xi_2}L^2_{\xi_1}} }\leq C \| f \|_{H^s},$$ we deduce \eqref{ker_est}. This completes the proof.
\end{proof}
\begin{proposition}\label{prop_sharp}
    Let $s \geq 1$ and $j \in \bbN \cup \{0\}$. For any $\varepsilon>0$, there exist a function $f^* \in H^{s} \cap W^{s-1,\infty}(\bbR^2)$ and a constant $c>0$ such that
    \begin{equation}\label{sharp_est}
        \bigg\| \frac{|\xi_1|^j}{|\xi|^j}e^{-N\frac{\xi_1^2}{|\xi|^2}t} \mathscr{F}f^*(\xi) \bigg\|_{L^1(\bbR^2)} \geq c(1+Nt)^{-(\frac{j}{2}+\frac{1}{4}+ \varepsilon)}
    \end{equation} and
    \begin{equation}\label{sharp_est_2}
        \bigg\| \frac{|\xi_1|^j}{|\xi|^j}e^{-N\frac{\xi_1^2}{|\xi|^2}t} \mathscr{F}f^*(\xi) \bigg\|_{L^2(\bbR^2)} \geq c(1+Nt)^{-(\frac{j}{2} + \varepsilon)}
    \end{equation}
    for all $Nt \geq 1$.
\end{proposition}
\begin{proof}
    We let $f^*$ satisfy $$\mathscr{F}f^* = |\xi_1|^{-\frac{1}{2}+2\varepsilon}(1+|\xi|^2)^{-\frac{s}{2}-\frac{1}{4}-2\varepsilon}.$$ Then, it is not hard to verify that $f^* \in H^s \cap W^{s-1,\infty}(\bbR^2)$. We can see \begin{align*}
        \int_{\bbR^2} \frac{|\xi_1|^j}{|\xi|^j}e^{-N\frac{\xi_1^2}{|\xi|^2}t} f^*(\xi) \,\ud \xi &\geq \int_{\bbR^2} e^{-N\frac{\xi_1^2}{|\xi|^2}t} |\xi_1|^{j-\frac{1}{2}+2\varepsilon}(1+|\xi|^2)^{-\frac{s+j}{2}-\frac{1}{4}-2\varepsilon} \,\ud \xi \\
        &\geq \int_{\{ |\xi_1| \leq |\xi_2| \}} e^{-\frac{N}{2} \frac{\xi_1^2}{\xi_2^2}t} |\xi_1|^{j-\frac{1}{2}+2\varepsilon}(1+2\xi_2^2)^{-\frac{s+j}{2}-\frac{1}{4}-2\varepsilon} \,\ud \xi \\
        &= \int_{\bbR} \int_{-|\xi_2|}^{|\xi_2|} e^{-\frac{N}{2} \frac{\xi_1^2}{\xi_2^2}t} |\xi_1|^{j-\frac{1}{2}+2\varepsilon} \,\ud \xi_1 (1+2\xi_2^2)^{-\frac{s+j}{2}-\frac{1}{4}-2\varepsilon} \,\ud \xi_2.
    \end{align*}
    Since $$\int_{-|\xi_2|}^{|\xi_2|} e^{-\frac{N}{2} \frac{\xi_1^2}{\xi_2^2}t} |\xi_1|^{j-\frac{1}{2}+2\varepsilon} \,\ud \xi_1 \geq c (Nt)^{-\frac{j}{2}-\frac{1}{4}-\varepsilon} |\xi_2|^{j+\frac{1}{2}+2\varepsilon}$$ for any $Nt \geq 1$, it follows
    \begin{align*}
        \int_{\bbR^2} \frac{|\xi_1|^j}{|\xi|^j}e^{-N\frac{\xi_1^2}{|\xi|^2}t} f^*(\xi) \,\ud \xi \geq c (Nt)^{-\frac{j}{2}-\frac{1}{4}-\varepsilon} \int_{\bbR} (1+\xi_2^2)^{-\frac{s}{2}-\varepsilon} \,\ud \xi_2 = c (Nt)^{-\frac{j}{2}-\frac{1}{4}-\varepsilon}.
    \end{align*}
    
    In the same way with the above, we can obtain \eqref{sharp_est_2}. We omit the details. This completes the proof.
\end{proof}

The following lemma is to control the $\sigma$ term in the next proposition.
\begin{lemma}
    Let $\theta$ be a global classical solution of \eqref{eq:SIPM} with $N>0$ and $\sigma \in \mathscr{S}(\bbR)$. Then, for any given $m > \frac{3}{2}$ and $s \in [0,m-\frac{3}{2})$, there exists a constant $C>0$ such that
    \begin{equation}\label{L1_est_tem}
        \begin{gathered}
            \left\| \int_0^{\infty} (1+|\xi|^2)^{\frac{s}{2}} |\mathscr{F} u_2| \,\ud t \right\|_{L^{1}_{\xi_2}L^{2}_{\xi_1}} \leq \frac{1}{N} \| (1+|\xi|^2)^{\frac{s}{2}} \mathscr{F} \tht_0 \|_{L^1_{\xi_2}L^2_{\xi_1}} + \frac{C}{N}\int_0^{\infty} \| R_1 \theta \|_{H^{m}}^2 \, \ud t \\
            + \frac{C}{N}\sup_{t \in [0,T]} \| \theta \|_{H^m} \int_0^{\infty} \int_{\mathbb{R}^2} (1+|\xi|^2)^{\frac s2} |\mathscr{F} u_2| \, \ud \xi \ud t + \frac{1}{N} \| (1+|\xi_2|^2)^{\frac{s}{2}} \widehat{\sigma} \|_{L^1} \left\| \int_0^{\infty} (1+|\xi|^2)^{\frac{s}{2}} |\mathscr{F} u_2| \,\ud t \right\|_{L^1_{\xi_2}L^2_{\xi_1}}.
        \end{gathered}
    \end{equation}
\end{lemma}
\begin{remark}
    Suppose \eqref{sol_bdd} and \eqref{ass} be satisfied. Then, there exists a constant $C>0$ such that
    \begin{equation}\label{L1_est_tem_2}
        \left\| \int_0^{\infty} (1+|\xi|^2)^{\frac{s}{2}} |\mathscr{F} u_2| \,\ud t \right\|_{L^{1}_{\xi_2}L^{2}_{\xi_1}} \leq \frac{C}{N} \| \tht_0 \|_{H^m}.
    \end{equation}
\end{remark}
\begin{proof}
    We recall $u_2 = -\partial_1^2 (-\Delta)^{-1}  \theta$ and \eqref{df_tht} to have
    \begin{equation*}
        \left\| \int_0^{\infty} (1+|\xi|^2)^{\frac{s}{2}} |\mathscr{F} u_2| \,\ud t \right\|_{L^{1}_{\xi_2}L^{2}_{\xi_1}} \leq I_4 + I_5 + I_6,
    \end{equation*}
    where
    \begin{align*}
        I_4 &:= \left\| \int_0^{\infty} (1+|\xi|^2)^{\frac{s}{2}} \frac{\xi_1^2}{|\xi|^2} e^{-N \frac{\xi_1^2}{|\xi|^2}t} |\mathscr{F} \tht_0| \,\ud t \right\|_{L^{1}_{\xi_2}L^{2}_{\xi_1}}, \\
        I_5 &:= \left\| \int_0^{\infty} (1+|\xi|^2)^{\frac{s}{2}} \frac{\xi_1^2}{|\xi|^2} \int_0^t e^{-N \frac{\xi_1^2}{|\xi|^2}(t-\tau)} |\mathscr{F} (u \cdot \nabla)\tht| \,\ud \tau \ud t \right\|_{L^{1}_{\xi_2}L^{2}_{\xi_1}}, \\
        I_6 &:= \left\| \int_0^{\infty} (1+|\xi|^2)^{\frac{s}{2}} \frac{\xi_1^2}{|\xi|^2} \int_0^t e^{-N \frac{\xi_1^2}{|\xi|^2}(t-\tau)} |\mathscr{F} (u_2 \sigma)| \,\ud \tau \ud t \right\|_{L^{1}_{\xi_2}L^{2}_{\xi_1}}.
    \end{align*}
    The fundamental theorem of calculus shows that
    \begin{equation*}
        I_4 \leq \frac{1}{N} \| (1+|\xi|^2)^{\frac{s}{2}} \mathscr{F} \tht_0 \|_{L^1_{\xi_2}L^2_{\xi_1}}.
    \end{equation*}
    By Fubini's theorem and Minkowski's integral inequality, we have
    \begin{equation*}
        I_5 \leq \frac{1}{N} \left\| \int_0^{\infty} (1+|\xi|^2)^{\frac{s}{2}} |\mathscr{F} (u \cdot \nabla)\tht| \,\ud t \right\|_{L^{1}_{\xi_2}L^{2}_{\xi_1}} \leq \frac{1}{N} \int_0^{\infty} \| (1+|\xi|^2)^{\frac{s}{2}} \mathscr{F} (u \cdot \nabla)\tht \|_{L^{1}_{\xi_2}L^{2}_{\xi_1}} \,\ud t.
    \end{equation*}
    Using \eqref{decom_est} and \eqref{conv_ineq_R2} shows
    \begin{equation*}
    \begin{aligned}
        &\int_0^{\infty} \| (1+|\xi|^2)^{\frac{s}{2}} \mathscr{F} (u \cdot \nabla)\tht \|_{L^{1}_{\xi_2}L^{2}_{\xi_1}} \,\ud t \\
        &\hphantom{\qquad\qquad}\leq \int_0^{\infty} \| (1+|\xi|^2)^{\frac{s}{2}} \mathscr{F} u_1 \|_{L^1} \| (1+|\xi|^2)^{\frac{s}{2}}  \mathscr{F} \partial_1 \tht \|_{L^{1}_{\xi_2}L^{2}_{\xi_1}} \,\ud t \\
        &\hphantom{\qquad\qquad\qquad\qquad} +\int_0^{\infty} \| (1+|\xi|^2)^{\frac{s}{2}} \mathscr{F} u_2 \|_{L^1} \| (1+|\xi|^2)^{\frac{s}{2}}  \mathscr{F} \partial_2 \tht \|_{L^{1}_{\xi_2}L^{2}_{\xi_1}} \,\ud t \\
        &\hphantom{\qquad\qquad}\leq C\int_0^{\infty} \| u \|_{H^m} \| \partial_1 \theta \|_{H^{m-1}} \, \ud t + C\int_0^{\infty} \| (1+|\xi|^2)^{\frac{s}{2}} \mathscr{F} u_2 \|_{L^1} \| \theta \|_{H^m} \,\ud t \\
		&\hphantom{\qquad\qquad}\leq C\int_0^{\infty} \| R_1 \theta \|_{H^{m}}^2 \, \ud t + C\sup_{t \in [0,T]} \| \theta \|_{H^m} \int_0^{\infty} \| (1+|\xi|^2)^{\frac{s}{2}} \mathscr{F} u_2 \|_{L^1} \, \ud t.
    \end{aligned}
    \end{equation*}
    Thus,
    \begin{equation*}
        \Black{I_5 \leq \frac{C}{N} \int_0^{\infty} \| R_1 \theta \|_{H^{m}}^2 \, \ud t + \frac{C}{N} \sup_{t \in [0,T]} \| \theta \|_{H^m} \int_0^{\infty} \| (1+|\xi|^2)^{\frac{s}{2}} \mathscr{F} u_2 \|_{L^1} \, \ud t.}
    \end{equation*}
    The Fubini's theorem implies
    \begin{align*}
        I_6 &\leq \frac{1}{N} \left\| \int_0^{\infty} (1+|\xi|^2)^{\frac{s}{2}} |\mathscr{F} (u_2 \sigma)| \,\ud t \right\|_{L^{1}_{\xi_2}L^{2}_{\xi_1}}.
    \end{align*}
    Then, applying \eqref{conv_ineq_R2_2} gives
    \begin{align*}
        I_6 \leq \frac{1}{N} \| (1+|\xi_2|^2)^{\frac{s}{2}} \widehat{\sigma} \|_{L^1} \left\| \int_0^{\infty} (1+|\xi|^2)^{\frac{s}{2}} |\mathscr{F} u_2| \,\ud t \right\|_{L^{1}_{\xi_2}L^{2}_{\xi_1}}.
    \end{align*}
    Collecting the above estimates for $I_4$, $I_5$, and $I_6$, we obtain \eqref{L1_est_tem}. This completes the proof.
\end{proof}

\begin{proposition}\label{prop_tem_sgm}
	Let $\theta$ be a global classical solution of \eqref{eq:SIPM} with $N>0$ and $\sigma \in \mathscr{S}(\bbR)$. Assume that \eqref{ass}, \eqref{sol_bdd}, and \eqref{sol_bdd1} are satisfied. Then for any given $m > \frac{5}{2}$ and $p \in (4,\infty]$, there exists a constant $C>0$ such that 
	\begin{equation}\label{sol_bdd_2}
		\left\| \int_{\bbR^2} |\mathscr{F} \tht(t,\xi)| \,\ud \xi \right\|_{L^p(0,\infty)} \leq \frac{C}{N^{\frac{1}{p}}} \| \tht_0 \|_{H^m}. 
	\end{equation}
\end{proposition}
\begin{proof}
    From \eqref{df_tht} we have
    \begin{equation*}
       \left\| \int_{\bbR^2} |\mathscr{F} \tht(t,\xi)| \,\ud \xi \right\|_{L^p(0,\infty)} \leq I_7 + I_8 + I_9,
    \end{equation*}
    where
    \begin{align*}
        I_7 &:= \Big\| e^{-N\frac{\xi_1^2}{|\xi|^2}t} \mathscr{F} \tht_0 \Big\|_{L^p(0,\infty;L^1)}, \\
        I_8 &:= \left\| \int_0^t e^{-N\frac{\xi_1^2}{|\xi|^2}(t-\tau)} \mathscr{F} (u \cdot \nabla) \tht \,\ud \tau \right\|_{L^p(0,\infty;L^1)}, \\
        I_9 &:= \left\| \int_0^t e^{-N\frac{\xi_1^2}{|\xi|^2}(t-\tau)} \mathscr{F} (u_2 \sigma) \,\ud \tau \right\|_{L^p(0,\infty;L^1)}.
    \end{align*}
    \eqref{ker_est} gives
    \begin{equation*}
        I_7 \leq C\left\| (1+Nt)^{-\frac{1}{4}} \right\|_{L^p(0,\infty)} \| \tht_0 \|_{H^m} \leq \frac{C}{N^{\frac{1}{p}}} \| \tht_0 \|_{H^m}.
    \end{equation*}
    To estimate $I_8$ we first note by \eqref{decom_est} and \eqref{conv_ineq_R2} that
    \begin{align*}
       \left\| \int_0^{t} e^{-N\frac{\xi_1^2}{|\xi|^2}(t-\tau)} |\mathscr{F} (u \cdot \nabla) \tht| \,\ud \tau \right\|_{L^1} &\leq \Black{\int_0^t \| \mathscr{F} (u \cdot \nabla) \tht \|_{L^1} \,\ud t} \\
       &\leq C \int_0^t (\| R_1 \tht \|_{H^m}^2 + \| \tht \|_{H^m} \| \mathscr{F} u_2 \|_{L^1}) \,\ud \tau.
    \end{align*}
    On the other hand, we can have from \eqref{ker_L2}
    \begin{equation*}
    \begin{aligned}
       \left\| \int_0^{t} e^{-N\frac{\xi_1^2}{|\xi|^2}(t-\tau)} |\mathscr{F} (u \cdot \nabla) \tht| \,\ud \tau \right\|_{L^1} &\leq \int_0^t \left\| \Big\| e^{-N\frac{\xi_1^2}{2|\xi|^2}(t-\tau)} \Big\|_{L^2_{\xi_1}}  \left\|\mathscr{F} (u \cdot \nabla) \tht \right \|_{L^2_{\xi_1}} \right\|_{L^1_{\xi_2}} \,\ud \tau \\
        &\leq C \int_0^t (N(t-\tau))^{-\frac{1}{4}} \left\| (1+|\xi|^2)^{\frac{1}{4}} \mathscr{F} (u \cdot \nabla) \tht \right\|_{L^1_{\xi_2}L^2_{\xi_1}} \,\ud \tau \\
        &\leq C \int_0^t (N(t-\tau))^{-\frac{1}{4}} (\| R_1 \tht \|_{H^m}^2 + \| \tht \|_{H^m} \| (1+|\xi|^2)^{\frac{1}{4}} \mathscr{F} u_2 \|_{L^1}) \,\ud \tau.
    \end{aligned}
    \end{equation*}
    By the above inequalities, it follows
    \begin{equation*}
    \begin{aligned}
       I_8 &\leq C \left\| \int_0^t (1+N(t-\tau))^{-\frac{1}{4}} (\| R_1 \tht \|_{H^m}^2 + \| \tht \|_{H^m} \| (1+|\xi|^2)^{\frac{1}{4}} \mathscr{F} u_2 \|_{L^1}) \,\ud \tau \right\|_{L^p(0,\infty)} \\
       &\leq \frac{C}{N^{\frac{1}{p}}} \left( \int_0^{\infty} \| R_1 \tht \|_{H^m}^2 \,\ud t + \sup_{t \in [0,\infty)} \| \tht \|_{H^m} \int_0^{\infty} \| (1+|\xi|^2)^{\frac{1}{4}} \mathscr{F} u_2 \|_{L^1} \,\ud t \right) \\
       &\leq \frac{C}{N^{\frac{1}{p}}} \| \tht_0 \|_{H^m}.
    \end{aligned}
    \end{equation*}
    We have used Minkowski's integral inequality with \eqref{sol_bdd}, \eqref{sol_bdd1}, and \eqref{ass}. 
    
    It remains to estimate $I_9$. By Minkowski's integral inequality, we have
    \begin{align*}
        \left\| \int_0^{t} e^{-N\frac{\xi_1^2}{|\xi|^2}(t-\tau)} |\mathscr{F} (u_2 \sigma)| \,\ud \tau \right\|_{L^p(0,\infty;L^1)} &= \left\| \left( \frac{|\xi|^2}{N\xi_1^2} \right)^{\frac{1}{p}} \int_0^{t} \left( N\frac{\xi_1^2}{|\xi|^2} \right)^{\frac{1}{p}} e^{-N\frac{\xi_1^2}{|\xi|^2}(t-\tau)} |\mathscr{F} (u_2 \sigma)| \,\ud \tau \right\|_{L^p(0,\infty;L^1)} \\
        &\leq \left\| \left( \frac{|\xi|^2}{N\xi_1^2} \right)^{\frac{1}{p}} \int_0^{\infty} |\mathscr{F} (u_2 \sigma)| \,\ud t \right\|_{L^1}.
    \end{align*}
    We split the domain $\bbR^2 = \{ |\xi_1| \geq |\xi_2| \} \cup \{ |\xi_1| \leq |\xi_2| \}$. On the first set, using \eqref{conv_ineq_R2_2}, \eqref{sol_bdd1}, and \eqref{ass} gives
    \begin{align*}
        \left\| \left( \frac{|\xi|^2}{N\xi_1^2} \right)^{\frac{1}{p}} \int_0^{\infty} |\mathscr{F} (u_2 \sigma)| \,\ud t \right\|_{L^1(\{ |\xi_1| \geq |\xi_2| \})} &\leq \frac{C}{N^{\frac{1}{p}}}\left\| \int_0^{\infty} |\mathscr{F} (u_2 \sigma)| \,\ud t \right\|_{L^1} \\
        &\leq \frac{C}{N^{\frac{1}{p}}} \| \widehat{\sigma} \|_{L^1} \int_0^{\infty} \| \mathscr{F} u_2 \|_{L^1} \, \ud t \\
        &\leq \frac{C}{N^{\frac{1}{p}}} \| \tht_0 \|_{H^m}.
    \end{align*}
    On the other set, we have for $s \in (1,m-\frac{3}{2})$ that
    \begin{align*}
        &\left\| \left( \frac{|\xi|^2}{N\xi_1^2} \right)^{\frac{1}{p}} \int_0^{\infty} |\mathscr{F} (u_2 \sigma)| \,\ud t \right\|_{L^1(\{ |\xi_1| \leq |\xi_2| \})} \\
        &\hphantom{\qquad\qquad}\leq \frac{C}{N^{\frac{1}{p}}}\left\| \frac{\xi_2^{\frac{2}{p}}}{\xi_1^{\frac{2}{p}}(1+\xi_1^2)^{\frac{s}{4}-\frac{1}{p}}} (1+\xi_1^2)^{\frac{s}{4}-\frac{1}{p}}\int_0^{\infty} |\mathscr{F} (u_2 \sigma)| \,\ud t \right\|_{L^1} \\
        &\hphantom{\qquad\qquad}\leq \frac{C}{N^{\frac{1}{p}}} \left\| \frac{1}{\xi_1^{\frac{2}{p}}(1+\xi_1^2)^{\frac{s}{4}-\frac{1}{p}}} \right\|_{L^2_{\xi_1}} \left\| \left\| \int_0^{\infty} (1+|\xi|^2)^{\frac{s}{4}} |\mathscr{F} (u_2 \sigma)| \,\ud t \right\|_{L^1_{\xi_2}} \right\|_{L^2_{\xi_1}} \\
        &\hphantom{\qquad\qquad}\leq \frac{C}{N^{\frac{1}{p}}} \left\| \left( \int_0^{\infty} (1+|\xi|^2)^{\frac{s}{4}} |\mathscr{F} u_2| \,\ud t \right) * (1+|\xi_2|^2)^{\frac{s}{4}} |\widehat{\sigma}(\xi_2)| \right\|_{L^{2}_{\xi_1}L^{1}_{\xi_2}} \\  &\hphantom{\qquad\qquad}\leq \frac{C}{N^{\frac{1}{p}}} \| (1+|\xi_2|^2)^{\frac{s}{4}} \widehat{\sigma} \|_{L^1} \left\| \int_0^{\infty} (1+|\xi|^2)^{\frac{s}{4}} |\mathscr{F} u_2| \,\ud t \right\|_{L^{2}_{\xi_1}L^{1}_{\xi_2}}.
    \end{align*}
    Applying \eqref{ass} and \eqref{L1_est_tem_2} with Minkowski's integral inequality, we obtain
    \begin{align*}
        \left\| \left( \frac{|\xi|^2}{N\xi_1^2} \right)^{\frac{1}{p}} \int_0^{\infty} |\mathscr{F} (u_2 \sigma)| \,\ud t \right\|_{L^1(\{ |\xi_1| \leq |\xi_2| \})} \leq \frac{C}{N^{\frac{1}{p}}}\| \tht_0 \|_{H^m}.
    \end{align*}
    Combining the above estimates, we obtain $$I_9 \leq \frac{C}{N^{\frac{1}{p}}}\| \tht_0 \|_{H^m}.$$ By the estimates for $I_7$, $I_8$, and $I_9$, \eqref{sol_bdd_2} follows. This completes the proof.
\end{proof}

Now, we are ready to prove \eqref{tem_decay_est_sgm}. Let all assumptions for Theorem~\ref{thm:stb_T} be satisfied. Then, we obtain $\tht$ which fulfills \eqref{eq:SIPM} and \eqref{sol_bdd}. We can see for any $t > t' > 0,$
\begin{equation*}
    \left\| \frac{\tht(t) - \tht(t')}{t-t'} \right\|_{L^{\infty}} \leq \frac{1}{t-t'} \int_{t'}^t \| \partial_t \tht \|_{L^{\infty}} \,\ud t \leq \sup_{t \in [0,\infty)} \| \partial_t \tht \|_{H^{m-1}} \leq C,
\end{equation*} thus,
\begin{equation*}
    \| \tht(t) - \tht(t')\|_{L^{\infty}} \leq C|t-t'|.
\end{equation*}
Suppose that \eqref{tem_decay_est_sgm} fails. Then, there exists a sequences $\{ N_k \}_{k \in \bbN}$ and $\{ t_k \}_{k \in \bbN}$ such that $N_k t_k \to \infty$ as $k \to \infty$ and $|\tht(t_n)| \geq \varepsilon$ for some $\varepsilon > 0$. When $t_k \to \infty$ as $k \to \infty$, we have for any $p > 4$
\begin{equation*}
    \| \tht \|_{L^p(0,\infty;L^{\infty})} = \infty,
\end{equation*}
which contradicts to \eqref{sol_bdd_2}. On the other hand, if $N_k \to \infty$ as $k \to \infty$, then
\begin{equation*}
    \| \tht \|_{L^p(0,\infty;L^{\infty})} \geq c \varepsilon, \qquad k \in \bbN
\end{equation*}
for some $c>0$. This contradicts to \eqref{sol_bdd_2}. Therefore, we obtain \eqref{tem_decay_est_sgm}. This completes the proof.

\subsection{Proof of \texorpdfstring{$L^{2}$}{} norm decay estimates}

In this section, we prove \eqref{tem_decay_est_sgm_2}. 
\begin{proposition}\label{prop_L2_R1}
	Let $\theta$ be a global classical solution of \eqref{eq:SIPM} with $N>0$ and $\sigma \in \mathscr{S}(\bbR)$. Assume that \eqref{ass} and \eqref{sol_bdd} are satisfied for any given $m \in \bbN$ with $m > 2$. Then, there exists a constant $C>0$ such that 
	\begin{equation}\label{R1_L2}
		(1+Nt) \| R_1 \tht \|_{H^m}^2 \leq C \| \tht_0 \|_{H^m}^2.
	\end{equation}
\end{proposition}
\begin{proof}
\Black{Recalling the $\theta$ equation in \eqref{eq:SIPM}} and multiplying it by $u_2$ and integrate over $\bbR^2$, we have
\begin{equation*}
    \frac{1}{2} \frac{\ud}{\ud t} \|R_1 \tht\|_{L^2}^2 \,\ud x = -\int_{\bbR^2} (u \cdot \nabla)\tht u_2 \,\ud x - \int_{\bbR^2} |u_2|^2(N+\sigma(x_2)) \,\ud x.
\end{equation*}
Since we can see
\begin{align*}
    \left| -\int_{\bbR^2} (u \cdot \nabla)\tht u_2 \,\ud x \right| &\leq \| u_1 \|_{L^2} \| \partial_1 \tht \|_{L^{\infty}} \| u_2 \|_{L^2} + \| u_2 \|_{L^2}^2 \| \partial_2 \tht \|_{L^{\infty}} \\
    &\leq \frac{C}{N} \| R_1\tht \|_{L^2}^2 \| R_1 \tht \|_{H^{m}}^2 + (\frac{N}{4}+ C\| \tht \|_{H^m}) \| R_1^2 \tht \|_{L^2}^2
\end{align*}
and $$- \int_{\bbR^2} |u_2|^2(N+\sigma(x_2)) \,\ud x \leq -(N - \| \sigma \|_{L^{\infty}}) \| R_1^2 \tht \|_{L^2}^2, $$
it holds that
\begin{equation}\label{L2_est_u}
\begin{aligned}
    \frac{1}{2} \frac{\ud}{\ud t} \| R_1 \tht \|_{L^2}^2 &\leq \frac{C}{N} \| R_1 \tht \|_{L^2}^2 \| R_1 \tht \|_{H^{m}}^2 - (\frac{3}{4}N - C \| \tht \|_{H^m} - \| \sigma \|_{L^{\infty}}) \| R_1^2 \tht \|_{L^2}^2 \\
    &\leq \frac{C}{N} \| R_1 \tht \|_{L^2}^2 \| R_1 \tht \|_{H^{m}}^2 - \frac{N}{2} \| R_1^2 \tht \|_{L^2}^2
\end{aligned}
\end{equation}
On the other hand, multiplying the $\tht$ equation in \eqref{eq:SIPM_g} by $\partial_1^2(-\Delta)^{m-1} \tht$ and integrating over $\bbR^2$, we can see
\begin{equation*}
    \Black{\frac{1}{2} \frac{\ud}{\ud t} \|R_1 \tht\|_{\dot{H}^m}^2} = -\sum_{|\alpha|=m-1} \int_{\bbR^2} \partial_1\partial^{\alpha} (u \cdot \nabla) \tht \partial_1\partial^{\alpha} \tht \,\ud x - \sum_{|\alpha|=m} \int_{\bbR^2} \partial^{\alpha}u_2(N+\sigma(x_2)) \partial^{\alpha}u_2 \,\ud x.
\end{equation*}
Integration by parts yields $$- \sum_{|\alpha|=m} \int_{\bbR^2} \partial^{\alpha}u_2(N+\sigma(x_2)) \partial^{\alpha}u_2 \,\ud x = -N \| R_1^2 \tht \|_{\dot{H}^m} + \| \sigma \|_{W^{m,\infty}} \| R_1^2 \tht \|_{H^m}^2.$$ For the remainder term, we consider $\partial^\alpha \neq \partial_2^{m-1}$ case first. From the divergence-free condition, it follows $$\left| -\int_{\bbR^2} \partial_1\partial^{\alpha} (u \cdot \nabla) \tht \partial_1\partial^{\alpha} \tht \,\ud x \right| \leq C \| R_1 \tht \|_{H^m}^2 \| R_1^2 \tht \|_{H^m}.$$ In the other case, we have
\begin{align*}
    \left| -\int_{\bbR^2} \partial_1\partial_2^{m-1} (u \cdot \nabla) \tht \partial_1\partial_2^{m-1} \tht \,\ud x \right| &= \left| -\int_{\bbR^2} \partial_2^{m} (u \cdot \nabla) \tht \partial_1^2\partial_2^{m-2} \tht \,\ud x \right| \\
    &\leq C\| R_1^2\tht \|_{H^m} \| R_1 \tht \|_{H^m}^2 + C \| R_1^2 \tht \|_{H^m}^2 \| \tht \|_{H^m}.
\end{align*}
Thus, we can have by Young's inequality that
\begin{equation*}
    \frac{1}{2} \frac{\ud}{\ud t} \|R_1 \tht\|_{\dot{H}^m}^2 \,\ud x \leq \frac{C}{N} \| R_1 \tht \|_{H^{m}}^4 + (C \| \tht \|_{H^m} + \| \sigma \|_{W^{m,\infty}}) \| R_1^2 \tht \|_{H^m}^2 - N \| R_1^2 \tht \|_{\dot{H}^m}^2.
\end{equation*}
Combining with \eqref{L2_est_u} gives $$\frac{\ud}{\ud t} \|R_1 \tht\|_{H^m}^2 \leq \frac{C}{N} \| R_1 \tht \|_{H^m}^4.$$ Multiplying both terms by $(1+Nt)$, we deduce
\begin{equation*}
    \frac{\ud}{\ud t} \left( (1+Nt) \| R_1 \tht \|_{H^m}^2 \right) \leq \frac{C}{N} (1+Nt) \| R_1 \tht \|_{H^m}^2 \| R_1 \tht \|_{H^{m}}^2 + N\| R_1 \tht \|_{H^m}^2.
\end{equation*}
By Gr\"onwall's inequality and \eqref{sol_bdd}, we obtain \eqref{R1_L2}. This completes the proof.
\end{proof}

\begin{proposition}\label{prop_L2_R2}
	Let $\theta$ be a global classical solution of \eqref{eq:SIPM} with $N>0$ and $\sigma \in \mathscr{S}(\bbR)$. Assume that \eqref{ass}, \eqref{sol_bdd}, and \eqref{sol_bdd1} are satisfied for any given $m \in \bbN$ with $m > 2$. Then, there exists a constant $C>0$ such that 
	\begin{equation}\label{R2_L2}
    (1+Nt)^2 \| \partial_1 R_1 \tht \|_{H^{m-1}}^2 \leq C \| \tht_0 \|_{H^m}^2.
    \end{equation}
\end{proposition}
\begin{proof}
Multiplying the $\tht$ equation in \eqref{eq:SIPM_g} by $-\partial_1^2 u_2$ and integrating over $\bbR^2$ gives that
\begin{equation*}
    \frac{1}{2} \frac{\ud}{\ud t} \int_{\bbR^2} | \partial_1 R_1 \tht|^2 \,\ud x = -\int_{\bbR^2} \partial_1 (u \cdot \nabla)\tht \partial_1 u_2 \,\ud x - \int_{\bbR^2} |\partial_1 u_2|^2(N+\sigma(x_2)) \,\ud x.
\end{equation*}
Since
\begin{gather*}
    \left| -\int_{\bbR^2} \partial_1 (u \cdot \nabla)\tht \partial_1 u_2 \,\ud x \right| \leq (\| \partial_1 u_1 \|_{L^2} \| \partial_1 \tht \|_{L^{\infty}} + \| u_1 \|_{L^{\infty}} \| \partial_1^2 \tht \|_{L^2}) \| \partial_1 u_2 \|_{L^2} \\
    + \| \nabla \tht \|_{L^{\infty}} \| \partial_1 u_2 \|_{L^2}^2 + \| u_2 \|_{L^{\infty}} \| \partial_1 \partial_2 \tht \|_{L^2} \| \partial_1 u_2 \|_{L^2} \\
    \leq C \| \partial_1 R_1 \tht \|_{H^{m-1}} \| R_1 \tht \|_{H^m} \| \partial_1 u_2 \|_{L^2} + C \| \tht \|_{H^m} \| \partial_1 u_2 \|_{L^2}^2, 
\end{gather*}
it follows from \eqref{ass} and Young's inequality that
\begin{equation}\label{R2_L2_est}
    \begin{gathered}
        \frac{1}{2} \frac{\ud}{\ud t} \int_{\bbR^2} | \partial_1 R_1 \tht|^2 \,\ud x \leq -(N - C(\| \sigma \|_{L^{\infty}} + \| \tht \|_{H^m})) \| \partial_1 u_2 \|_{L^2}^2 + C \| \partial_1 R_1 \tht \|_{H^{m-1}} \| R_1 \tht \|_{H^m} \| \partial_1 u_2 \|_{L^2} \\
        \leq -\frac{N}{2} \| \partial_1 u_2 \|_{L^2}^2 \Black{+ \frac{C}{N} \| \partial_1 R_1 \tht \|_{H^{m-1}}^2 \| R_1 \tht \|_{H^m}^2.}
    \end{gathered}
\end{equation}

On the other hand, multiplying the $\tht$ equation in \eqref{eq:SIPM_g} by $\partial_1^4(-\Delta)^{m-2} \tht$ and integrating over $\bbR^2$ gives that
\begin{equation*}
    \frac{1}{2} \frac{\ud}{\ud t} \| \partial_1 R_1 \tht \|_{\dot{H}^{m-1}}^2 = - \sum_{|\alpha| = m-2} \int_{\bbR^2} \partial^{\alpha} \partial_1^2 (u \cdot \nabla) \tht \partial^{\alpha} \partial_1^2 \tht \,\ud x - \sum_{|\alpha| = m-1} \int_{\bbR^2} \partial^{\alpha} \partial_1 (u_2 (N+\sigma(x_2))) \partial^{\alpha} \partial_1 u_2 \,\ud x.
\end{equation*}
It is clear
\begin{align*}
    - \sum_{|\alpha| = m-1} \int_{\bbR^2} \partial^{\alpha} \partial_1 (u_2 (N+\sigma(x_2))) \partial^{\alpha} \partial_1 u_2 \,\ud x \leq - N\| \partial_1 R_1^2 \tht \|_{\dot{H}^{m-1}}^2 + C\| \sigma \|_{W^{m-1,\infty}} \| \partial_1 R_1^2 \tht \|_{H^{m-1}}^2.
\end{align*}
Note that
\begin{gather*}
    - \sum_{|\alpha| = m-2} \int_{\bbR^2} \partial^{\alpha} \partial_1^2 (u \cdot \nabla) \tht \partial^{\alpha} \partial_1^2 \tht \,\ud x \\
    = - \sum_{|\alpha| = m-2} \int_{\bbR^2} \partial^{\alpha} \partial_1 (\partial_1 u \cdot \nabla) \tht \partial^{\alpha} \partial_1^2 \tht \,\ud x - \sum_{|\alpha| = m-2} \int_{\bbR^2} \partial^{\alpha} (u \cdot \nabla) \partial_1^2 \tht \partial^{\alpha} \partial_1^2 \tht \,\ud x .
\end{gather*}
We can estimate the first integral on the right-hand side as follows:
\begin{gather*}
    \left| - \sum_{|\alpha| = m-2} \int_{\bbR^2} \partial^{\alpha} \partial_1 (\partial_1 u \cdot \nabla) \tht \partial^{\alpha} \partial_1^2 \tht \,\ud x \right| = \left| \sum_{|\alpha| = m-2} \int_{\bbR^2} \nabla \partial^{\alpha} (\partial_1 u \cdot \nabla) \tht \cdot \nabla \partial^{\alpha} \partial_1^3 (-\Delta)^{-1}\tht \,\ud x \right| \\
    \leq C \| \partial_1 u_1 \|_{H^{m-1}} \| \partial_1 \tht \|_{H^{m-1}} \| \partial_1 R_1^2 \tht \|_{H^{m-1}} + C\| \partial_1 u_2 \|_{H^{m-1}} \| \tht \|_{H^m} \| \partial_1 R_1^2 \tht \|_{H^{m-1}} \\
    \leq C\| \partial_1 R_1 \tht \|_{H^{m-1}} \| R_1 \tht \|_{H^m} \| \partial_1 R_1^2 \tht \|_{H^{m-1}} + C \| \tht \|_{H^m} \| \partial_1 R_1^2 \tht \|_{H^{m-1}}^2.
\end{gather*}
On the other hand, if $\partial^{\alpha} \neq \partial_2^{m-2}$, we can show the remainder term to satisfy $$\left| - \int_{\bbR^2} \partial^{\alpha} (u \cdot \nabla) \partial_1^2 \tht \partial^{\alpha} \partial_1^2 \tht \,\ud x \right| \leq C \| R_1 \tht \|_{H^{m}} \| \partial_1 R_1^2 \tht \|_{H^{m-1}} \| \partial_1 R_1 \tht \|_{H^{m-1}}.$$ In the other case, $(u \cdot \nabla)\tht = u_1 \partial_1 \tht + u_2 \partial_2 \tht$ and the cancellation property give
\begin{gather*}
\left| \int_{\bbR^2} \partial_2^{m-2} (u \cdot \nabla) \partial_1^2 \tht \partial_2^{m-2} \partial_1^2 \tht \,\ud x \right| \\
\leq C \| R_1 \tht \|_{H^{m}} \| \partial_1 R_1^2 \tht \|_{H^{m-1}} \| \partial_1 R_1 \tht \|_{H^{m-1}} + \left| \int_{\bbR^2} \partial_2^{m-3} (\partial_2 u_2 \partial_2 \partial_1^2 \tht) \partial_2^{m-2} \partial_1^2 \tht \,\ud x \right|.
\end{gather*}
We note that
\begin{gather*}
\left| \int_{\bbR^2} \partial_2^{m-3} (\partial_2 u_2 \partial_2 \partial_1^2 \tht) \partial_2^{m-2} \partial_1^2 \tht \,\ud x \right| \leq \| \partial_2 u_2 \|_{L^{\infty}} \| \partial_1 R_1 \tht \|_{H^{m-1}}^2 + \left| \int_{\bbR^2} \partial_2^{m-4} (\partial_2^2 u_2 \partial_2 \partial_1^2 \tht) \partial_2^{m-2} \partial_1^2 \tht \,\ud x \right|,
\end{gather*}
where the last integral vanishes when $m=3$. Using the integration by parts twice with the divergence-free condition, we can have 
\begin{align*}
    \left| \int_{\bbR^2} \partial_2^{m-4} (\partial_2^2 u_2 \partial_2 \partial_1^2 \tht) \partial_2^{m-2} \partial_1^2 \tht \,\ud x \right| &= \left| -\int_{\bbR^2} \partial_2^{m-4} (\partial_2 \partial_1 u_1 \partial_2 \partial_1^2 \tht) \partial_2^{m-2} \partial_1^2 \tht \,\ud x \right| \\
    &\leq C \| R_1 \tht \|_{H^{m}} \| \partial_1 R_1^2 \tht \|_{H^{m-1}} \| \partial_1 R_1 \tht \|_{H^{m-1}}.
\end{align*} Thus,
\begin{gather*}
    \left| - \sum_{|\alpha| = m-2} \int_{\bbR^2} \partial^{\alpha} (u \cdot \nabla) \partial_1^2 \tht \partial^{\alpha} \partial_1^2 \tht \,\ud x \right| \\
    \leq C\| R_1 \tht \|_{H^m} \| \partial_1 R_1^2 \tht \|_{H^{m-1}} \| \partial_1 R_1 \tht \|_{H^{m-1}} + \| \partial_2 u_2 \|_{L^{\infty}} \| \partial_1 R_1 \tht \|_{H^{m-1}}^2.
\end{gather*}
Therefore, we deduce by Young's inequality that
\begin{align*}
    \frac{1}{2} \frac{\ud}{\ud t} \| \partial_1 R_1 \tht \|_{\dot{H}^{m-1}}^2 &\leq - N \| \partial_1 R_1^2 \tht \|_{\dot{H}^{m-1}}^2 + C(\| \sigma \|_{W^{m,\infty}} + \| \tht \|_{H^m}) \| \partial_1 R_1^2 \tht \|_{H^{m-1}}^2 \\
    &\hphantom{\qquad\qquad}+ C\| \partial_1 R_1 \tht \|_{H^{m-1}} \| R_1 \tht \|_{H^m} \| \partial_1 R_1^2 \tht \|_{H^{m-1}} + \| \partial_2 u_2 \|_{L^{\infty}} \| \partial_1 R_1 \tht \|_{H^{m-1}}^2.
\end{align*}
Combining with \eqref{R2_L2_est}, we have
\begin{equation*}
        \frac{\ud}{\ud t} \| \partial_1 R_1 \tht \|_{H^{m-1}}^2 \leq -N \| \partial_1 R_1^2 \tht \|_{H^{m-1}}^2 + C \| \partial_1 R_1 \tht \|_{H^{m-1}}^2 (\frac{1}{N}\| R_1 \tht \|_{H^m}^2 + \| \partial_2 u_2 \|_{L^{\infty}}).
\end{equation*}
Note that for any $M \geq 1$,
\Black{\begin{equation}\label{ineq_balancing}\begin{split}
    \frac{1}{M} \| \partial_1 R_1 \tht \|_{\dot{H}^{m-1}}^2 - \| \partial_1 R_1^2 \tht \|_{\dot{H}^{m-1}}^2 &= \int_{\bbR^2} \xi_1^4 |\xi|^{2(m-2)} \left( \frac{1}{M} - \frac{\xi_1^2}{|\xi|^2} \right) |\mathscr{F} \tht|^2 \,\ud \xi \\
    &\leq \frac{1}{M} \int_{M\xi_1^2 \leq |\xi|^2} \xi_1^4 |\xi|^{2(m-2)} |\mathscr{F} \tht|^2 \,\ud \xi \\
    &\leq \frac{1}{M^2} \int_{\bbR^2} \xi_1^2 |\xi|^{2(m-1)} |\mathscr{F} \tht|^2 \,\ud \xi \\
    &= \frac{1}{M^2} \| R_1 \tht \|_{\dot{H}^m}^2.
    \end{split}
\end{equation}}
Thus, there holds
\begin{gather*}
        \frac{\ud}{\ud t} \| \partial_1 R_1 \tht \|_{H^{m-1}}^2 + \frac{N}{M} \| \partial_1 R_1 \tht \|_{H^{m-1}}^2 \leq \frac{N}{M^2} \| R_1 \tht \|_{H^m}^2 + C \| \partial_1 R_1 \tht \|_{H^{m-1}}^2 (\frac{1}{N}\| R_1 \tht \|_{H^m}^2 + \| \partial_2 u_2 \|_{L^{\infty}}).
\end{gather*}
Now, we take $M = (1+\frac{Nt}{2})^2$ and multiplying both terms by $M^2$. Then,
\begin{align*}
    \frac{\ud}{\ud t} \left( (1+\frac{Nt}{2})^2 \| \partial_1 R_1 \tht \|_{H^{m-1}}^2 \right) \leq N \| R_1 \tht \|_{H^m}^2 + C (1+\frac{Nt}{2})^2 \| \partial_1 R_1 \tht \|_{H^{m-1}}^2 (\frac{1}{N}\| R_1 \tht \|_{H^m}^2 + \| \partial_2 u_2 \|_{L^{\infty}}).
\end{align*}
By Gr\"{o}nwall's inequality, it holds that
\begin{gather*}
    (1+\frac{Nt}{2})^2 \| \partial_1 R_1 \tht(t) \|_{H^{m-1}}^2 \\
    \leq \left(\| \tht_0 \|_{H^m}^2 + N \int_0^\infty \| R_1 \tht \|_{H^m}^2 \,\ud t \right) \exp\left( \frac{C}{N}\int_0^\infty \| R_1 \tht \|_{H^{m}}^2 \,\ud t + C\int_0^{\infty} \| \partial_2 v_2 \|_{L^{\infty}} \,\ud t \right)
\end{gather*}
for all $t>0$. Thus, we obtain \eqref{R2_L2}. This completes the proof.
\end{proof}

\section{Proof of Theorem~\ref{thm:stb_T}}\label{sec_T}
In this section, we write the Fourier coefficient as
\begin{equation*}
    \mathscr{F} f(n) = \langle f, e^{2\pi i n \cdot x} \rangle = \int_{\bbT^2} e^{-2\pi i n \cdot x} f(x) \,\ud x.
\end{equation*}
We can express our solution $(\tht,u,P)$ via the Fourier series expansion. Since $\tht$ is periodic in $x$,
\begin{equation*}
    \tht = \sum_{n \in \bbZ^2} \mathscr{F}\tht(n) e^{2\pi i n \cdot x}.
\end{equation*}
Taking divergence operator to the $u$ equations in \eqref{eq:SIPM_T}, we have $-\Delta P = -\partial_2 \tht$, thus, $$P = -\sum_{n \neq 0} \frac{2\pi i n_2}{4\pi^2|n|^2}\mathscr{F} \tht (n) e^{2\pi i n \cdot x} + \mathscr{F}P(0).$$ This formula implies
\begin{equation}\label{sf_u}
    u_1 = -\sum_{n \neq 0} \frac{n_1n_2}{|n|^2}\mathscr{F}\tht(n) e^{2\pi i n \cdot x}, \qquad  u_2 = \sum_{n \neq 0} \frac{n_1^2}{|n|^2}\mathscr{F}\tht(n) e^{2\pi i n \cdot x} + \mathscr{F}\tht(0).
\end{equation}
We introduce simple observations:
\begin{lemma}\label{lem_avg}
     Let $(\tht,u,P)$ be a classical solution to \eqref{eq:SIPM_T}. Then, we have
     \begin{equation}\label{avg_est_1}
         \int_{\bbT^2} \tht(t,x) \,\ud x = \int_{\bbT} u_2(t,x) \,\ud x_1 = e^{-Nt} \int_{\bbT^2} \tht_0 \,\ud x
     \end{equation}
     and
     \begin{equation}\label{avg_est_2}
         \int_{\bbT} u_1(t,x) \,\ud x_1 = 0.
     \end{equation}
\end{lemma}
\begin{proof}
    From the $u_1$ equation in \eqref{eq:SIPM_T} and the periodicity of $P$, we obtain \eqref{avg_est_2}. From the $\tht$ equation in \eqref{eq:SIPM_T}, we have
    \begin{equation*}
        \frac{\ud}{\ud t} \int_{\bbT^2} \tht \,\ud x + \int_{\bbT^2} (u \cdot \nabla) \tht \,\ud x = -N \int_{\bbT^2} u_2 \,\ud x.
    \end{equation*}
    Integration by parts and divergence-free condition shows that
    \begin{equation*}
        \int_{\bbT^2} (u \cdot \nabla) \tht \,\ud x = 0.
    \end{equation*}
    Since $u_2$ equation in \eqref{eq:SIPM_T} implies
    \begin{equation*}
        \int_{\bbT} u_2 \,\ud x_2 = \int_{\bbT} \tht \,\ud x_2, \qquad \forall x_1 \in \bbT,
    \end{equation*}
    we have
    \begin{equation*}
        \Black{\frac{\ud}{\ud t} \int_{\bbT^2} \tht \,\ud t + N\int_{\bbT^2} \tht \,\ud x = 0.}
    \end{equation*}
    By Gr\"{o}nwall's inequality, it follows $$\int_{\bbT^2} \tht \,\ud x = e^{-Nt} \int_{\bbT^2} \tht_0 \,\ud x.$$
    The divergence-free condition and the periodicity of $u_1$ imply $$\partial_2 \int_{\bbT} u_2 \,\ud x_1 = 0, \qquad \forall x_2 \in \bbT,$$ hence, $$\int_{\bbT} u_2 \,\ud x_1 = \int_{\bbT^2} u_2 \,\ud x = e^{-Nt} \int_{\bbT^2} \tht_0 \,\ud x.$$ This completes the proof.
\end{proof}
For simplicity, we let $R_1 : L^2(\bbT^2) \to L^2(\bbT^2)$ be an operator defined by
\begin{equation*}
    R_1 f(x) := -\sum_{n \neq 0} i\frac{n_1}{|n|} \mathscr{F}f(n) e^{2\pi i n \cdot x}.
\end{equation*}

\subsection{Energy inequality}
We introduce an energy estimate. Since the proof of the energy estimate in the periodic domain is parallel to the confined case, \Black{we only prove the latter} (see Proposition~\ref{prop_energy_O}). We remark that the condition $m \in \bbN$ in the following proposition can be dropped via a similar estimate with \eqref{commutator} on the domain $\bbT^2$. In this paper, we do not show the details.
\begin{proposition}\label{prop_energy_T}
Let $\theta$ be a classical solution of \eqref{eq:SIPM_T} with $N \in \bbR$. Then for any given $m \in \bbN$ with $m > 2$, there exists a constant $C=C(m)>0$ such that
\begin{equation}\label{energy_T}
\begin{gathered}
    \frac{1}{2}\sup_{t \in [0,T]} \|\theta\|_{H^{m}}^2 + N \int_0^T \|R_1\theta\|_{H^{m}}^{2} \,\ud t \\
    \leq \frac{1}{2} \| \tht_0 \|_{H^{m}}^2 + C\sup_{t \in [0,T]} \| \theta \|_{H^m} \int_0^T \|R_1 \theta \|_{H^m}^2 \,\ud t + C \sup_{t \in [0,T]} \| \theta \|_{H^m}^2 \sum_{n \in \mathbb{Z}^2} \int_0^T |\mathscr{F}\nabla u_2(n)| \, \ud t.
\end{gathered}
\end{equation}
\end{proposition}

\subsection{Proof of the global existence of solutions}

To obtain the upper bound of the key integral in \eqref{energy_T}, we need the following lemma.
\begin{lemma}
     Let $\tht$ be a global classical solution of \eqref{eq:SIPM_T}. Then, for any given $m > 3$ and $s \in [0,m-2)$, there exists a constant $C>0$ such that
        \begin{equation}\label{key_est_T} \begin{gathered}
            \sum_{n \in \mathbb{Z}^2} \int_0^T (1+|n|^2)^{\frac s2} |\mathscr{F}u_2(n)| \, \ud t \leq \frac{1}{N}\sum_{n \in \bbZ^2} (1+|n|^2)^{\frac s2} |\mathscr{F} \theta_0(n) | \\
    		+ \frac{C}{N}\int_0^T \| R_1 \theta \|_{H^{m}}^2 \, \ud t + \frac{C}{N} \sup_{t \in [0,T]} \| \theta \|_{H^m} \sum_{n \in \mathbb{Z}^2} \int_0^T (1+|n|^2)^{\frac s2} |\mathscr{F}u_2(n)| \, \ud t
		\end{gathered}
	\end{equation}
	for all $T>0$.
\end{lemma}
\begin{proof}
    Recalling \eqref{sf_u} gives
	\begin{gather*}
		\sum_{n \in \mathbb{Z}^2} \int_0^T (1+|n|^2)^{\frac s2} |\mathscr{F}u_2(n)| \, \ud t \leq \Black{\int_0^T (1+|n|^2)^{\frac{s}{2}}|\mathscr{F} \tht(0)| \,\ud t} + \sum_{n \neq 0} \int_0^T (1+|n|^2)^{\frac s2} \frac{n_1^2}{|n|^2}|\mathscr{F}\tht(n)| \, \ud t.
	\end{gather*}
	We see from \eqref{avg_est_1} that $\mathscr{F}\tht(0) = e^{-Nt} \mathscr{F} \tht_0(0)$ for all $t \geq 0$, thus,
    \begin{equation*}
        \int_0^T (1+|n|^2)^{\frac{s}{2}}|\mathscr{F}\tht(0)| \,\ud t \leq \frac{1}{N} (1+|n|^2)^{\frac{s}{2}}|\mathscr{F}\tht_0(0)|.
    \end{equation*}
	We note that
	\begin{gather*}
		\sum_{n \neq 0} \int_0^T (1+|n|^2)^{\frac s2} \frac{n_1^2}{|n|^2}|\mathscr{F}\tht(n)| \, \ud t \\
		\leq \frac{1}{N} \sum_{n \neq 0} \int_0^T (1+|n|^2)^{\frac s2} N\frac {n_1^2}{|n|^2} e^{- N\frac {n_1^2}{|n|^2} t} |\mathscr{F} \theta_0 (n)| \, \ud t, \\
		+ \frac{1}{N} \sum_{n \neq 0} \int_0^T \int_0^t (1+|n|^2)^{\frac{s}{2}} N\frac {n_1^2}{|n|^2} e^{- N\frac {n_1^2}{|n|^2}(t-\tau)} |\mathscr{F} u \cdot \nabla \theta (n,\tau)| \,\ud \tau \ud t.
	\end{gather*}
	By Fubini's theorem, it is clear that
	\begin{equation*}
		\frac{1}{N} \sum_{n \neq 0} \int_0^T (1+|n|^2)^{\frac s2} N\frac {n_1^2}{|n|^2} e^{- N\frac {n_1^2}{|n|^2} t} |\mathscr{F} \theta_0 (n)| \, \ud t \leq \frac{1}{N}\sum_{n \neq 0} (1+|n|^2)^{\frac s2} |\mathscr{F} \theta_0(n) |.
	\end{equation*}
	With \eqref{conv_ineq_T2} and \eqref{decom_est}, we can have
	\begin{align*}
		&\frac{1}{N} \sum_{n \neq 0} \int_0^T \int_0^t (1+|n|^2)^{\frac{s}{2}} N\frac {n_1^2}{|n|^2} e^{- N\frac {n_1^2}{|n|^2}(t-\tau)} |\mathscr{F} u \cdot \nabla \theta (n,\tau)| \,\ud \tau \ud t \\
		&\hphantom{\qquad\qquad}\leq \frac{1}{N}\int_0^T \sum_{n \neq 0} (1+|n|^2)^{\frac s2} |\mathscr{F} u \cdot \nabla \theta | \, \ud t \\
		&\hphantom{\qquad\qquad}\leq \frac{C}{N}\int_0^T \| (1+|n|^2)^{\frac s2} \mathscr{F} u_1 \|_{l^1} \| (1+|n|^2)^{\frac s2} \mathscr{F} \partial_1 \theta \|_{l^1}\, \ud t \\
		&\hphantom{\qquad\qquad\qquad\qquad} + \frac{C}{N}\int_0^T \| (1+|n|^2)^{\frac s2} \mathscr{F} u_2 \|_{l^1} \| (1+|n|^2)^{\frac s2} \mathscr{F} \partial_2 \theta \|_{l^1}\, \ud t \\
		&\hphantom{\qquad\qquad}\leq \frac{C}{N}\int_0^T \| u \|_{H^m} \| \partial_1 \theta \|_{H^{m-1}} \, \ud t + \frac{C}{N}\int_0^T \| (1+|n|^2)^{\frac s2} \mathscr{F} u_2 \|_{l^1} \| \theta \|_{H^m} \,\ud t \\
		&\hphantom{\qquad\qquad}\leq \frac{C}{N}\int_0^T \| R_1 \theta \|_{H^{m}}^2 \, \ud t + \frac{C}{N}\sup_{t \in [0,T]} \| \theta \|_{H^m} \int_0^T \| (1+|n|^2)^{\frac s2} \mathscr{F} u_2 \|_{l^1} \, \ud t.
	\end{align*}
	\Black{Note that $s<m-2$ is used in the third inequality to obtain $$ \| (1+|n|^2)^{\frac s2} \mathscr{F} \partial_1 \theta \|_{l^1} \leq C\| \partial_1 \theta \|_{H^{m-1}}, \qquad \| (1+|n|^2)^{\frac s2} \mathscr{F} \partial_2 \theta \|_{l^1} \leq C\| \theta \|_{H^m}.$$} Combining the above estimates, we obtain \eqref{key_est_T}. This completes the proof.
\end{proof}

Now we are ready to prove the global existence part of Theorem~\ref{thm:stb_T}. Let $\theta$ be the local-in-time solution to \eqref{eq:SIPM_T} with the maximal time of existence $T^* \in (0,\infty]$. Let $T \in (0,T^*)$ with \eqref{con}. Then, from \eqref{key_est_T} and \eqref{ass}, we have
\begin{equation}\label{key1_T}
	\int_0^T \sum_{n \in \mathbb{Z}^2} (1+|n|^2)^{\frac s2} |\mathscr{F}u_2(n)| \, \ud t \leq \frac{1}{N}\sum_{n \in \bbZ^2} (1+|n|^2)^{\frac s2} |\mathscr{F} \theta_0 | + \frac{C}{N}\int_0^T \| R_1 \theta \|_{H^{m}}^2 \, \ud t \leq \frac{C}{N} \| \tht_0 \|_{H^m}.
\end{equation}
Applying \eqref{con} and this inequality to \eqref{energy_T}, we have
\begin{gather*}
    \frac{1}{2}\sup_{t \in [0,T]} \|\theta\|_{H^{m}}^2 + N \int_0^T \|R_1\theta\|_{H^{m}}^{2} \,\ud t \leq \frac{1}{2} \| \tht_0 \|_{H^{m}}^2 + \frac{C}{N} \| \tht_0 \|_{H^{m}}^3.
\end{gather*}
Taking $C_0$ in Theorem~\ref{thm:stb_T} large enough, we obtain
\begin{gather*}
    \frac{1}{2}\sup_{t \in [0,T]} \|\theta\|_{H^{m}}^2 + N \int_0^T \|R_1\theta\|_{H^{m}}^{2} \,\ud t \leq \| \tht_0 \|_{H^{m}}^2, \qquad T \in (0,T^*),
\end{gather*}
which holds \eqref{con}. Hence, we can deduce $T^*=\infty$ and \eqref{sol_bdd} follows. This completes the proof. \qed

\subsection{Proof of the decay estimates}

In this section, we prove \eqref{tem_decay_est_T}. In the case of periodic domain, $\tht$ may not decay over time. This makes a difficulty to obtain a temporal decay estimates compared to the whole domain case. To overcome this, we employ the brilliant idea in \cite{Elgindi}, which makes being able to obtain \eqref{tem_decay_est_T} with difference between optimal decay rates.  We use the notation $$\bar{f} := f - \int_{\bbT} f(x) \,\ud x_1 = \sum_{n_1 \neq 0} \mathscr{F}f(n) e^{2\pi in \cdot x}, \qquad f \in L^2(\bbT^2).$$ We show two Propositions.
\begin{proposition}\label{prop_tem_1}
    Let $\theta$ be a global classical solution of \eqref{eq:SIPM_T} with $N>0$. Suppose that \eqref{sol_bdd} is satisfied for any given $m \in \bbN$ with $m > 2$. Then, there exists a constant $C>0$ such that
    \begin{equation}\label{tem_decay_est_T_1}
        \| \bar{\tht}(t) \|_{L^2} \leq C(1+Nt)^{-\frac{m}{2}} \| \tht_0 \|_{H^m}
    \end{equation}
    and
    \begin{equation}\label{tem_decay_est_T_4}
        \| u(t) \|_{L^2} \leq C(1+Nt)^{-(\frac{1}{2}+\frac{m}{2})} \| \tht_0 \|_{H^m}
    \end{equation}  
    for all $t>0$.
\end{proposition}
\begin{remark}
    From \eqref{tem_decay_est_T_1} and \eqref{sol_bdd}, we obtain
    \begin{equation}\label{tem_est_tht}
        \| \bar{\tht}(t) \|_{H^s} \leq C(1+Nt)^{-\frac{m-s}{2}} \| \tht_0 \|_{H^m}.
    \end{equation}
\end{remark}
\begin{proof}
Multiplying $\tht$ equation in \eqref{eq:SIPM_T} by $\bar{\tht}$ and integrating over $\bbT^2$ gives
\begin{equation*}
    \frac{1}{2} \frac{\ud}{\ud t} \int_{\bbT^2} |\bar{\tht}|^2 \,\ud x = -\int_{\bbT^2} (u \cdot \nabla) \tht \bar{\tht} \,\ud x -N \int_{\bbT^2} u_2 \bar{\tht} \,\ud x.
\end{equation*}
Integration by parts and the divergence free condition yield $$-\int_{\bbT^2} (u \cdot \nabla) \tht \bar{\tht} \,\ud x = -\int_{\bbT^2} u_2 \left( \int_{\bbT} \partial_2 \tht(x) \,\ud x_1 \right) \bar{\tht} \,\ud x = -\int_{\bbT^2} \bar{u}_2 \left( \int_{\bbT} \partial_2 \tht(x) \,\ud x_1 \right) \bar{\tht} \,\ud x.$$ We used \eqref{sf_u} in the last inequality. With $$-N \int_{\bbT^2} u_2 \bar{\tht} \,\ud x = -N \int_{\bbT^2} |R_1 \bar{\tht}|^2 \,\ud x,$$ we have
\begin{equation*}
    \frac{1}{2} \frac{\ud}{\ud t} \int_{\bbT^2} |\bar{\tht}|^2 \,\ud x = -N \int_{\bbT^2} |R_1 \bar{\tht}|^2 \,\ud x - \int_{\bbT^2} \bar{u}_2 \left( \int_{\bbT} \partial_2 \tht(x) \,\ud x_1 \right) \bar{\tht} \,\ud x.
\end{equation*}
Using that $\bar{\tht} = -\partial_1^2(-\Delta)^{-1} \bar{\tht} - \partial_2^2(-\Delta)^{-1} \bar{\tht}$, we can write
\begin{gather*}
    - \int_{\bbT^2} \bar{u}_2 \left( \int_{\bbT} \partial_2 \tht(x) \,\ud x_1 \right) \bar{\tht} \,\ud x \\
    = -\int_{\bbT^2} |\partial_1^2 (-\Delta)^{-1} \bar{\tht}|^2 \left( \int_{\bbT} \partial_2 \tht(x) \,\ud x_1 \right) \,\ud x - \int_{\bbT^2} \partial_1^2 (-\Delta)^{-1} \tht \left( \int_{\bbT} \partial_2 \tht(x) \,\ud x_1 \right) \partial_2^2 (-\Delta)^{-1} \bar{\tht} \,\ud x \\
    \leq -\int_{\bbT^2} |R_1 \tht|^2 \left( \int_{\bbT} \partial_2 \tht(x) \,\ud x_1 \right) \,\ud x + \int_{\bbT^2} |\Lambda^{-1} R_1 \tht| |R_1 \tht| \left( \int_{\bbT} |\partial_2 \tht(x)| \,\ud x_1 + \int_{\bbT} |\partial_2^2 \tht(x)| \,\ud x_1 \right) \,\ud x. 
\end{gather*}
By $$\|\Lambda^{-1}R_1 \bar{\tht}\|_{L^2_{x_1}L^{\infty}_{x_2}} \leq \| \frac{1}{|(1,n_2)|} \mathscr{F} R_1 \bar{\tht}(n) \|_{l^2_{n_1}l^1_{n_2}} \leq C \| R_1 \bar{\tht} \|_{L^2}$$ and \begin{equation*}
    \left\| \int_{\bbT} |\partial_2 \tht(x)| \,\ud x_1 \right\|_{L^2_{x_2}} \leq \| \tht \|_{H^1}, \qquad \left\| \int_{\bbT} |\partial_2 ^2\tht(x)| \,\ud x_1 \right\|_{L^2_{x_2}} \leq \| \tht \|_{H^2},
\end{equation*} we have
\begin{gather*}
    \frac{1}{2} \frac{\ud}{\ud t} \int_{\bbT^2} |\bar{\tht}|^2 \,\ud x \leq -N(1-CN^{-1}\| \tht \|_{H^{m}}) \|R_1 \bar{\tht}\|_{L^2}^2 \leq -\frac{N}{2} \|R_1 \bar{\tht}\|_{L^2}^2. 
\end{gather*}
As estimating \eqref{ineq_balancing}, we can see 
\begin{equation*}\begin{split}
    \frac{1}{M} \| \bar{\tht} \|_{L^2}^2 - \| R_1 \bar{\tht} \|_{L^2}^2 &= \sum_{n_1 \neq 0} \left( \frac{1}{M} - \frac{n_1^2}{|n|^2} \right) |\mathscr{F} \tht(n)|^2 \\
    &\leq \frac{1}{M} \sum_{\substack{Mn_1^2 \leq |n|^2 \\ n_1 \neq 0}} |\mathscr{F} \tht(n)|^2 \\
    &\leq \frac{1}{M^m} \sum_{\substack{Mn_1^2 \leq |n|^2 \\ n_1 \neq 0}} \left(\frac{|n|^{2}}{n_1^{2}}\right)^{m-1} |\mathscr{F} \tht(n)|^2 \\
    &\leq \frac{1}{M^m} \| \partial_1 \tht \|_{H^{m-1}}^2.
    \end{split}
\end{equation*}
This implies
\begin{equation*}
    \frac{\ud}{\ud t} \| \bar{\tht} \|_{L^2}^2 \leq - \frac{N}{M} \| \bar{\tht} \|_{L^2}^2 + \frac{CN}{M^m} \| \partial_1 \tht \|_{H^{m-1}}^2.
\end{equation*}
We take $M>0$ with $M = 1+\frac{Nt}{m}$. Then, by Gr\"onwall's inequality, we obtain \eqref{tem_decay_est_T_1}.

Next, we show \eqref{tem_decay_est_T_4}. Multiplying the $\tht$ equation in \eqref{eq:SIPM_T} by $R_1^2\tht$ and integrating over $\bbT^2$. Then, it follows
\begin{align*}
    \frac{1}{2} \frac{\ud}{\ud t} \int_{\bbT^2} |R_1\tht|^2 \,\ud x &=- \int_{\bbT^2} (u \cdot \nabla)\tht R_1^2 \tht \,\ud x -N \| R_1^2 \tht \|_{L^2}^2 \\
    &\leq \| u_1 \|_{L^2} \| \partial_1 \tht \|_{L^{\infty}} \| R_1^2 \tht \|_{L^2} + \| u_2 \|_{L^2} \| \partial_2 \tht \|_{L^{\infty}} \| R_1^2 \tht \|_{L^2} -N \| R_1^2 \tht \|_{L^2}^2.
\end{align*}
Since \eqref{avg_est_1} gives 
\begin{equation}\label{u2_est}
\| u_2 \|_{L^2}^2 = \| \bar{u}_2 \|_{L^2}^2 + e^{-2Nt} \left| \int_{\bbT^2} \tht_0 \,\ud x \right|^2 \leq \| R_1^2 \tht \|_{L^2}^2 + e^{-2Nt} \| \tht_0 \|_{L^2}^2,
\end{equation}
we can see with Young's inequality that
\begin{align*}
    \frac{1}{2} \frac{\ud}{\ud t} \int_{\bbT^2} |R_1\tht|^2 \,\ud x &\leq \frac{C}{N} \| R_1 \tht \|_{L^2}^2 \| R_1 \tht \|_{H^m}^2 + \frac{C}{N} e^{-2Nt} \| \tht_0 \|_{L^2}^2 \| \tht \|_{H^m}^2 -(\frac{3}{4}N - C\| \tht \|_{H^m}) \| R_1^2 \tht \|_{L^2}^2 \\
    &\leq \frac{C}{N} \| R_1 \tht \|_{L^2}^2 \| R_1 \tht \|_{H^m}^2 + \frac{C}{N} e^{-2Nt} \| \tht_0 \|_{L^2}^2 \| \tht \|_{H^m}^2 - \frac{N}{2} \| R_1^2 \tht \|_{L^2}^2.
\end{align*}
We note for $M \geq 1$ that
\begin{align*}
    \frac{1}{M} \| R_1 \tht \|_{L^2}^2 - \|R_1^2 \tht\|_{L^2}^2 &= \sum_{n_1 \neq 0} \left( \frac{1}{M} - \frac{n_1^2}{|n|^2} \right)|\mathscr{F} R_1 \tht(n)|^2 \\
    &\leq \frac{1}{M}\sum_{\substack{Mn_1^2 \leq |n|^2 \\ n_1 \neq 0}} |\mathscr{F} R_1 \tht(n)|^2 \leq \frac{1}{M^{m+1}} \| R_1 \tht \|_{H^m}^2.
\end{align*}
Thus, we have $$\frac{\ud}{\ud t} \| R_1\tht \|_{L^2}^2  \leq \frac{C}{N} \| R_1 \tht \|_{L^2}^2 \| R_1 \tht \|_{H^m}^2 + \frac{C}{N} e^{-2Nt} \| \tht_0 \|_{L^2}^2 \| \tht \|_{H^m}^2 - \frac{N}{M} \| R_1 \tht \|_{L^2}^2 + \frac{CN}{M^{m+1}} \| R_1 \tht \|_{H^m}^2.$$ Taking $M = 1+\frac{Nt}{m+1}$ and multiplying both sides by $M^{m+1}$ gives that
\begin{gather*}
    \frac{\ud}{\ud t} \left( (1+\frac{Nt}{m+1})^{m+1} \| R_1 \tht \|_{L^2}^2 \right) \\
    \leq \frac{C}{N} (1+\frac{Nt}{m+1})^{m+1} \| R_1 \tht \|_{L^2}^2 \| R_1 \tht \|_{H^m}^2 + \frac{C}{N} e^{-Nt} \| \tht_0 \|_{L^2}^2 \| \tht \|_{H^m}^2 + CN \| R_1 \tht \|_{H^m}^2.
\end{gather*}
Since $$\frac{C}{N} e^{-Nt} \| \tht_0 \|_{L^2}^2 \| \tht \|_{H^m}^2 \leq C e^{-Nt} \| \tht_0 \|_{L^2}^2,$$ applying Lemma~\ref{cal_lem}, we obtain \eqref{tem_decay_est_T_4}.
\end{proof}

\begin{lemma}\label{lem_tem_2}
    Let $\theta$ be a global classical solution of \eqref{eq:SIPM_T} with $N>0$. Suppose that \eqref{sol_bdd} and \eqref{sol_bdd1} are satisfied for any given $m \in \bbN$ with $m > 3$. Then, there exists a constant $C>0$ such that
    \begin{equation}\label{tem_decay_est_T_2}
        \| u_2(t) \|_{L^2} \leq C(1+Nt)^{-(1+\frac{m}{2})} \| \tht_0 \|_{H^m}
    \end{equation}
    and
    \begin{equation}\label{tem_decay_est_T_3}
        \| u_2(t) \|_{\dot{H}^m} \leq C(1+Nt)^{-1} \| \tht_0 \|_{H^m}
    \end{equation}
    for all $t>0$.
\end{lemma}
\begin{proof}
To prove \eqref{tem_decay_est_T_2} and \eqref{tem_decay_est_T_3}, it suffices to estimate $\bar{u}_2$ instead of $u_2$ thanks to \eqref{sf_u} and \eqref{avg_est_1}. We show \eqref{tem_decay_est_T_3} first. Multiplying $\tht$ equation in \eqref{eq:SIPM_T} by $-(-\Delta)^{m-1} \partial_1^2 u_2$ and integrating over $\bbT^2$, we have
\begin{equation*}
    \frac{1}{2} \frac{\ud}{\ud t} \|\bar{u}_2\|_{\dot{H}^m}^2 = \sum_{|\alpha| = m-1} \int_{\bbT^2} \partial^{\alpha} \partial_1 (u \cdot \nabla)\tht \partial^{\alpha} \partial_1 u_2 \,\ud x -N \| R_1 u_2 \|_{\dot{H}^m}^2.
\end{equation*}
The first term on the right-hand side is bounded by $K_1+K_2$, where
\begin{align*}
    K_1 &:= \sum_{|\alpha| = m-1} \int_{\bbT^2} \partial^{\alpha} (\partial_1 u \cdot \nabla)\tht \partial^{\alpha} \partial_1 u_2 \,\ud x, \\
    K_2 &:= \sum_{|\alpha| = m-1} \int_{\bbT^2} \partial^{\alpha} (u \cdot \nabla)\partial_1 \tht \partial^{\alpha} \partial_1 u_2 \,\ud x.
\end{align*}
It is clear by $(u \cdot \nabla)\tht = u_1 \partial_1 \tht + u_2 \partial_2 \tht$ that 
\begin{align*}
    |K_1| &\leq C\| \partial_1 u_1 \|_{H^{m-1}} \| R_1 \tht \|_{H^m} \| \partial_1 u_2 \|_{H^{m-1}} + C\| \partial_1 u_2 \|_{H^{m-1}}^2 \| \tht \|_{H^m} \\
    &\leq C\| \bar{u}_2 \|_{H^{m}} \| R_1 \tht \|_{H^m} \| R_1 u_2 \|_{H^{m}} + C\| R_1 u_2 \|_{H^{m}}^2 \| \tht \|_{H^m}.
\end{align*}
We note by $-\Delta u_2 = -\partial_1^2 \tht$ that $$K_2 = \sum_{|\alpha| = m-2} \int_{\bbT^2} \partial^{\alpha} (\partial_1 u \cdot \nabla)\partial_1 \tht \partial^{\alpha} \partial_1^2 \tht \,\ud x + \sum_{|\alpha| = m-2} \int_{\bbT^2} \partial^{\alpha} (u \cdot \nabla)\partial_1^2 \tht \partial^{\alpha} \partial_1^2 \tht \,\ud x.$$
By interpolation inequalities and $(u \cdot \nabla)\tht = u_1 \partial_1 \tht + u_2 \partial_2 \tht$ again, we can infer 
\begin{gather*}
    \left| \sum_{|\alpha| = m-2} \int_{\bbT^2} \partial^{\alpha} (\partial_1 u \cdot \nabla)\partial_1 \tht \partial^{\alpha} \partial_1^2 \tht \,\ud x \right| \leq C\| \nabla u_2 \|_{L^{\infty}} \| \bar{u}_2 \|_{H^{m}}^2 + C \| R_1 u_2 \|_{H^{m}} \| R_1 \tht \|_{H^{m}} \| \bar{u}_2 \|_{H^{m}}.
\end{gather*}
Using the divergence-free condition, we can similarly have
\begin{align*}
    \left| \sum_{|\alpha| = m-2} \int_{\bbT^2} \partial^{\alpha} (u \cdot \nabla)\partial_1^2 \tht \partial^{\alpha} \partial_1^2 \tht \,\ud x \right| \leq C \| R_1 u_2 \|_{H^{m}} \| R_1 \tht \|_{H^{m}} \| \bar{u}_2 \|_{H^{m}} + C\| \nabla u_2 \|_{L^{\infty}} \| \bar{u}_2 \|_{H^{m}}^2.
\end{align*}
Combining the estimates for $K_1$ and $K_2$ gives
\begin{align*}
    \frac{1}{2} \frac{\ud}{\ud t} \|\bar{u}_2\|_{\dot{H}^m}^2 &\leq \frac{C}{N} \| R_1 \tht \|_{H^m}^2 \| \bar{u}_2 \|_{H^m}^2 + C\| \nabla u_2 \|_{L^{\infty}} \| \bar{u}_2 \|_{H^m}^2 -(\frac{3}{4}N - C \| \tht \|_{H^m} )\| R_1 u_2 \|_{\dot{H}^m}^2 \\
    &\leq \frac{C}{N} \| R_1 \tht \|_{H^m}^2 \| \bar{u}_2 \|_{H^m}^2 + C\| \nabla u_2 \|_{L^{\infty}} \| \bar{u}_2 \|_{H^m}^2 - \frac{N}{2}\| R_1 u_2 \|_{\dot{H}^m}^2.
\end{align*}
We note for $M>0$ that
\begin{align*}
    \frac{1}{M} \|\bar{u}_2\|_{\dot{H}^m}^2 -\|R_1 u_2\|_{\dot{H}^m}^2 &= - \sum_{n_1 \neq 0} \left(\frac{n_1^2}{|n|^2} - \frac{1}{M} \right)|n|^{2m}|\mathscr{F} u_2(n)|^2 \\
    &\leq \frac{1}{M}\sum_{\substack{Mn_1^2 \leq |n|^2 \\ n_1 \neq 0}} |n|^{2m} |\mathscr{F} u_2(n)|^2 \leq \frac{1}{M^2} \| R_1 \tht \|_{H^m}^2.
\end{align*}
Then \eqref{sol_bdd}, \eqref{ass}, and $\| \bar{u}_2 \|_{H^m} \leq C\| \bar{u}_2 \|_{\dot{H}^m}$ show
\begin{gather*}
    \frac{\ud}{\ud t} \|\bar{u}_2\|_{\dot{H}^m}^2 \leq -\frac{N}{M} \|\bar{u}_2\|_{\dot{H}^m}^2 + \frac{C}{N} \| R_1 \tht \|_{H^m}^2 \| \bar{u}_2 \|_{\dot{H}^m}^2 + C\| \nabla u_2 \|_{L^{\infty}} \| \bar{u}_2 \|_{\dot{H}^m}^2  + \frac{CN}{M^2} \| R_1 \tht \|_{H^m}^2.
\end{gather*} We take $M = 1+\frac{Nt}{2}$ multyplying both sides by $M^2$. Then Lemma~\ref{cal_lem} with \eqref{sol_bdd1} implies \eqref{tem_decay_est_T_3}.

It remains to show \eqref{tem_decay_est_T_2}. Multiplying the $\tht$ equation in \eqref{eq:SIPM_T} by $R_1^2 u_2$ and integrating over $\bbT^2$, we have
\begin{equation*}
    \frac{1}{2} \frac{\ud}{\ud t} \int_{\bbT^2} |\bar{u}_2|^2 \,\ud x =- \int_{\bbT^2} (u \cdot \nabla)\tht R_1^2 u_2 \,\ud x -N \| R_1 u_2 \|_{L^2}^2.
\end{equation*}
Due to $R_1^2 u_2 = -\partial_1^2 (-\Delta)^{-1} \bar{u}_2$, there holds
\begin{align*}
    \left| - \int_{\bbT^2} (u \cdot \nabla)\tht R_1^2 u_2 \,\ud x \right| \leq \left| \int_{\bbT^2} (\partial_1u \cdot \nabla)\tht \partial_1 (-\Delta)^{-1} \bar{u}_2 \,\ud x \right| + \left| \int_{\bbT^2} (u \cdot \nabla)\partial_1 \tht \partial_1 (-\Delta)^{-1} \bar{u}_2 \,\ud x \right|.
\end{align*}
We can estimate the first term on the right-hand side as follows
\begin{gather*}
    \left| \int_{\bbT^2} (\partial_1u \cdot \nabla)\tht \partial_1 (-\Delta)^{-1} \bar{u}_2 \,\ud x \right| \\
    \leq \left| \int_{\bbT^2} \nabla (-\Delta)^{-1} \partial_1 u_1  \cdot \nabla (\partial_1 \tht \partial_1 (-\Delta)^{-1} \bar{u}_2) \,\ud x \right| + \left| \int_{\bbT^2} \nabla (-\Delta)^{-1} \partial_1 u_2  \cdot \nabla (\partial_2 \tht \partial_1 (-\Delta)^{-1} \bar{u}_2) \,\ud x \right| \\
    \leq C \| u_2 \|_{L^2} \| R_1 \tht \|_{H^m} \| R_1 u_2 \|_{L^2} + C \| \tht \|_{H^m} \| R_1 u_2 \|_{L^2}^2.
\end{gather*}
We have used $$\|\partial_1 (-\Delta)^{-1} \bar{u}_2 \|_{L^2} \leq \| R_1 u_2 \|_{L^2}$$ in the last inequality. On the other hand, the second term is bounded by
\begin{gather*}
    \left| \int_{\bbT^2} u_1 \partial_1^2 \tht \partial_1 (-\Delta)^{-1} \bar{u}_2 \,\ud x \right| + \left| \int_{\bbT^2} u_2 \partial_1 \partial_2 \tht \partial_1 (-\Delta)^{-1} \bar{u}_2 \,\ud x \right| \\
    \leq \| u_1 \|_{L^2} \| \Delta u_2 \|_{L^{\infty}} \| R_1 u_2 \|_{L^2} + \| u_2 \|_{L^2} \| \partial_1 \partial_2 \tht \|_{L^{\infty}} \| R_1 u_2 \|_{L^2} \\
    \leq C \| u_1 \|_{L^2} \| \bar{u}_2 \|_{L^2}^{1 - \frac{3}{m}} \| u_2 \|_{\dot{H}^m}^{\frac{3}{m}} \| R_1 u_2 \|_{L^2} + C (\| \bar{u}_2 \|_{L^2} + e^{-Nt} \| \tht_0 \|_{L^2}) \| R_1 \tht \|_{H^m} \| R_1 u_2 \|_{L^2}.
\end{gather*}
We have used \eqref{u2_est} and $$\| \Delta u_2 \|_{L^{\infty}} \leq \| \bar{u}_2 \|_{L^2}^{1- \frac{3}{m}} \| \bar{u}_2 \|_{\dot{H}^m}^{\frac{3}{m}}$$ which is obtained by $m>3$ and the interpolation inequality. Combining the above estimates, we arrive at
\begin{gather*}
    \frac{1}{2}\frac{\ud}{\ud t} \| \bar{u}_2 \|_{L^2}^2 \leq - (\frac{3}{4}N - \| \tht \|_{H^m}) \| R_1 u_2 \|_{L^2}^2 + \frac{C}{N} \| \bar{u}_2 \|_{L^2}^2 \| R_1 \tht \|_{H^m}^2 + \frac{C}{N} e^{-2Nt} \| \tht_0 \|_{L^2}^2 \| R_1 \tht \|_{H^m}^2 \\
    + \frac{C}{N} \| u_1 \|_{L^2}^2 \| \bar{u}_2 \|_{L^2}^{2 - \frac{6}{m}} \| u_2 \|_{\dot{H}^m}^{\frac{6}{m}}  \\
    \leq -\frac{N}{2} \| R_1 u_2 \|_{L^2}^2 + \frac{C}{N} \| \bar{u}_2 \|_{L^2}^2 \| R_1 \tht \|_{H^m}^2 + \frac{C}{N} e^{-2Nt} \| \tht_0 \|_{L^2}^2 \| R_1 \tht \|_{H^m}^2 + \frac{C}{N} \| u_1 \|_{L^2}^2 \| \bar{u}_2 \|_{L^2}^{2 - \frac{6}{m}} \| u_2 \|_{\dot{H}^m}^{\frac{6}{m}} .
\end{gather*}
We note for $M \geq 1$ that
\begin{align*}
    \frac{1}{M} \| \bar{u}_2 \|_{L^2}^2 - \|R_1 u_2\|_{L^2}^2 &= - \sum_{n_1 \neq 0} \left( \frac{1}{M} - \frac{n_1^2}{|n|^2} \right)|\mathscr{F} u_2(n)|^2 \\
    &\leq \frac{1}{M}\sum_{\substack{Mn_1^2 \leq |n|^2 \\ n_1 \neq 0}} |\mathscr{F} R_1^2 \tht(n)|^2 \leq \frac{1}{M^{m+2}} \| R_1 \tht \|_{H^m}^2.
\end{align*}
Thus, there holds
\begin{gather*}
    \frac{\ud}{\ud t} \| \bar{u}_2 \|_{L^2}^2 \leq - \frac{N}{M} \| \bar{u}_2 \|_{L^2}^2 + \frac{CN}{M^{m+2}} \| R_1 \tht \|_{H^m}^2 + \frac{C}{N} \| \bar{u}_2 \|_{L^2}^2 \| R_1 \tht \|_{H^m}^2 \\
    + \frac{C}{N} e^{-2Nt} \| \tht_0 \|_{L^2}^2 \| R_1 \tht \|_{H^m}^2 + \frac{C}{N} \| u_1 \|_{L^2}^2 \| \bar{u}_2 \|_{L^2}^{2 - \frac{6}{m}} \| u_2 \|_{\dot{H}^m}^{\frac{6}{m}} .
\end{gather*}
We take $M = 1+\frac{Nt}{m+2}$ and multiplying both sides by $M^{m+2}$. Then, it follows \begin{gather*}
    \frac{\ud}{\ud t} \left( (1+\frac{Nt}{m+2})^{m+2} \| \bar{u}_2 \|_{L^2}^2 \right) \leq CN \| R_1 \tht \|_{H^m}^2 + \frac{C}{N} (1+\frac{Nt}{m+2})^{m+2} \| \bar{u}_2 \|_{L^2}^2 \| R_1 \tht \|_{H^m}^2 \\
    \frac{C}{N} \| \tht_0 \|_{L^2}^2 \| R_1 \tht \|_{H^m}^2 + \frac{C}{N} (1+\frac{Nt}{m+2})^{m+2} \| u_1 \|_{L^2}^2 \| \bar{u}_2 \|_{L^2}^{2 - \frac{6}{m}} \| u_2 \|_{\dot{H}^m}^{\frac{6}{m}}.
\end{gather*}
Note from \eqref{tem_decay_est_T_4} and \eqref{tem_decay_est_T_3} that \begin{align*}
    &\frac{C}{N} (1+\frac{Nt}{m+2})^{m+2} \| u_1 \|_{L^2}^2 \| \bar{u}_2 \|_{L^2}^{2 - \frac{6}{m}} \| u_2 \|_{\dot{H}^m}^{\frac{6}{m}} \\ &\hphantom{\qquad\qquad}\leq \frac{C}{N} (1+\frac{Nt}{m+2})^{2-m} \left( (1+\frac{Nt}{m+2})^{m+2} \| \bar{u}_2 \|_{L^2}^2 \right)^{1 - \frac{3}{m}} \\
    &\hphantom{\qquad\qquad}\leq \frac{C}{N} (1+\frac{Nt}{m+2})^{2-m} + \frac{C}{N} (1+\frac{Nt}{m+2})^{2-m}\left( (1+\frac{Nt}{m+2})^{m+2} \| \bar{u}_2 \|_{L^2}^2 \right).
\end{align*}
Therefore, by Lemma~\ref{cal_lem}, we obtain \eqref{tem_decay_est_T_3}. This completes the proof.
\end{proof}

Now, we are ready to show \eqref{tem_decay_est_T}.     \eqref{tem_decay_est_T_2} and \eqref{tem_decay_est_T_3} implies $$(1+Nt)^{1+\frac{m-s}{2}}\| u_2(t) \|_{H^{s}} \leq C \| \tht_0 \|_{H^m}, \qquad s \in [0,m].$$ Interpolating this with \eqref{tem_decay_est_T_1}, we obtain $$(1+Nt)^{\frac{1}{2}+\frac{m-s}{2}}\| u(t) \|_{H^{s}} = (1+Nt)^{\frac{1}{2}+\frac{m-s}{2}}\| R_1 \tht(t) \|_{H^{s}} \leq C \| \tht_0 \|_{H^m}, \qquad s \in [0,m].$$
We recall \eqref{def_sgm}
\begin{equation*}
    \sigma(x_2) = \int_{\bbT} \tht_0 \,\ud x_1 - \int_0^{\infty} \int_{\bbT} \left((u \cdot \nabla)\tht + Nu_2 \right)\,\ud x_1 \ud t.
\end{equation*}
We note from \eqref{avg_est_1} that
\begin{align*}
    \int_{\bbT} \tht(t,x) \,\ud x_1 - \sigma(x_2) &= \left(\int_{\bbT} \tht_0 \,\ud x_1 - \int_0^t \int_{\bbT} \left((u \cdot \nabla) \tht + N u_2 \right) \,\ud x_1 \ud \tau\right) - \sigma(x_2) \\
    &= - \int_t^{\infty} \int_{\bbT} (u \cdot \nabla) \tht \,\ud x_1 \ud \tau - \int_t^{\infty} \int_{\bbT} N u_2 \,\ud x_1 \ud \tau.
\end{align*}
Since
\begin{align*}
    \left\| - \int_t^{\infty} \int_{\bbT} N u_2 \,\ud x_1 \ud \tau \right\|_{L^2} \leq N\int_t^{\infty} \| u_2 \|_{L^2}\,\ud \tau &\leq CN \| \tht_0 \|_{H^m} \int_t^{\infty} (1+N\tau)^{-\frac{m+2}{2}} \,\ud \tau \\
    &\leq C(1+Nt)^{-\frac{m}{2}} \| \tht_0 \|_{H^m}
\end{align*}
and
\begin{align*}
    \left\| - \int_t^{\infty} \int_{\bbT} (u \cdot \nabla) \tht \,\ud x_1 \ud \tau \right\|_{L^2} &\leq \int_t^{\infty} (\| u_1 \|_{L^2} \| R_1\tht \|_{H^{m}} + \| u_2 \|_{L^2} \| \tht \|_{H^m}) \,\ud \tau \\
    &\leq C \| \tht_0 \|_{H^m}^2 \int_t^{\infty} (1+N\tau)^{-\frac{m+2}{2}} \,\ud \tau \\
    &\leq \frac{C}{N} (1+Nt)^{-\frac{m}{2}} \| \tht_0 \|_{H^m}^2,
\end{align*}
we obtain $$\| \theta(t) - \sigma \|_{L^2} \leq C(1+Nt)^{-\frac{m}{2}} \| \tht_0 \|_{H^m}.$$ It remains to show that $$\| \tht(t) - \sigma \|_{H^m} \leq C\| \tht_0 \|_{H^m}.$$ It is clear that $\| \tht(t) - \sigma \|_{H^m} \leq C\| \tht_0 \|_{H^m} + \| \sigma \|_{H^m}$. To show that
$$\| \sigma \|_{H^m}^2 = \sum_{n_1 = 0} (1+|n|^{2})^{m} \left| \lim_{t \to \infty} \mathscr{F}\tht(t,n) \right|^2 \leq C \| \tht_0 \|_{H^m}^2,$$ we define a sequence $\{ \sigma_k \}_{k \in \bbN} \subset H^m(\bbT^2)$ by $$\sigma_k := \sum_{\substack{n_1 = 0 \\ |n_2| \leq k}} \left( \lim_{t \to \infty} \mathscr{F}\tht(t,n) \right) e^{2\pi i n \cdot x}.$$
For each $k \in \bbN$, there exists a constant $C>0$ not depending on $k$ such that  
\begin{align*}
    \| \sigma_k \|_{H^m}^2 &\leq \left\| \sigma_k - \int_{\bbT} \tht(t,x) \,\ud x_1 \right\|_{H^m}^2 + \sup_{t \in [0,\infty)} \left\| \int_{\bbT} \tht(t,x) \,\ud x_1 \right\|_{H^m}^2 \\
    &\leq (1+k^2)^m \sum_{\substack{n_1 = 0 \\ |n_2| \leq k}} \left| \mathscr{F} \tht(t,n) - \lim_{t \to \infty} \mathscr{F}\tht(t,n) \right|^2 + C\| \tht_0 \|_{H^m}^2 \\
    &\leq C \| \tht_0 \|_{H^m}^2
\end{align*} by taking $t$ large enough. Since $\| \sigma \|_{H^m}^2 = \sup_{k \in \bbN} \| \sigma_k \|_{H^m}^2$, we obtain the claim. This completes the proof.

\section{Sketch of proof of Theorem~\ref{thm:stb_O}}\label{sec_O}
Here we sketch the proof of Theorem~\ref{thm:stb_O} omitting the details that can be proved analogously to the ones for the $\bbT^2$ case. The key is to mimic our method of heavily using Fourier transforms even in the confined strip $\Omega=\bbT\times[-1,1]$ by introducing suitable analogues. 

The horizontally periodic strip  was considered in \cite{Cordoba} as a physically more meaningful scenario for the IPM equation \eqref{eq:IPM}. The targeted linearly stratified density $\rho_s(x_2)=x_2$ is \emph{not} bounded in $\bbR^2$ and it is \emph{not} periodic in $\bbT^2$. Motivated by the fact that gravity causes the IPM equation to be anisotropic in its nature, certain finiteness was imposed on the vertical depth of the target domain in \cite{Cordoba}. But this immediately gives rise to nontrivial difficulties due to the presence of \emph{boundary.} The lack of the boundary conditions for the higher order derivatives forces us to tweak the standard Sobolev spaces - the new function spaces $X^m$ and $Y^m$ are introduced in \cite{Cordoba}, whose definitions can be found in \eqref{def_Xm}-\eqref{def_Ym}.

Our goal is to develop the $\Omega$-analogues that will allow us to access the Fourier-wise blueprint we adopted in the previous sections. The rest of this section is as follows. In Section~\ref{sec_Xm}, we gather some basic properties of $X^m$ and $Y^m$ and obtain the analogues of Fourier transform and the Fourier inversion formula. This produces the analogous representations of the pressure and the velocity field via the density. The next section is devoted to the analogous energy inequality. The proof of Theorem~\ref{thm:stb_O} is only sketched in the last section where we ommitted the details that can be verified in a completely analogous way to the way the torus case is proved.

\subsection{Properties of $X^m$ and $Y^m$}\label{sec_Xm}

We recall our solution spaces $X^m$ and $Y^m$. Let $\{ b_q \}_{q \in \bbN}$ and $\{ c_q \}_{q \in \bbN \cup \{0\}}$ be the sets of the functions, which are pairwisely orthonormal in $L^2$, defined by
\begin{align*}
	b_{q} (x_2) = 
	\begin{cases}
		\displaystyle \sin(\frac {\pi}{2} q x_2)  &\quad  q : \mbox{ even,}  \vspace{2mm}\\
		\displaystyle \cos(\frac {\pi}{2} q x_2)  &\quad  q : \mbox{ odd,} 
	\end{cases} \qquad c_{q} (x_2) = 
	\begin{cases}
		\displaystyle -\sin(\frac {\pi}{2} q x_2)  &\quad  q : \mbox{ odd,} \vspace{2mm}\\
		\displaystyle \cos(\frac {\pi}{2} q x_2)  &\quad  q : \mbox{ even.}
	\end{cases}
\end{align*}
It is not hard to verify that $\{\mathscr{B}_{p,q}\}_{(p,q) \in \bbZ \times \bbN}$ and $\{\mathscr{B}_{p,q}\}_{(p,q) \in \bbZ \times (\bbN \cup \{0\})}$ are orthonormal bases for $X^m$ and $Y^m$ respectively, where $$\mathscr{B}_{p,q}(x) := e^{2\pi i p x_1} b_{q}(x_2), \qquad \mathscr{C}_{p,q}(x) := e^{2 \pi i p x_1} c_{q}(x_2).$$ For details, we refer to \cite{Cordoba}. Employing the below notations $$\mathscr{F}_b f(p,q) = \langle f,\mathscr{B}_{p,q} \rangle = \int_{\Omega} \overline{\mathscr{B}_{p,q}(x)}f(x) \,\ud x, \qquad \mathscr{F}_c f(p,q) = \langle f,\mathscr{C}_{p,q} \rangle = \int_{\Omega} \overline{\mathscr{C}_{p,q}(x)}f(x) \,\ud x,$$
we immediately see that the following inversion formula holds, $$f(x) = \sum_{(p,q) \in \bbZ \times \bbN} \mathscr{F}_b f(p,q) \mathscr{B}_{p,q}(x), \qquad g(x) = \sum_{(p,q) \in \bbZ \times (\bbN \cup \{0\})} \mathscr{F}_c f(p,q) \mathscr{C}_{p,q}(x),$$ for any $f \in X^m$ and $g \in Y^m$. \\

From now on, we fix $m \geq 2$. We collect the basic properties of $X^m$ and $Y^m$: derivatives transform into polynomial multipliers on the frequency side, the products of two functions from $X^m$ or $Y^m$ belong to either $X^m$ or $Y^m$, and the solution can be expressed with the analogues of Riesz transforms. For the exact definition of the $\Omega$-version of the Risez transform, see \eqref{def_Riesz_X}-\eqref{def_Riesz_Y}.   \\

\textbf{Partial derivatives}: Note that $\partial_1 \mathscr{B}_{p,q}(x) = 2\pi i p \mathscr{B}_{p,q}(x)$, $\partial_2 \mathscr{B}_{p,q}(x) = \Black{\frac{\pi}{2} q \mathscr{B}_{p,q}(x)}$, $\partial_1 \mathscr{C}_{p,q}(x) = 2 \pi i p \mathscr{C}_{p,q}(x)$, and $\partial_2 \mathscr{C}_{p,q}(x) = -\frac{\pi}{2} q \mathscr{B}_{p,q}(x)$. Using them, we can see that for $f \in X^m$ and $g \in Y^m$, we have $\partial_1 f \in X^{m-1}$, $\partial_2 f \in Y^{m-1}$, $\partial_1 g \in Y^{m-1}$, and $\partial_2 g \in X^{m-1}.$  We also have $\mathscr{F}_b\partial_1 f = 2\pi i p \mathscr{F}_b f$, $\mathscr{F}_c \partial_2 f = - \frac{\pi}{2} q \mathscr{F}_b f$, $\mathscr{F}_c\partial_1 g = 2\pi i p \mathscr{F}_c g$, and $\mathscr{F}_b \partial_2 g = \frac{\pi}{2} q \mathscr{F}_c g$. \\

\textbf{Products}: The product $fg$ of the two functions  $f\in X^m$ and $g\in Y^m$ belongs to $X^m$. Meanwhile, the product of the two functions in the same space must be in $Y^m$. For the proofs, we refer to Lemma~\ref{lem_basis1} and Lemma~\ref{lem_basis2}.\\

\textbf{Solution formulas}: We express a local-in-time solution $(\tht,u,P)$ given in \cite[Theorem 4.1]{Cordoba} by the use of \Black{$\mathscr{B}_{p,q}$ and $\mathscr{C}_{p,q}$.} Since $\tht \in X^m$ for $m \geq 3$, there holds
\begin{equation*}
    \tht = \sum_{(p,q) \in \bbZ \times \bbN} \mathscr{F}_b \tht(p,q) \mathscr{B}_{p,q}(x).
\end{equation*}
Due to the relationship $-\Delta P = -\partial_2 \tht$, we have $$P(x) = \sum_{\substack{(p,q) \in \bbZ \times (\bbN \cup \{0\}) \\ (p,q) \neq 0}} \frac{-\frac{\pi}{2}q}{(2\pi p)^2 + (\frac{\pi}{2}q)^2}\mathscr{F}_b \tht (p,q) \mathscr{C}_{p,q}(x) + \mathscr{F}_bP(0,0).$$ Thus, we obtain
\begin{equation}\label{sf_u_O}
    \begin{aligned}
        u_1 &= \sum_{\substack{(p,q) \in \bbZ \times (\bbN \cup \{0\}) \\ (p,q) \neq 0}} \frac{-i\pi^2 pq}{(2\pi p)^2 + (\frac{\pi}{2}q)^2}\mathscr{F}_b \tht (p,q) \mathscr{C}_{p,q}(x), \\
        u_2 &= \sum_{\substack{(p,q) \in \bbZ \times \bbN \\ (p,q) \neq 0}} \frac{(2\pi p)^2}{(2\pi p)^2 + (\frac{\pi}{2}q)^2}\mathscr{F}_b \tht (p,q) \mathscr{B}_{p,q}(x),
    \end{aligned}
\end{equation}
which are consistent with the velocity field formula by means of the stream function (see \cite{Cordoba}).

The above representations of the solutions lead to the analogue of Lemma~\ref{lem_avg}; we see that the mean values of solutions in $\bbT^2$ and solutions in $\Omega$ are fundamentally distinguishable. One may compare the below to Lemma~\ref{lem_avg}. 
\begin{lemma}\label{lem_avg_O}
    Let $(\tht,u,P)$ be a classical solution to \eqref{eq:SIPM_T}. Then, there hold
    \begin{equation}\label{avg_est_1_O}
        \int_{\Omega} \tht(t,x) \,\ud x = \int_{\Omega} \tht_0(x) \,\ud x
    \end{equation}
    and
    \begin{equation}\label{avg_est_2_O}
        \int_{\bbT} u(t,x) \,\ud x_1 = 0.
    \end{equation}
\end{lemma}
\begin{proof}
    From the representation formula \eqref{sf_u_O}, we immediately obtain \eqref{avg_est_2_O}. Using the $\tht$ equation in \eqref{eq:SIPM_T} with \eqref{avg_est_2_O}, we have $$\frac{\ud}{\ud t} \int_{\Omega} \tht \,\ud x = - \int_{\Omega} (u \cdot \nabla) \tht \,\ud x.$$ Integration by parts further implies $$- \int_{\Omega} (u \cdot \nabla) \tht \,\ud x = \int_{\partial \Omega} u_2 \tht \,\ud x_1 = 0.$$ This proves \eqref{avg_est_1_O}. The proof is finished.
\end{proof}

\subsection{Energy inequality}
In this section, we improve Theorem 5.2 in \cite{Cordoba}. To this end, we need to show the energy estimate of the types in Proposition~\ref{prop_energy_R} and Proposition~\ref{prop_energy_T}. This necessitates the analogous notions of Riesz transforms, which are defined as follows. We define an automorphism $R_1$ on $X^m$ and $Y^m$ by
\begin{equation}\label{def_Riesz_X}
    R_1 f(x) := \sum_{(p,q) \in \bbZ \times \bbN} \frac{-2\pi i p}{\sqrt{(2\pi p)^2 + (\frac{\pi}{2}q)^2}} \mathscr{F}_b f(p,q) \mathscr{B}_{p,q}(x), \qquad f \in X^m
\end{equation}
and
\begin{equation}\label{def_Riesz_Y}
    R_1 g(x) := \sum_{(p,q) \in \bbZ \times (\bbN \cup \{0\})} \frac{-2\pi i p}{\sqrt{(2\pi p)^2 + (\frac{\pi}{2}q)^2}} \mathscr{F}_c g(p,q) \mathscr{C}_{p,q}(x), \qquad g \in Y^m.
\end{equation}
\begin{proposition}\label{prop_energy_O}
Let $\theta$ be a classical solution of \eqref{eq:SIPM_T}. Then for any given $m \in \bbN$ with $m \geq 3$, there exists a constant $C=C(m)>0$ such that
\begin{equation}\label{energy_O}
\begin{gathered}
    \frac{1}{2}\sup_{t \in [0,T]} \|\theta\|_{H^{m}}^2 + N \int_0^T \|R_1\theta\|_{H^{m}}^{2} \,\ud t \\
    \leq \frac{1}{2} \| \tht_0 \|_{H^{m}}^2 + C\sup_{t \in [0,T]} \| \theta \|_{H^m} \int_0^T \|R_1 \theta \|_{H^m}^2 \,\ud t + C \sup_{t \in [0,T]} \| \theta \|_{H^m}^2 \int_0^T \| \partial_2 u_2 \|_{L^{\infty}} \, \ud t.
\end{gathered}
\end{equation}
\end{proposition}
\begin{proof}
    From the boundary condition and $\tht$ equation in \eqref{eq:SIPM_T}, we can have $$\frac{1}{2} \frac{\ud}{\ud t} \| \tht \|_{L^2}^2 + N \int_{\Omega} u_2 \tht \,\ud x = 0.$$ Since \eqref{sf_u_O} gives $$\int_{\Omega} u_2 \tht \,\ud x = \sum_{(p,q) \in \bbZ \times \bbN} \mathscr{F}_b u_2(p,q) \overline{\mathscr{F}_b \tht(p,q)} = \| R_1 \tht \|_{L^2}^2,$$ there holds $$\frac{1}{2} \frac{\ud}{\ud t} \| \tht \|_{L^2}^2 + N \| R_1 \tht \|_{L^2}^2 = 0.$$ Next, let $\alpha$ be a multi-index with $|\alpha| = m$. Taking $\partial^{\alpha}$ to the $\tht$ equation and multiplying $\partial^{\alpha}$, and integrating over $\Omega$ yields $$\frac{1}{2} \frac{\ud}{\ud t} \| \tht \|_{\dot{H}^m}^2 + N \| R_1 \tht \|_{\dot{H}^m}^2 = -\sum_{|\alpha| = m} \int_{\Omega} \partial^{\alpha} (u \cdot \nabla) \tht \partial^{\alpha} \tht \,\ud x.$$ We first consider the case $\partial^{\alpha} \neq \partial_2^m$. Due to the boundary condition, we can see
    \begin{align*}
        \left| -\sum_{|\alpha| = m} \int_{\Omega} \partial^{\alpha} (u \cdot \nabla) \tht \partial^{\alpha} \tht \,\ud x \right| &= \left| -\sum_{|\alpha| = m} 
        \left( \int_{\Omega} \partial^{\alpha} (u \cdot \nabla) \tht \partial^{\alpha} \tht \,\ud x - \int_{\Omega} (u \cdot \nabla) \partial^{\alpha} \tht \partial^{\alpha} \tht \,\ud x \right) \right| \\
        &\leq \sum_{|\alpha| = m} \| \partial^{\alpha} (u \cdot \nabla) \tht - (u \cdot \nabla) \partial^{\alpha} \tht \|_{L^2} \| \partial^{\alpha} \tht \|_{L^2}.
    \end{align*}
    Employing \cite[Lemma 4.2]{Cordoba}, we clearly have 
    \begin{align*}
        \left| -\sum_{|\alpha| = m} \int_{\Omega} \partial^{\alpha} (u \cdot \nabla) \tht \partial^{\alpha} \tht \,\ud x \right| &\leq C (\| \nabla u \|_{L^{\infty}} \| \tht \|_{H^m} + \| u \|_{H^m} \| \tht \|_{L^{\infty}}) \| R_1 \tht \|_{H^m} \\
        &\leq C \| R_1 \tht \|_{H^m}^2 \| \tht \|_{H^m}.
    \end{align*}
    We have used $m > 2$ and $\| u \|_{H^m} \leq \| R_1 \tht \|_{H^m}$ in the last inequality. For $\partial^{\alpha} = \partial_2^m$, we can have \Black{by the divergence-free condition and $\partial^m(fg) = \sum_{k=1}^{m} \partial^{m-k} (\partial f \partial^{k-1}g) + f\partial^m g$ that
    \begin{gather*}
        \int_{\Omega} \partial_2^m (u \cdot \nabla) \tht \partial_2^m \tht \,\ud x = \int_{\Omega} \left( \partial_2^m (u \cdot \nabla) \tht - (u \cdot \nabla)\partial_2^m \tht \right) \partial_2^m \tht \,\ud x\\
        = \sum_{k=1}^{m} \int_{\Omega} \partial_2^{m-k} (\partial_2 u_1 \partial_1 \partial_2^{k-1} \tht) \partial_2^m \tht \,\ud x + \sum_{k=1}^{m} \int_{\Omega} \partial_2^{m-k} (\partial_2 u_2 \partial_2 \partial_2^{k-1} \tht) \partial_2^m \tht \,\ud x \\
        = \sum_{k=1}^{m} \int_{\Omega} \partial_2^{m-k} (\partial_2 u_1 \partial_1 \partial_2^{k-1} \tht) \partial_2^m \tht \,\ud x + m \int_{\Omega} \partial_2u_2 |\partial_2^m \tht|^2 \,\ud x \\
        + \sum_{k=1}^{m} \sum_{\ell=1}^{m-k} \int_{\Omega} \partial_2^{m-k-\ell} (\partial_2^2 u_2 \partial_2 \partial_2^{k-1} \partial_2^{\ell-1} \tht) \partial_2^m \tht \,\ud x.
    \end{gather*}}
    It is not hard to show that the first two integrals on the right-hand side are bounded by $$C \| R_1 \tht \|_{H^m}^2 \| \tht \|_{H^m} + C \| \partial_2 u_2 \|_{L^{\infty}} \| \tht \|_{H^m}^2.$$ We estimate the remainder integral. \Black{We only provide the estimate for $k=\ell=1$ case since the others can be inferred from it. Let us consider odd $m$ first. Since $\partial_2^{m-1} \tht$ vanishes on the boundary, it follows
    \begin{align*}
        \int_{\Omega} \partial_2^{m-2} (\partial_2^2 u_2 \partial_2 \tht) \partial_2^m \tht \,\ud x &= \int_{\Omega} \partial_2^{m-1} (\partial_1\partial_2 u_1 \partial_2 \tht) \partial_2^{m-1} \tht \,\ud x \\
        &= -\int_{\Omega} \partial_2^{m-1} (\partial_2 u_1 \partial_1\partial_2 \tht) \partial_2^{m-1} \tht \,\ud x - \int_{\Omega} \partial_2^{m-1} (\partial_2 u_1 \partial_2 \tht) \partial_1\partial_2^{m-1} \tht \,\ud x \\
        &= \int_{\Omega} \partial_2^{m-2} (\partial_2 u_1 \partial_1\partial_2 \tht) \partial_2^{m} \tht \,\ud x - \int_{\Omega} \partial_2^{m-1} (\partial_2 u_1 \partial_2 \tht) \partial_1\partial_2^{m-1} \tht \,\ud x.
    \end{align*}}
    By \cite[Lemma 4.2]{Cordoba}, we have \begin{equation}\label{bdry_est}
        \left| \int_{\Omega} \partial_2^{m-2} (\partial_2^2 u_2 \partial_2 \tht) \partial_2^m \tht \,\ud x \right| \leq C \| R_1 \tht \|_{H^m}^2 \| \tht \|_{H^m}.
    \end{equation} On the other hand, for even $m$,
    \Black{\begin{align*}
        \int_{\Omega} \partial_2^{m-2} (\partial_2^2 u_2 \partial_2 \tht) \partial_2^m \tht \,\ud x &= \int_{\Omega} \partial_2^{m-2} (\partial_2 u_1 \partial_1\partial_2 \tht) \partial_2^{m} \tht \,\ud x - \int_{\Omega} \partial_2^{m-2} \partial_1 (\partial_2 u_1 \partial_2 \tht) \partial_2^{m} \tht \,\ud x \\
        &= \int_{\Omega} \partial_2^{m-2} (\partial_2 u_1 \partial_1\partial_2 \tht) \partial_2^{m} \tht \,\ud x - \int_{\Omega} \partial_2^{m-1} (\partial_2 u_1 \partial_2 \tht) \partial_1\partial_2^{m-1} \tht \,\ud x.
    \end{align*}}
    We have used the orthogonality of the basis of $X^m$ so that
    \begin{align*} \int_{\Omega} \partial_2^{m-2} \partial_1 (\partial_2 u_1 \partial_2 \tht) \partial_2^{m} \tht \,\ud x &= \sum_{(p,q) \in \bbZ \times \bbN} 2\pi i p \mathscr{F}_b \partial_2^{m-2} (\partial_2 u_1 \partial_2 \tht) \overline{\mathscr{F}_b \partial_2^{m} \tht} \\
    &= \sum_{(p,q) \in \bbZ \times \bbN} \mathscr{F}_c \partial_2^{m-1} (\partial_2 u_1 \partial_2 \tht) \overline{\mathscr{F}_c \partial_1\partial_2^{m-1} \tht} \\
    &= \int_{\Omega} \partial_2^{m-1} (\partial_2 u_1 \partial_2 \tht) \partial_1\partial_2^{m-1} \tht \,\ud x.
    \end{align*}
    Hence, \eqref{bdry_est} follows. Combining the above estimates gives $$\frac{1}{2} \frac{\ud}{\ud t} \| \tht \|_{\dot{H}^m}^2 + N \| R_1 \tht \|_{\dot{H}^m}^2 \leq C \| R_1 \tht \|_{H^m}^2 \| \tht \|_{H^m} + C \| \partial_2 u_2 \|_{L^{\infty}} \| \tht \|_{H^m}^2.$$
    
    Repeating the above procedure with the multi-index $\alpha$ with $|\alpha| = 1,2, \cdots, m-1$, we deduce that $$\frac{1}{2} \frac{\ud}{\ud t} \| \tht \|_{H^m}^2 + N \| R_1 \tht \|_{H^m}^2 \leq C \| R_1 \tht \|_{H^m}^2 \| \tht \|_{H^m} + C \| \partial_2 u_2 \|_{L^{\infty}} \| \tht \|_{H^m}^2.$$ Thus, \eqref{energy_O} is obtained. This finishes the proof.
\end{proof}

\subsection{Proof of the global existence of solutions}
Employing the convolution inequalities in Lemma~\ref{lem_conv_O}, which is a counterpart of Lemma~\ref{conv_ineq_T2}, we can estimate the key integral in Proposition~\ref{prop_energy_O}. Since the proof is almost \Black{the same} with the 2D torus case, we leave it to the interested readers. Now Theorem~\ref{thm:stb_O} follows from Proposition~\ref{prop_energy_O} and the following lemma.
\begin{lemma}
     Let $\tht$ be a global classical solution of \eqref{eq:SIPM_T}. Then, for any given $m > 3$ and $s \in [0,m-2)$, there exists a constant $C>0$ such that
        \Black{\begin{equation*} \begin{gathered}
            \sum_{(p,q) \in \mathbb{Z} \times \bbN} \int_0^T \left(1+(2\pi p)^2 + (\frac{\pi}{2}q)^2 \right)^{\frac s2} |\mathscr{F}_b u_2(p,q)| \, \ud t \\
            \leq \frac{1}{N} \sum_{(p,q) \in \mathbb{Z} \times \bbN} \int_0^T \left(1+(2\pi p)^2 + (\frac{\pi}{2}q)^2 \right)^{\frac s2} |\mathscr{F}_b \theta_0(p,q) |\, \ud t\\
    		+ \frac{C}{N}\int_0^T \| R_1 \theta \|_{H^{m}}^2 \, \ud t + \frac{C}{N} \sup_{t \in [0,T]} \| \theta \|_{H^m} \sum_{(p,q) \in \mathbb{Z} \times \bbN} \int_0^T \left(1+(2\pi p)^2 + (\frac{\pi}{2}q)^2 \right)^{\frac s2} |\mathscr{F}_b u_2(p,q)| \, \ud t
		\end{gathered}
	\end{equation*}}
	for all $T>0$.
\end{lemma}

\subsection{Proof of the decay estimates}

We use the notation $$\bar{f} := f - \int_{\bbT} f(x) \,\ud x_1 = \sum_{p \neq 0} \mathscr{F}_b f(p,q) \mathscr{B}_{p,q}(x), \qquad f \in X^m.$$ The following two propositions can be proved analogously to the ones for the $\bbT^2$ setting, replacing the usual Fourier basis $e^{2\pi i n \cdot x}$ by $\mathscr{B}_{p,q}$ and $\mathscr{C}_{p,q}$ that are adapted to the horizontally periodic strip.  
\begin{proposition}\label{prop_tem_3}
    Let $\theta$ be a global classical solution of \eqref{eq:SIPM_T} with $N>0$. Suppose that \eqref{sol_bdd} is satisfied for any given $m \in \bbN$ with $m > 2$. Then, there exists a constant $C>0$ such that
    \begin{equation}\label{tem_decay_est_O_1}
        \| \bar{\tht}(t) \|_{L^2} \leq C(1+Nt)^{-\frac{m}{2}} \| \tht_0 \|_{H^m}
    \end{equation}
    and
    \begin{equation}\label{tem_decay_est_O_4}
        \| u(t) \|_{L^2} \leq C(1+Nt)^{-(\frac{1}{2}+\frac{m}{2})} \| \tht_0 \|_{H^m}
    \end{equation}  
    for all $t>0$.
\end{proposition}

\begin{lemma}\label{lem_tem_4}
    Let $\theta$ be a global classical solution of \eqref{eq:SIPM_T} with $N>0$. Suppose that \eqref{sol_bdd} and \eqref{sol_bdd1} are satisfied for any given $m \in \bbN$ with $m > 3$. Then, there exists a constant $C>0$ such that
    \begin{equation}\label{tem_decay_est_O_2}
        \| u_2(t) \|_{L^2} \leq C(1+Nt)^{-(1+\frac{m}{2})} \| \tht_0 \|_{H^m}
    \end{equation}
    and
    \begin{equation}\label{tem_decay_est_O_3}
        \| u_2(t) \|_{\dot{H}^m} \leq C(1+Nt)^{-1} \| \tht_0 \|_{H^m}
    \end{equation}
    for all $t>0$.
\end{lemma}

\newpage

\section*{Competing Interests}
 The authors have no competing interests to declare that are relevant to the content of this article.
\section*{Data Availability}
Data sharing not applicable to this article as no datasets were generated or analysed during the current study.

\section*{Acknowledgment}

\noindent Junha Kim was supported by a KIAS Individual Grant (MG086501) at Korea Institute for Advanced Study. We sincerely thank the anonymous referees for very careful reading of the manuscript and providing several suggestions, which have been reflected in the paper. 
\\

\appendix

\bibliographystyle{amsplain}


\newpage

\section{Appendix}
\subsection{A calculus lemma}
\begin{lemma}\label{cal_lem}
    Let $f$ be a positive $C^1$ function. Suppose that
    \begin{equation*}
        \partial_t((1+\frac{Nt}{1+s})^{1+s}f(t)) \leq A(t) + B(t) (1+\frac{Nt}{1+s})^{1+s}f(t)
    \end{equation*}
    for any $N>0$, $t>0$, and $s \geq 0$. Then, \begin{equation*}
        f(t) \leq \frac{f(0) + \int_0^\infty A(t)\,\ud t}{(1+\frac{Nt}{1+s})^{1+s}}\exp\left(\int_0^{\infty}B(t) \,\ud t\right).
    \end{equation*}
\end{lemma}
\begin{proof}
    By Gr\"onwall's inequality, we get
    \begin{equation*}
        (1+\frac{Nt}{1+s})^{1+s}f(t) \leq \left( \int_0^{\infty} A(t)\,\ud t + f(0) \right) \exp\left(\int_0^{\infty}B(t) \,\ud t\right).
    \end{equation*}
This completes the proof.
\end{proof}

\subsection{Proof of Proposition~\ref{prop_energy_R}}
Here, we consider \eqref{eq:SIPM} with $\sigma = 0$. We introduce a commutator estimate which is crucial to prove \Black{Proposition~\ref{prop_energy_R}.}
\begin{lemma}\label{lem_commu}
    Let $s >2$. Then, there exists a constant $C=C(s)>0$ such that
    \begin{equation}\label{commutator}
    \left|\int_{\bbR^2} (\Lambda^{s-2}\partial_2^2(u_2 \partial_2 \tht) - u_2 \Lambda^{s-2}\partial_2^3 \tht) \Lambda^{s-2} \partial_2^2 \tht \,\ud x \right| \leq C\| |\xi| \mathscr{F} u_2 \|_{L^1} \| \tht \|_{H^s}^2 + C\| R_1 \tht \|_{H^s}^2 \| \tht \|_{H^s}.
    \end{equation}
    for any $\tht \in H^s(\bbR^2)$ and \Black{$u_2 = -\partial_1^2 (-\Delta)^{-1} \tht$}
\end{lemma}
\begin{proof}
By Plancherel theorem, we have \begin{gather*}
    \int_{\bbR^2} (\Lambda^{s-2}\partial_2^2(u_2 \partial_2 \tht) - u_2 \Lambda^{s-2}\partial_2^3 \tht) \Lambda^{s-2} \partial_2^2 \tht \,\ud x \\
    = \int_{\bbR^2} \left( \int_{\bbR^2} (|\xi|^{s-2}(i\xi_2)^2-|\xi-\eta|^{s-2}(i(\xi_2-\eta_2))^2) \mathscr{F}u_2(\eta)\mathscr{F}\partial_2\tht(\xi-\eta)\,\ud\eta \right) \overline{|\xi|^{s-2}(i\xi_2)^2 \mathscr{F}\tht(\xi)} \,\ud \xi.
\end{gather*}
We split the integral into the two regions $\Omega_{1}:=\{\eta\in\bbR^2:|\eta|<|\xi|/4\}$ and $\Omega_{2}:=\{\eta\in\bbR^2:|\eta|\geq|\xi|/4\}$. In the first region $\Omega_1$, we claim that
\begin{equation*}
    \big||\xi|^{s-2}(i\xi_2)^2-|\xi-\eta|^{s-2}(i(\xi_2-\eta_2))^2\big|\leq C|\xi-\eta|^{s-1}|\eta|
\end{equation*}
for some constant $C>0$ independent of the choice of $\xi \in \bbR^2$. For this purpose, we define a function $h:[0,1]\to\bbR$ by $h(t)=-|\xi-t\eta|^{s-2}|\xi_2-t\eta_2|^2$. Since $|\eta|<|\xi|/4$ in the region $\Omega_1$, $h$ is smooth on $[0,1]$ for each given $\xi \in \bbR^2$. Observing that
  \begin{equation*}
      h'(t)=(s-2)(\xi-t\eta)\cdot\eta |\xi-t\eta|^{-(4-s)}|\xi_2-t\eta_2|^2 + 2\eta_2|\xi-t\eta|^{s-2}(\xi_2-t\eta_2),
  \end{equation*}
we apply the mean value theorem to $h$ as
\begin{equation*}
    \big||\xi|^{s-2}(i\xi_2)^2-|\xi-\eta|^{s-2}(i(\xi_2-\eta_2))^2\big|\leq C|\xi-t\eta|^{s-1}|\eta|, \qquad t \in (0,1).
\end{equation*} However, it is clear that $$\big| |\xi-t\eta| - |\xi-\eta| \big| \leq (1-t)|\eta| \leq \frac{1-t}{3} |\xi-\eta| \leq \frac{1}{3}|\xi-\eta|.$$ Thus, we deduce the claim. Using it, we have
\begin{gather*}
    \int_{\bbR^2} \left( \int_{\Omega_1} (|\xi|^{s-2}(i\xi_2)^2-|\xi-\eta|^{s-2}(i(\xi_2-\eta_2))^2) \mathscr{F}u_2(\eta)\mathscr{F}\partial_2\tht(\xi-\eta)\,\ud\eta \right) \overline{|\xi|^{s-2}(i\xi_2)^2 \mathscr{F}\tht(\xi)} \,\ud \xi \\
    \leq \left\| \int_{\bbR^2} |\xi-\eta|^{s-1}|\eta| |\mathscr{F}u_2(\eta)||\mathscr{F}\partial_2\tht(\xi-\eta)| \,\ud\eta \right \|_{L^2} \big\| |\xi|^{s-2}(i\xi_2)^2 \mathscr{F}\tht(\xi) \big\|_{L^2} \\
    \leq C\| |\xi| \mathscr{F} u_2 \|_{L^1} \| \tht \|_{H^s}^2.
\end{gather*}

In the second region $\Omega_2,$ we note $|\xi|\leq 4|\eta|$ and $|\xi-\eta|\leq 5|\eta|$, which implies
\begin{equation}\label{eta_est}
    \big||\xi|^{s-2}(i\xi_2)^2-|\xi-\eta|^{s-2}(i(\xi_2-\eta_2))^2 \big| \leq C|\eta|^s
\end{equation}On the other hand, there holds $$\mathscr{F}u_2(\eta) = (i(\xi_1-\eta_1)-i\xi_1) \left(\frac{i\eta_1}{|\eta|^2} \mathscr{F}u_2(\eta) - \frac{i\eta_2}{|\eta|^2} \mathscr{F}u_1(\eta) \right).$$ For brevity, let us define the quantity $F:\bbC\to\bbC$ by
$$F(\eta)=\frac{i\eta_1}{|\eta|^2} \mathscr{F}u_2(\eta) - \frac{i\eta_2}{|\eta|^2} \mathscr{F}u_1(\eta).$$
The above observations lead to
\begin{gather*}
    \int_{\bbR^2} \left( \int_{\Omega_2} (|\xi|^{s-2}(i\xi_2)^2-|\xi-\eta|^{s-2}(i(\xi_2-\eta_2))^2) \mathscr{F}u_2(\eta)\mathscr{F}\partial_2\tht(\xi-\eta)\,\ud\eta \right) \overline{|\xi|^{s-2}(i\xi_2)^2 \mathscr{F}\tht(\xi)} \,\ud \xi \\
    = \int_{\bbR^2} \left( \int_{\Omega_2} (|\xi|^{s-2}(i\xi_2)^2-|\xi-\eta|^{s-2}(i(\xi_2-\eta_2))^2) F(\eta) \mathscr{F} \partial_1\partial_2\tht(\xi-\eta)\,\ud\eta \right) \overline{|\xi|^{s-2}(i\xi_2)^2 \mathscr{F}\tht(\xi)} \,\ud \xi \\
    +\int_{\bbR^2} \left( \int_{\Omega_2} (|\xi|^{s-2}(i\xi_2)^2-|\xi-\eta|^{s-2}(i(\xi_2-\eta_2))^2) F(\eta)\mathscr{F}\partial_2\tht(\xi-\eta)\,\ud\eta \right) \overline{|\xi|^{s-2}(i\xi_2)^2 \mathscr{F}\partial_1\tht(\xi)} \,\ud \xi.
\end{gather*}
By \eqref{eta_est}, the first integral on the \Black{right-hand side is bounded by}
\begin{gather*}
    C\left| \int_{\bbR^2} \int_{\Omega_2} |\mathscr{F}\Lambda^{s-1} u(\eta)||\mathscr{F} \partial_1\partial_2\tht(\xi-\eta)| \,\ud\eta \bigg| \overline{|\xi|^{s-2}(i\xi_2)^2 \mathscr{F}\tht(\xi)} \bigg| \,\ud \xi \right| \leq C\| R_1 \tht \|_{H^s}^2 \| \tht \|_{H^s}.
\end{gather*}
For the second one, using \eqref{eta_est} and $$\overline{|\xi|^{s-2}(i\xi_2)^2 \mathscr{F}\partial_1\tht(\xi)} = -i((\xi_2-\eta_2) + \eta_2) \overline{|\xi|^{s-2}(i\xi_2) \mathscr{F}\partial_1\tht(\xi)},$$ we obtain the upper bound
\begin{gather*}
C\left| \int_{\bbR^2} \int_{\Omega_2} |\mathscr{F}\Lambda^{s} u(\eta)||\mathscr{F}\partial_2\tht(\xi-\eta)| \,\ud\eta \bigg| \overline{|\xi|^{s-2}(i\xi_2) \mathscr{F}\partial_1\tht(\xi)} \bigg| \,\ud \xi \right| \\
+C\left| \int_{\bbR^2} \int_{\Omega_2} |\mathscr{F}\Lambda^{s-1} u(\eta)||\mathscr{F}\partial_2^2\tht(\xi-\eta)| \,\ud\eta \bigg| \overline{|\xi|^{s-2}(i\xi_2) \mathscr{F}\partial_1\tht(\xi)} \bigg|\,\ud \xi \right| \\
\leq C\| R_1 \tht \|_{H^s}^2 \| \tht \|_{H^s}.
\end{gather*}
This completes the proof.
\end{proof}

\begin{proof}[Proof of Proposition~\ref{prop_energy_R}]
We are ready to prove Proposition~\ref{prop_energy_R}. 
From \eqref{eq:SIPM}, we have
\begin{equation*}
	\frac 12 \frac {\ud}{\ud t} \int|\Lambda^s \theta|^2 \,\ud x + \int |R_1 \Lambda^s \theta|^2 \,\ud x = - \int \Lambda^s(u \cdot \nabla) \theta \Lambda^s \theta \,\ud x = - \int \Lambda^{s-2}(u \cdot \nabla) \theta \Lambda^{s-2}(-\Delta)^2 \theta \,\ud x.
\end{equation*}
From the definition $\Lambda^2 = -\Delta$, the right-hand side can be rewritten as
\begin{equation*}
	\int \Lambda^{s-2}(u \cdot \nabla) \theta \Lambda^{s-2}(-\Delta)^2 \theta \,\ud x = \sum_{i,j=1,2} \int \Lambda^{s-2} \partial_i\partial_j(u \cdot \nabla) \theta \Lambda^{s-2} \partial_i\partial_j\theta \,\ud x .
\end{equation*}
We estimate the cases where $j=1$ first and the nontrivial case \Black{$(i,j,k)=(2,2,2)$} in a separate way.

For $j=1$, we estimate as
\begin{equation*}
    \begin{split}
        &\int\Lambda^{s-2}\partial_i\partial_1(u\cdot\nb)\theta \Lambda^{s-2}\partial_i\partial_1\theta \,\ud x\\
         &\hphantom{\qquad\qquad} =\int\Lambda^{s-2}\partial_i^2 (u\cdot\nb)\theta \Lambda^{s-2}\partial_1^2 \theta \,\ud x \\
         &\hphantom{\qquad\qquad} \leq \Big(\|\Lambda^{s-2}[(u\cdot\nb)\partial_i^2\theta]-(u\cdot\nb)\Lambda^{s-2}\partial_i^2\theta\|_{L^2}+ \| \Lambda^{s-2} \partial_i(\partial_i u \cdot \nabla) \tht \|_{L^2} \Big) \|R_1^2\theta\|_{H^s} \\
        &\hphantom{\qquad\qquad}\leq C\|R_1\theta\|_{H^s}^2\|\theta\|_{H^s},
    \end{split}
\end{equation*}
where we used the higher-order commutator estimate \cite[Theorem 1.2]{FMRR} for Riesz potential, $H^{s-2}$ algebra with $s-2>1$, and the simple fact $\|u\|_{H^s}=\|R_1\theta\|_{H^s}.$ It remains to estimate the case $(i,j,k)=(2,2,2),$ namely,
\begin{equation*}
    \int_{\bbR^2} \Lambda^{s-2}\partial_2^2(u \cdot \nabla)\tht \Lambda^{s-2} \partial_2^2 \tht \,\ud x.
\end{equation*} Using the cancellation property, we separate the integral into $$\int_{\bbR^2} (\Lambda^{s-2}\partial_2^2(u_1 \partial_1 \tht) - u_1 \Lambda^{s-2}\partial_1 \partial_2^2 \tht) \Lambda^{s-2} \partial_2^2 \tht \,\ud x + \int_{\bbR^2} (\Lambda^{s-2}\partial_2^2(u_2 \partial_2 \tht) - u_2 \Lambda^{s-2} \partial_2^3 \tht) \Lambda^{s-2} \partial_2^2 \tht \,\ud x.$$
The first integral can be bounded similarly by $C\|R_1\theta\|_{H^s}^2\|\theta\|_{H^s}$. Recalling \eqref{commutator} gives the upper bound of the second integral by $$C\| |\xi| \mathscr{F} u_2 \|_{L^1} \| \tht \|_{H^s}^2 + C\| R_1 \tht \|_{H^s}^2 \| \tht \|_{H^s}.$$ Thus, we deduce \eqref{energy_R}.
\end{proof}

\subsection{Additional properties of $X^m$ and $Y^m$}
Here we gather the lemmas that can be easily proved. Still, we refer to the recent preprint \cite{JK2}, where the proofs will be presented, for the sake of completeness. The following two lemmas are about to which function space the product $fg$ belongs, where $f$ and $g$ are in either $X^m$ or $Y^m$.
\begin{lemma}[\cite{JK2}]\label{lem_basis1}
	Let $q_e$ and $q_o$ be even and odd number respectively. Then, we have
	\begin{equation*}
		\begin{aligned}
			-\sin(\frac {\pi}{2}q_e x_d) \sin(\frac {\pi}{2} q_o x_d) &= \frac 12 \left( \cos(\frac {\pi}{2}(q_e + q_o)x_d) - \cos(\frac {\pi}{2}(q_e - q_o)x_d) \right), \\
			\sin(\frac {\pi}{2}q_e x_d) \cos(\frac {\pi}{2}q'_e x_d) &= \frac 12 \left( \sin(\frac {\pi}{2}(q_e+q_e')x_d) + \sin(\frac {\pi}{2}(q_e-q_e')x_d) \right), \\
			-\cos(\frac {\pi}{2} q_o x_d) \sin(\frac {\pi}{2} q_o' x_d) &= -\frac 12 \left( \sin(\frac {\pi}{2}(q'_o+q_o)x_d) + \sin(\frac {\pi}{2}(q_o'-q_o)x_d) \right), \\
			\cos(\frac {\pi}{2}q_o x_d) \cos(\frac {\pi}{2} q_e x_d) &= \frac 12 \left( \cos(\frac {\pi}{2}(q_o+q_e)x_d) + \cos(\frac {\pi}{2}(q_o - q_e)x_d) \right).
		\end{aligned}
	\end{equation*}
\end{lemma}
\begin{lemma}[\cite{JK2}]\label{lem_basis2}
	Let $q-q'$ and $q-q''$ be odd and even number respectively. Then, we have
	\begin{equation*}
		\begin{aligned}
			-\sin(\frac {\pi}{2}q x_d) \sin(\frac {\pi}{2} q'' x_d) &= \frac 12 \left( \cos(\frac {\pi}{2}(q + q'')x_d) - \cos(\frac {\pi}{2}(q - q'')x_d) \right), \\
			\sin(\frac {\pi}{2}q x_d) \cos(\frac {\pi}{2}q' x_d) &= \frac 12 \left( \sin(\frac {\pi}{2}(q+q')x_d) + \sin(\frac {\pi}{2}(q-q')x_d) \right), \\
			-\cos(\frac {\pi}{2} q x_d) \sin(\frac {\pi}{2} q' x_d) &=- \frac 12 \left( \sin(\frac {\pi}{2}(q'+q)x_d) + \sin(\frac {\pi}{2}(q'-q)x_d) \right), \\
			\cos(\frac {\pi}{2}q x_d) \cos(\frac {\pi}{2} q'' x_d) &= \frac 12 \left( \cos(\frac {\pi}{2}(q+q'')x_d) + \cos(\frac {\pi}{2}(q - q'')x_d) \right).
		\end{aligned}
	\end{equation*}
\end{lemma}
The last lemma is the $\Omega$-version of the special case of Young's convolution inequality. 
\begin{lemma}[\cite{JK2}]\label{lem_conv_O}
    Let $f,f' \in X^m$ and $g,g' \in Y^m$. Then, we have
\begin{equation*}
    \begin{aligned}
    	\sum_{(p,q) \in \bbZ \times \bbN} |\mathscr{F}_b(fg)| &\leq \left( \sum_{(p,q) \in \bbZ \times \bbN} |\mathscr{F}_b f| \right) \left( \sum_{(p,q) \in \bbZ \times \bbN \cup \{ 0 \}} |\mathscr{F}_c g| \right), \\
	   	\sum_{(p,q) \in \bbZ \times (\bbN \cup \{ 0 \})} |\mathscr{F}_c(ff')| &\leq \left( \sum_{(p,q) \in \bbZ \times \bbN} |\mathscr{F}_b f| \right) \left( \sum_{(p,q) \in \bbZ \times \bbN} |\mathscr{F}_b f'| \right), \\
	    \sum_{(p,q) \in \bbZ \times (\bbN \cup \{0\})} |\mathscr{F}_c(gg')| &\leq \left( \sum_{(p,q) \in \bbZ \times (\bbN \cup \{ 0 \})} |\mathscr{F}_c g| \right) \left( \sum_{(p,q) \in \bbZ \times (\bbN \cup \{ 0 \})} |\mathscr{F}_c g'| \right).
	\end{aligned}
\end{equation*}
\end{lemma}

\end{document}